\documentclass[11pt,reqno]{amsart}
  \usepackage[T1]{fontenc}
  \usepackage[utf8]{inputenc}
\usepackage{mathtools,amssymb}
  \usepackage{lmodern}
  \usepackage{xcolor}
\usepackage[a4paper,left=27mm,right=27mm,top=25mm,bottom=27mm]{geometry}
  \usepackage{microtype}
\usepackage[numbers,sort&compress]{natbib}
  \usepackage[hidelinks]{hyperref}
  \numberwithin{equation}{section}
  \newtheorem{theorem}{Theorem}[section]
  \newtheorem{lemma}[theorem]{Lemma}
  \newtheorem{proposition}[theorem]{Proposition}
  \newtheorem{corollary}[theorem]{Corollary}
  \theoremstyle{remark}
  
  \allowdisplaybreaks
  \title[Sharp gradient comparison and rigidity]{Sharp gradient comparison and rigidity for semilinear equations on manifolds}
  \author{Linlin Sun}
  \address{School of Mathematics and Computational Science, Xiangtan University, Xiangtan 411105, China}
  \email{sunll@xtu.edu.cn}

\begin{document}
  \begin{abstract}
  We prove sharp amplitude and gradient comparisons for positive solutions of a class of semilinear elliptic equations on complete manifolds under a Ricci lower bound. Equality at one point characterizes hyperbolic space and the corresponding radial solution. The local estimates are uniform in the mixing coefficients and number of terms for positive mixtures of subcritical powers. For a class of conformal compactifications defined by a source equation, we obtain sharp boundary scalar-curvature bounds and hyperbolic rigidity. We also construct non-Einstein examples whose normalization persists under small compactly supported metric perturbations.
  \end{abstract}
  \maketitle
  \enlargethispage{3pt}
  
  \section{Introduction}
Let $(M^n,g)$ be a complete and connected Riemannian manifold without boundary, $n\ge2$, with
\begin{equation}\label{eq:ric-lower}
\operatorname{Ric}\ge-(n-1)\kappa g,\qquad \kappa\ge0.
\end{equation}
We study positive solutions of
\begin{equation}\label{eq:main}
-\Delta u=f(u).
\end{equation}
For $\kappa>0$, a radial hyperbolic solution gives sharp bounds for both
$u$ and $|\nabla u|$. Equality at one point determines the complete
manifold and the solution. A second application concerns conformal
compactifications, where the comparison controls the scalar curvature
of a boundary metric selected by a prescribed source equation.

\subsection{Comparison with a hyperbolic profile}
Let $\mathbb H^n_{-\kappa}$ be the complete simply connected space of
sectional curvature $-\kappa$, and set 
\begin{align*}
    p_S(n)=\begin{cases}
        \frac{n+2}{n-2},&n\geq3,\\
        \infty,&n=2.
    \end{cases}
\end{align*}
\begin{theorem}\label{thm:intro-gradient}
  Let $(M^n,g)$ be a complete connected Riemannian manifold without
  boundary, $n\ge2$, satisfying
  $\operatorname{Ric}_g\ge-(n-1)\kappa g$ for some $\kappa>0$.
  Let
  \[
  f(s)=\sum_{i=1}^m a_i s^{p_i},\qquad
  a_i>0,\quad \sum_{i=1}^m a_i=1,\quad 1<p_i<p_S(n).
  \]
  Let $u$ and $U_H$ be positive $C^3$ solutions of
  \[
  \begin{aligned}
  -\Delta_g u&=f(u)&&\text{on }M,\\
  -\Delta_{\mathbb H^n_{-\kappa}}U_H&=f(U_H)
  &&\text{on }\mathbb H^n_{-\kappa}.
  \end{aligned}
  \]
  Assume that $U_H=U_H(r)$ is radial about a pole and has the largest
  central value among all positive entire radial solutions of the
  hyperbolic equation. Write $A=U_H(0)$ and let
  $U_H^{-1}:(0,A]\to[0,\infty)$ be the inverse of its strictly
  decreasing radial profile, which tends to zero at infinity.
  Then $u\le A$ on $M$ and
  \begin{equation}\label{eq:intro-gradient}
  |\nabla u|\le -U_H'(U_H^{-1}(u))
  \quad\text{on }\{u<A\}.
  \end{equation}
  If $u=A$ at one point, or equality in \eqref{eq:intro-gradient}
  holds at one point of $\{u<A\}$, then $(M,g)$ is isometric to
  $\mathbb H^n_{-\kappa}$ and $u(x)=U_H(d_g(o,x))$ for some $o\in M$.
  \end{theorem}

The radius $U_H^{-1}(u(x))$ is determined by matching solution values;
in the rigidity case it equals $d_g(o,x)$. No boundedness or finite
energy is assumed for $u$. The model attains the gradient bound away
from its pole and the amplitude bound at the pole.

At $f(s)=s^{1+2/(n-1)}$, the known hyperbolic profile gives the explicit
form
\[
\begin{gathered}
U_H(r)=[n(n-1)\kappa]^{(n-1)/2}
\operatorname{sech}^{n-1}(\sqrt\kappa r),\\
|\nabla\log u|^2+\frac{n-1}{n}u^{2/(n-1)}
\le(n-1)^2\kappa.
\end{gathered}
\]
The radial solution at curvature $-1$ appears in
\cite{Mancini:2008aa}.
Proposition~\ref{prop:explicit-power} identifies it with the threshold
and obtains this sharp estimate and its equality case on general
complete manifolds satisfying \eqref{eq:ric-lower}.

Modica \cite{Modica:1985aa} proves
$|\nabla v|^2\le2W(v)$ for bounded $C^3$ entire Euclidean solutions of
$\Delta v=W'(v)$, with $W\in C^2(\mathbb R)$ nonnegative on $\mathbb R$.
Andrews and Xiong \cite{Andrews:2019aa} establish two-point estimates
for the inverse of a comparison profile and the resulting gradient
bounds at the same solution height. Their results treat quasilinear
equations on compact manifolds and bounded solutions on complete
manifolds of bounded geometry, under the corresponding barrier
hypotheses. Their compact theory also obtains comparison by
continuation from a family of barriers with uniformly large slopes.
Earlier complete- or compact-manifold gradient bounds under different
curvature and coefficient assumptions are due, for example, to
Ratto--Rigoli and Farina--Valdinoci \cite{RattoRigoli:1995aa,FarinaValdinoci:2011aa};
those results provide context rather than the present source-selected
hyperbolic comparison. Espinar, Gonz\'alez-Ib\'a\~nez and Mar\'in
\cite{Espinar:2026aa} compare bounded-domain
Dirichlet solutions with radial models of the same maximum height,
selected by boundary gradient data under shooting-admissibility and
derivative conditions on the source.

For the sources considered here, Theorem~\ref{thm:intro-gradient}
proves the universal sharp amplitude bound $u\le U_H(0)$ and gradient
comparison with the same maximal-center radial profile on hyperbolic
space. This model depends only on $n$, $\kappa$ and $f$ and applies
to every positive entire solution on a complete manifold under the
stated Ricci lower bound, without a bounded-geometry assumption.
Equality at one point, including attainment of the central height,
determines both the hyperbolic geometry and the radial solution.
Theorem~\ref{thm:gradient-general} treats the larger admissible source
class of Section~\ref{sec:local}, which also contains the sources
arising from conformal compactification below.

\subsection{Local estimates for mixed sources}
The local estimate controls the logarithmic gradient and the source
amplitude using only the radius, the Ricci lower bound and a fixed
subcritical exponent interval. Its constant is uniform over the
coefficients and the number of powers.
On Euclidean domains, Souplet \cite{Souplet:2023aa} obtains universal
amplitude estimates for fixed sources with subcritical regular variation
at zero and infinity, under the corresponding Liouville assumptions.
These hypotheses include each fixed positive finite or continuous mixture
of powers in a compact subcritical interval.
For a single power under a Ricci lower bound, Lu
\cite{Lu:2023aa} proves the logarithmic-gradient and
amplitude estimate throughout the subcritical range.
His general-source results
\cite{Lu:2026ab} cover
positive mixtures whose largest exponent is below $(n+3)/(n-1)$,
and part of the higher subcritical range under additional lower-growth
and inverse-growth conditions. Quasilinear estimates involving one or
two powers are studied in \cite{He:2024aa,He:2025aa}; general semilinear
sources with restricted upper growth are treated in
\cite{ChenHuangRuan:2026}. Classical Euclidean a priori and Liouville
results, and related single-equation manifold estimates, include
Gidas--Spruck, Li, and Li--Zhang \cite{Gidas:1981aa,Li:1991aa,LiZhang:2003aa};
they are cited as background only and do not supply the present
Ricci-lower-bound mixture estimate. Ma--Huang--Luo give another
single-equation comparison under their own sign and curvature
assumptions \cite{MaHuangLuo:2018aa}.

\begin{theorem}\label{thm:intro-mixtures}
  Let $(M^n,g)$ be complete, $n\ge2$, with $\operatorname{Ric}\ge-(n-1)\kappa g$ on $B_{2R}(x_0)$, $\kappa\ge0$. Let $f$ be any finite convex combination of powers $s^p$ with $p\in[p_-,p_+]\Subset(1,p_S(n))$. Every positive $C^3$ solution of $-\Delta u=f(u)$ on $B_{2R}(x_0)$ satisfies
  \[
  \sup_{B_R(x_0)}\left(|\nabla\log u|^2+\frac{f(u)}u\right)
  \le C(n,p_-,p_+)\left(R^{-2}+\kappa\right).
  \]
  \end{theorem}

The two exponent endpoints can be chosen independently throughout
$(1,p_S(n))$. For example, in dimension four the estimate applies to
$f(s)=(s^{11/10}+s^{29/10})/2$, whose exponents lie near opposite ends
of $(1,3)$ and which satisfies neither of Lu's two general-source
theorems cited above. The uniformity also includes continuous mixtures:
Proposition~\ref{prop:local-mixtures} gives the same estimate for
\[
f(s)=\int_{[p_-,p_+]}s^p\,d\nu(p),\qquad
0<\nu([p_-,p_+])<\infty,
\]
with a constant independent of the measure and its total mass.

The upper subcritical restriction is essential:
rescaled Euclidean critical bubbles preclude a solution-independent
amplitude bound at $p=p_S(n)$ when $n\ge3$.
For entire solutions, exhaustion gives
$|\nabla\log u|^2+f(u)/u\le C\kappa$.
The resulting nonexistence for $\kappa=0$ is already covered, for these
finite or continuous mixtures, by He, Sun and Wang
\cite{He:2025ab}; see also
\cite{Catino:2026aa} for related Liouville results.
For $\kappa>0$, this yields finite amplitude and gradient bounds and
supplies the initial control needed for the sharp comparison.

\subsection{Boundary curvature of a compactification}
Let $n\ge3$ and let $\overline M^n$ be compact with boundary and
connected interior $M$. A smooth defining function $\rho$ is positive
in $M$, vanishes on $\partial M$, and has $d\rho\ne0$ there.
A smooth metric $h$ on $\overline M$ gives the complete metric
$g=\rho^{-2}h$ and the boundary metric $\gamma=h|_{T\partial M}$.
For conformally compact Einstein metrics, Graham--Lee construct
fillings with prescribed conformal infinity, while Lee derives a
spectral boundary identity and Qing proves round-sphere rigidity under
Einstein and positive-mass hypotheses \cite{Graham:1991aa,Lee:1995aa,Qing:2003aa}.
Under stronger round-infinity, pole, decay, or curvature assumptions,
related rigidity statements are due to Bonini--Miao--Qing,
Dutta--Javaheri, Shi--Tian, and Li--Qing--Shi
\cite{Bonini:2006aa,DuttaJavaheri:2010aa,ShiTian:2005aa,LiQingShi:2017aa}.
These results concern different asymptotic or Einstein settings.
We select a boundary representative by imposing
$-\Delta_h\rho=q(\rho)$, with $q$ smooth, nonnegative, nondecreasing
and concave on $[0,\infty)$, and positive on $(0,\infty)$.

For Poincar\'e--Einstein fillings, Gursky and Han
\cite{Gursky:2017aa} relate the type I Yamabe
constant of a compactification to its boundary Yamabe constant through
an isoperimetric ratio. Chen, Lai and Wang
\cite{Chen:2019aa} characterize equality in this inequality and prove
the type II and mixed Yamabe inequalities, with hyperbolic space as
the equality case in all three. These results use minimizing Yamabe
compactifications with constant scalar curvature and constant boundary
mean curvature, whose existence is ensured by the dimension or
geometric hypotheses in those works.
Wang and Wang \cite{Wang:2022aa} obtain the type II
inequality without assuming attainment of the critical Yamabe
infimum. Their later result \cite{Wang:2025ab}
allows a Ricci lower bound and nonnegative Yamabe conformal infinity
under an asymptotically Poincar\'e--Einstein condition.
Our comparison applies to the boundary representative selected by the
prescribed source equation, under a Ricci lower bound. It gives a
pointwise upper bound for its scalar curvature, with rigidity from
equality at a single boundary point.

With $k=n-1$, the change $u=\rho^k$ gives
\[
-\Delta_g u=f_q(u),\qquad
f_q(s)=k s^{1+1/k}q(s^{1/k}).
\]
This source belongs to the admissible class
(Lemma~\ref{lem:q-admissible}). Its threshold solution $U_H$ determines
a smooth radial compactification
$h_H=\rho_H^2g_{\mathbb H^n_{-\kappa}}$, with $\rho_H=U_H^{1/k}$,
of the closed hyperbolic ball (Lemma~\ref{lem:model-compactification}).
Denote by $R_{q,\kappa}$ the scalar curvature of its round boundary metric.

\begin{theorem}\label{thm:intro-boundary}
  Let $(\overline M^n,h)$ be a compact smooth Riemannian manifold with
  nonempty boundary and connected interior $M$, $n\ge3$.
  Let $\kappa>0$ and let $q\in C^\infty([0,\infty))$ be nonnegative,
  nondecreasing and concave, with $q(s)>0$ for $s>0$.
  Suppose that a smooth defining function $\rho$ satisfies
  \begin{equation}\label{eq:intro-source}
  -\Delta_h\rho=q(\rho),\qquad g=\rho^{-2}h,
  \end{equation}
  and
  \begin{equation}\label{eq:asymptotic-small}
  \operatorname{Ric}_g\ge-(n-1)\kappa g,\qquad
  R_g+n(n-1)\kappa=o(\rho^2)\quad\text{as }\rho\to0.
  \end{equation}
  Then $\gamma=h|_{T\partial M}$ satisfies
  \[
  R_\gamma\le R_{q,\kappa}\quad\text{on }\partial M.
  \]
  Equality at one boundary point holds if and only if there is an
  isometry
  \[
  \Phi:(\overline M,h)\longrightarrow(\overline{\mathbb H^n},h_H)
  \quad\text{with}\quad \rho_H\circ\Phi=\rho.
  \]
  \end{theorem}

The model constant also satisfies
\[
R_{q,\kappa}\le(n-1)(n-2)
\left(\frac{q'(0)}n+\frac{q(0)^2}{n^2\kappa}\right).
\]
Equality here forces $q$ to be affine on $[0,\max\rho_H]$
(Lemma~\ref{lem:auxiliary}). The sharp constant depends on the source
throughout the range of the model, whereas the explicit upper bound
uses only $q(0)$ and $q'(0)$.
For $q(s)=a+bs$, with $a,b\ge0$ and $a+b>0$, the bound is attained.
The model compactifications are the standard spherical caps for $b>0$,
including the hemisphere for $a=0$, and the Euclidean ball for $b=0$,
which occur in the equality cases of \cite{Chen:2019aa}.
Corollary~\ref{cor:affine-model} gives their formulas in the source
parameters.
For $b>0$, the sources $q(s)=bs$ and $q(s)=b\tanh s$ have the same
boundary value and derivative, but
\[
R_{b\tanh,\kappa}<\frac{(n-1)(n-2)b}{n}=R_{bs,\kappa}.
\]
The source associated with $b\tanh s$ is not a positive finite sum of
powers, so this example uses the larger admissible source class.

The decay hypothesis cannot simply be weakened to an order-two scalar
error. Under $R_g+n(n-1)\kappa=O(\rho^2)$,
Theorem~\ref{thm:boundary-general} instead gives
\[
R_\gamma-S|_{\partial M}\le R_{q,\kappa},\qquad
S=\rho^{-2}\bigl(R_g+n(n-1)\kappa\bigr).
\]
The cylindrical example in Section~\ref{sec:boundary} shows that this
correction is necessary.

\subsection{Existence on nonmodel fillings}
For a prescribed complete metric $g$, the normalization requires a
positive solution of $-\Delta_g u=f_q(u)$ for which
$\rho=u^{1/(n-1)}$ is a smooth defining function.
Section~\ref{sec:existence} constructs such solutions for every source
in Theorem~\ref{thm:intro-boundary}, including $q(0)>0$, on two-ended
non-Einstein metrics. It also gives examples on AdS--Schwarzschild
metrics with connected boundary of positive scalar curvature when
$n\ge4$ (Theorem~\ref{thm:schwarzschild}).
Volume-preserving tangential deformations retain the same defining
function for this full source class
(Proposition~\ref{prop:volume-deformations}).

For a Riemannian manifold $(N^k,h_N)$ and an integer $\ell\ge1$,
put $\varepsilon_c=(\ell+1)^{-1}$. For $0<\varepsilon<\varepsilon_c$,
let $g_\varepsilon$ on $M=\mathbb R\times N$ be the smooth double of
\[
\frac{dr^2}{r^2-1+\varepsilon(1-\varepsilon)^\ell r^{-2\ell}}
+r^2h_N,\qquad r>\sqrt{1-\varepsilon},
\]
joined at $r=\sqrt{1-\varepsilon}$.
Its two expanding ends are separated by a neck whose length tends to
infinity as $\varepsilon\uparrow\varepsilon_c$.
The constant $R_{q,1}$ is the model boundary scalar curvature defined
in the preceding subsection, with $\kappa=1$.

\begin{theorem}\label{thm:intro-nonsymmetric}
  Let $n\ge3$, $k=n-1$, and let $\ell\ge k$ be an integer.
  Let $(N^k,h_N)$ be connected, closed and Einstein, with
  $\operatorname{Ric}_{h_N}=-(k-1)h_N$.
  Let $q\in C^\infty([0,\infty))$ be nonnegative, nondecreasing and
  concave, positive on $(0,\infty)$, and admit a smooth odd extension
  across zero.
  There exists $\varepsilon_0\in(0,\varepsilon_c)$ such that the
  following holds for every $\varepsilon\in(\varepsilon_0,\varepsilon_c)$.
  For each compact $K\subset M$, there is $\delta>0$ such that every
  smooth Riemannian metric $g$ on $M$ satisfying
  \[
  g=g_\varepsilon\quad\text{on }M\setminus K,\qquad
  \|g-g_\varepsilon\|_{C^2(K,g_\varepsilon)}<\delta
  \]
  admits a compactification $\overline M\cong[-1,1]\times N$ and
  a smooth defining function $\rho$ such that $h=\rho^2g$ extends to
  a smooth Riemannian metric on $\overline M$ and
  \[
  -\Delta_h\rho=q(\rho)\quad\text{in }M.
  \]
  Moreover,
  \[
  \operatorname{Ric}_g\ge-kg,\qquad
  R_g+n(n-1)=O(\rho^{2\ell+2})=o(\rho^2)
  \quad\text{as }\rho\to0,
  \]
  and $\gamma=h|_{T\partial M}$ satisfies
  \[
  R_\gamma<R_{q,1}\quad\text{at every point of }\partial M.
  \]
  \end{theorem}

The defining function can be chosen so that $u=\rho^k\in H^1(M,g)$
is a ground state of $-\Delta_g u=f_q(u)$:
$\mathcal E_g(u)=\inf_{\mathcal N_g}\mathcal E_g$, where
$F_q(s)=\int_0^s f_q(\tau)\,d\tau$ and
\[
\mathcal E_g(v)=\frac12\int_M|dv|_g^2\,dV_g
-\int_M F_q(|v|)\,dV_g.
\]
Here $\mathcal N_g$ is the set of nonzero $v\in H^1(M,g)$ satisfying
\[
\int_M|dv|_g^2\,dV_g=\int_M|v|f_q(|v|)\,dV_g.
\]
This minimization imposes no symmetry on $u$.
The odd-extension assumption is used for smooth boundary regularity;
finite-energy existence and two-sided decay hold for the full source
class. Both $bs$ and $b\tanh s$, $b>0$, satisfy the odd-extension
condition. The boundary metrics are positive smooth conformal multiples
of $h_N$, with factors that need not be constant or equal on the two ends.

Each boundary component is therefore of negative Yamabe type.
Two-ended Einstein metrics with this boundary sign already appear in
Witten--Yau \cite{Witten:1999aa} and are recalled by Cai--Galloway
\cite{Cai:1999aa}. Their connectedness results require a positive or
nonnegative scalar-curvature representative on at least one boundary
component, respectively, together with the relevant curvature and
asymptotic hypotheses.

\section{Universal local estimates}\label{sec:local}
  We formulate the local estimate in terms of the primitive of the source
  and its tail integral. For background on universal estimates and
  Liouville problems, see \cite{Serrin:2002aa,Polacik:2007aa}.

  We call $f$ \emph{admissible} in dimension $n\ge2$ if
  \begin{equation}\label{eq:f-basic}
  f\in C^1([0,\infty))\cap C^2((0,\infty)),\qquad
  f(0)=f'(0)=0,
  \end{equation}
  and it has a finite decomposition $f=\sum_{i=1}^m f_i$ with the
  following properties. Each $f_i$ is positive and belongs to
  $C^2((0,\infty))\cap L^1_{\mathrm{loc}}([0,\infty))$.
   Set
  \[
  F_i(s)=\int_0^s f_i(t)\,dt,\quad
  w_i(s)=\int_s^\infty F_i(t)^{-1/2}\,dt,\quad
  F=\sum_i F_i,\quad w(s)=\int_s^\infty F(t)^{-1/2}\,dt.
  \]
  The component tails are finite for $s>0$, and, for some $C_2\ge1$,
  \begin{equation}\label{eq:component-order}
  \frac{f_1} {F_1}\le\cdots\le\frac{f_m}{F_m},\qquad
  \sum_{i=1}^m(F_iw_i)^2\le C_2(Fw)^2\quad\text{on }(0,\infty).
  \end{equation}
  There is an exponent
  \begin{equation}\label{eq:mu}
  \mu>\frac{n-2}{2(n-1)},
  \end{equation}
  such that $(F_iw_i)^\mu$ is concave on $(0,\infty)$ for every $i$.

  For $f(s)=a s^p$, $a>0$, $p>1$,
  \[
  F(s)=\frac a{p+1}s^{p+1},\qquad
  w(s)=\frac{2\sqrt{(p+1)/a}}{p-1}s^{-(p-1)/2}.
  \]
  Hence $(Fw)^\mu$ is proportional to $s^{\mu(p+3)/2}$, and
  $\mu=2/(p+3)$ verifies concavity. The lower bound in
  \eqref{eq:mu} holds precisely for $1<p<p_S(n)$.
  
  \begin{theorem}\label{thm:local-tail}
  Let $(M^n,g)$ be complete and $\Omega\subset M$ open. Suppose $\operatorname{Ric}\ge-(n-1)\kappa g$ on $\Omega$, with $\kappa\ge0$, and let $f$ be admissible. Every positive $C^3$ solution of $-\Delta u=f(u)$ in $\Omega$ satisfies
  \[
  |\nabla\log w(u)|(x)+w(u(x))^{-1}
  \le C(n,\mu,m,C_2)\bigl(d_\Omega(x)^{-1}+\sqrt\kappa\bigr).
  \]
  \end{theorem}
  
  \subsection{The tail estimate}
  Let $f>0$ be $C^1$ on $(0,\infty)$, and let $F>0$ satisfy $F'=f$. For a positive solution put
  \[
  E=\frac{1}{2}|\nabla u|^2+F(u),\qquad t_0=\frac{n-2}{2(n-1)}.
  \]

  \begin{lemma}\label{lem:energy}
  For every $t>0$ with $t\ge t_0$,
  \begin{equation}\label{eq:energy-refined}
  \Delta E^t\ge
  -\frac{t f(u)^2 E^t}{\frac n{n-1}F(u)+(t-t_0)|\nabla u|^2}
  +tE^{t-1}\operatorname{Ric}(\nabla u,\nabla u).
  \end{equation}
  In particular,
  \begin{equation}\label{eq:energy-simple}
  \Delta E^t\ge-\frac{t(n-1)}n\frac{f(u)^2}{F(u)}E^t
  +tE^{t-1}\operatorname{Ric}(\nabla u,\nabla u).
  \end{equation}
  \end{lemma}
  \begin{proof}
  At a regular point choose $e_1=\nabla u/|\nabla u|$ and let
  $e_2,\ldots,e_n$ be tangent to its level hypersurface. Write
  $u_{ij}=\nabla^2u(e_i,e_j)$. Since
  \[
  \sum_{\alpha=2}^n u_{\alpha\alpha}=-f-u_{11},\qquad
  E_i=|\nabla u|(u_{1i}+f\delta_{1i}),
  \]
 we have
  \begin{align}
  &|\nabla u|^2\bigl(|\nabla^2u|^2-f^2\bigr)
  +2f\langle\nabla u,\nabla E\rangle
  -\frac n{n-1}|\nabla E|^2\notag\\
  &\qquad=|\nabla u|^2\left[
  \sum_{\alpha,\beta=2}^n
  \left(u_{\alpha\beta}+\frac{f+u_{11}}{n-1}\delta_{\alpha\beta}\right)^2
  +\frac{n-2}{n-1}\sum_{\alpha=2}^n u_{1\alpha}^2\right]\ge0.
  \label{eq:normal-tangent}
  \end{align}
  Bochner's identity and the chain
  rule give
  \begin{align*}
  \frac12\Delta|\nabla u|^2
  &=|\nabla^2u|^2-f'(u)|\nabla u|^2
  +\operatorname{Ric}(\nabla u,\nabla u),\\
  \Delta F(u)&=f'(u)|\nabla u|^2-f(u)^2.
  \end{align*}
  Adding gives
  \[
  \Delta E=|\nabla^2u|^2-f^2+\operatorname{Ric}(\nabla u,\nabla u).
  \]
  Consequently, at regular points,
  \begin{align*}
  \Delta E^t\ge tE^{t-1}\bigg[
  &\left(\frac{n}{(n-1)|\nabla u|^2}+\frac{t-1}{E}\right)|\nabla E|^2\\
  &-\frac{2f}{|\nabla u|^2}
  \langle\nabla u,\nabla E\rangle
  +\operatorname{Ric}(\nabla u,\nabla u)\bigg].
  \end{align*}
  The coefficient of $|\nabla E|^2$ is
  \[
  \frac{n}{(n-1)|\nabla u|^2}+\frac{t-1}{E}
  =\frac{\frac n{n-1}F+(t-t_0)|\nabla u|^2}
  {E|\nabla u|^2}.
  \]
  Since $t\ge t_0$, the numerator is positive. Completing the square
  in $\nabla E$ gives
  \begin{align*}
  &\left(\frac{n}{(n-1)|\nabla u|^2}+\frac{t-1}{E}\right)|\nabla E|^2
  -\frac{2f}{|\nabla u|^2}\langle\nabla u,\nabla E\rangle\\
  &\qquad\ge-
  \frac{f^2E}{\frac n{n-1}F+(t-t_0)|\nabla u|^2}.
  \end{align*}
  This proves \eqref{eq:energy-refined} at regular points. At a critical
  point, $E=F$ and $\nabla E=0$; the estimate
  $|\nabla^2u|^2\ge f^2/n$ proves the same inequality directly.
  Finally, the denominator in \eqref{eq:energy-refined} is at least
  $nF/(n-1)$, which gives \eqref{eq:energy-simple}.
  \end{proof}
  We balance the negative source term in \eqref{eq:energy-simple} by
  dividing $E^t$ by a function of $u$. Fix $0<t<1/2$ with $t\ge t_0$,
  and write $d=t(n-1)/n$. For a positive $C^2$ scalar function $L$ set
  $G=E^t/L(u)$. To track the first-order term, differentiate
  the identity $E^t=L(u)G$:
  \[
  \Delta E^t=L\Delta G+2L'\langle\nabla u,\nabla G\rangle
  +G\bigl(L''|\nabla u|^2-fL'\bigr).
  \]
  Dividing by $L$ and using \eqref{eq:energy-simple} gives
  \begin{equation}\label{eq:weighted-local}
  \Delta G+2\frac{L'}L\langle\nabla u,\nabla G\rangle
  \ge G\left[-\frac{L''}L|\nabla u|^2
  +f\left(\frac{L'}L-d\frac fF\right)
  +\frac tE\operatorname{Ric}(\nabla u,\nabla u)\right].
  \end{equation}
  \begin{proposition}\label{prop:localize}
  Fix $0<t<1/2$ with $t\ge t_0$, and set $d=t(n-1)/n$.
  On a complete manifold $(M^n,g)$, let $u\in C^3(B_{2R}(p))$ satisfy the hypotheses of Lemma~\ref{lem:energy} and let $\operatorname{Ric}\ge-(n-1)\kappa g$ on this ball. Set $\Lambda=R^{-1}+\sqrt\kappa$. Suppose $L>0$ is $C^2$ on an interval containing the range of $u$ and, for some $\varepsilon,\delta>0$ and $C_0\ge1$,
  \begin{align}
  -L^{1/t}\left(\frac{L''}L+\varepsilon\frac{(L')^2}{L^2}\right)&\ge\delta\Lambda,
  \label{eq:mult-second}\\
  L^{1/t}\frac fF\left(\frac{L'}L-d\frac fF\right)
  &\ge\delta\Lambda-C_0\Lambda^2\frac{L^{1/t}}F.
  \label{eq:mult-zero}
  \end{align}
  Then
  \[
  \sup_{B_R(p)}\frac{E}{L(u)^{1/t}}
  \le C(n,t,\varepsilon)\frac{C_0+1}{\delta}\Lambda.
  \]
  \end{proposition}
  \begin{proof}
Use the standard distance cutoff and Calabi support construction
\cite{Cheng:1975aa,Yau:1975aa}: choose $0\le\eta\le1$,
$\eta=1$ on $B_R(p)$, supported in $B_{3R/2}(p)$, such that
\[
|\nabla\eta|\le CR^{-1},\qquad
\Delta\eta\ge-C(n)(R^{-2}+R^{-1}\sqrt\kappa)
\]
in the support sense. The positive function $\eta G$ attains its
maximum on the compact support. All maximum-point calculations below
use Calabi supports if the maximum is on the cut locus.
At an interior positive maximum $x_0$ of $\eta G$, where the cutoff is smooth,
  \[
  \frac{\nabla G}G=-\frac{\nabla\eta}\eta,\qquad
  0\ge\frac{\Delta G}G+\frac{\Delta\eta}\eta
  -2\frac{|\nabla\eta|^2}{\eta^2}.
  \]
  Using \eqref{eq:weighted-local}, this implies
  \begin{align*}
  0\ge{}&-\frac{L''}L|\nabla u|^2
  +f\left(\frac{L'}L-d\frac fF\right)
  +\frac tE\operatorname{Ric}(\nabla u,\nabla u)\\
  &+\frac{\Delta\eta}\eta-2\frac{|\nabla\eta|^2}{\eta^2}
  +2\frac{L'}L\frac{\langle\nabla u,\nabla\eta\rangle}\eta.
  \end{align*}
  The last term is bounded below by
  \[
  -\varepsilon\left(\frac{L'}L\right)^2|\nabla u|^2
  -\varepsilon^{-1}\frac{|\nabla\eta|^2}{\eta^2},
  \]
  and, because $|\nabla u|^2\le2E$, the Ricci term is bounded
  below by $-2t(n-1)\kappa$. Conditions
  \eqref{eq:mult-second}--\eqref{eq:mult-zero} imply
  \begin{align*}
  -\left(\frac{L''}L+\varepsilon\frac{(L')^2}{L^2}\right)|\nabla u|^2
  &\ge\delta\Lambda\frac{|\nabla u|^2}{L^{1/t}},\\
  f\left(\frac{L'}L-d\frac fF\right)
  &\ge\delta\Lambda\frac{F}{L^{1/t}}-C_0\Lambda^2.
  \end{align*}
  Multiply the maximum-point inequality by $\eta^{1/t}$.
  The cutoff errors contain the factors $\eta^{1/t-1}$ and
  $\eta^{1/t-2}$, both at most one since $1/t>2$.
  They are therefore bounded below by
  $-C(n,t,\varepsilon)\Lambda^2$, even if the maximum lies where
  $\eta$ is small. We obtain
  \[
  \delta\Lambda\eta^{1/t}\frac{|\nabla u|^2+F}{L^{1/t}}
  \le C(n,t,\varepsilon)(C_0+1)\Lambda^2.
  \]
  Since $E\le|\nabla u|^2+F$, the left side controls
  $\delta\Lambda(\eta G)^{1/t}(x_0)$.
  Thus $(\eta G)^{1/t}(x_0)\le C(C_0+1)\Lambda/\delta$.
  On $B_R(p)$, where $\eta=1$, maximality proves the assertion.
  
  \end{proof}
  
  We construct the weight $L$ for
  \eqref{eq:mult-second}--\eqref{eq:mult-zero} from the admissible decomposition.

  \begin{lemma}\label{lem:theta}
  Let $f>0$ belong to $C^1((0,\infty))\cap L^1_{\mathrm{loc}}([0,\infty))$, let $F(s)=\int_0^s f$, and suppose $w(s)=\int_s^\infty F^{-1/2}$ is finite. If $(Fw)^\mu$ is concave for $\mu>0$, then $\mu<1/2$ and
  \begin{equation}\label{eq:theta-bound}
  0<\Theta(s):=\frac{\sqrt{F(s)}}{f(s)w(s)}
  \le\frac{1-2\mu}{2(1-\mu)}<\frac{1}{2}.
  \end{equation}
  \end{lemma}
  \begin{proof}
  Put $\Phi=(Fw)^\mu$. A positive concave function on
  $(0,\infty)$ is nondecreasing: a negative derivative at one point,
  together with the tangent-line upper bound, would force the function
  to become negative at a sufficiently large argument. Thus
  $\Phi'\ge0$, and concavity gives
  $\Phi(s+r)\le\Phi(s)+r\Phi'(s)$ for $r\ge0$. Also
  \[
  (\sqrt w)'=-\tfrac{1}{2}\Phi^{-1/(2\mu)}.
  \]
  Since $w(r)\to0$ as $r\to\infty$,
  \[
  \sqrt{w(s)}=\frac{1}{2}\int_s^\infty\Phi(r)^{-1/(2\mu)}\,dr
  \ge\frac{1}{2}\int_0^\infty[\Phi(s)+\Phi'(s)r]^{-1/(2\mu)}\,dr.
  \]
  The last integral would diverge if $\Phi'(s)=0$ or $\mu\ge1/2$. Its value is
  \[
  \frac{\mu}{1-2\mu}\frac{\Phi(s)^{1-1/(2\mu)}}{\Phi'(s)}.
  \]
  Rearranging the integral bound gives
  \[
  \sqrt w\,\Phi'\Phi^{1/(2\mu)-1}
  \ge\frac{\mu}{1-2\mu},\qquad
  \sqrt w\,(\Phi^{1/(2\mu)})'
  \ge\frac1{2(1-2\mu)}.
  \]
  Since $\Phi^{1/(2\mu)}=\sqrt{Fw}$ and
  $w'=-F^{-1/2}$, the last left side equals
  \[
  \frac{fw}{2\sqrt F}-\frac12
  =\frac12(\Theta^{-1}-1).
  \]
  Solving the resulting inequality for $\Theta>0$ proves
  \eqref{eq:theta-bound}.
  \end{proof}
  
  \begin{lemma}\label{lem:growth-tail}
  For an admissible source, put
  \[
  A_\mu=\frac{2(1-\mu)}\mu,\qquad b_\mu=\frac{\mu}{1-2\mu}.
  \]
  Then
  \begin{equation}\label{eq:primitive-growth}
  \frac{s f_i(s)}{F_i(s)}\le A_\mu,\qquad
  \frac{s f(s)}{F(s)}\le A_\mu,\qquad
  \sqrt{F(s)}w(s)\ge b_\mu s.
  \end{equation}
  The last estimate holds for every component as well.
  \end{lemma}
  \begin{proof}
  Write $\theta=(1-2\mu)/(2(1-\mu))$,
  $\Theta_i=\sqrt{F_i}/(f_iw_i)$ and $\Phi_i=(F_iw_i)^\mu$.
  Each $\Phi_i$ is positive and concave on $(0,\infty)$, hence
  nondecreasing. Its tangent-line inequality at $s$ gives
  \[
  0\le\Phi_i(0+)\le\Phi_i(s)-s\Phi_i'(s).
  \]
  Therefore
  \[
  \mu\frac{s f_i}{F_i}(1-\Theta_i)
  =\frac{s\Phi_i'}{\Phi_i}\le1,
  \qquad
  \frac{s f_i}{F_i}\le\frac1{\mu(1-\theta)}=A_\mu.
  \]
  Taking the weighted mean with weights $F_i/F$ gives the bound for $f$. Integrating it gives $F(r)/F(s)\le(r/s)^{A_\mu}$ for $r\ge s>0$. Since $A_\mu>2$,
  \[
  \sqrt{F(s)}w(s)
  \ge\int_s^\infty(r/s)^{-A_\mu/2}\,dr
  =\frac{2s}{A_\mu-2}=b_\mu s.
  \]
  The component proof is identical.
  \end{proof}
  
  For a positive $C^2$ function $H$ with $H^\mu$ concave, $0<t<\mu$, and $\varepsilon=(\mu-t)/(2t)$, direct differentiation gives
  \begin{equation}\label{eq:power-concavity}
  -\frac{(H^t)''}{H^t}-\varepsilon\left(\frac{(H^t)'}{H^t}\right)^2
  =-\frac t\mu\frac{(H^\mu)''}{H^\mu}
  +\frac{t(\mu-t)}2\left(\frac{H'}H\right)^2
  \ge\frac{t(\mu-t)}2\left(\frac{H'}H\right)^2.
  \end{equation}
  For every positive decomposition with finite tails,
  \begin{equation}\label{eq:tail-sum}
  (Fw)^2\le\sum_i(F_iw_i)^2.
  \end{equation}
  Indeed, fix $s$ and write $\nu_i=F_i(s)/F(s)$. Jensen's inequality and then Cauchy's inequality give
  \begin{align*}
  \sqrt{F(s)}w(s)
  &=\int_s^\infty\left(\sum_i\nu_i\frac{F_i(r)}{F_i(s)}\right)^{-1/2}\,dr\\
  &\le\sum_i\nu_i\sqrt{F_i(s)}w_i(s)
  \le\left(\frac{\sum_i(F_i(s)w_i(s))^2}{F(s)}\right)^{1/2}.
  \end{align*}
  
  Choose an exponent between $t_0$ and $\mu$ by setting
  \begin{equation}\label{eq:choose-t}
  t=\frac{t_0+\mu}{2},\qquad
  \varepsilon=\frac{\mu-t}{2t},\qquad
  \lambda=1+\frac{n}{(n-1)(\mu-t_0)}C_2^{t/2}.
  \end{equation}
  Then $0<t<\mu<1/2$ and $t>t_0$. In dimension two these choices give $t=\mu/2$ and $\varepsilon=1/2$.
  
  \begin{lemma}\label{lem:weight}
  For $\Lambda>0$, define
  \begin{equation}\label{eq:weight}
  L(s)=\sum_{i=1}^m\lambda^{i-1}(F_iw_i)^t(s)+\Lambda^t s^{2t}.
  \end{equation}
  There are $\delta>0$ and $C_0\ge1$, depending only on $n,\mu$, for which \eqref{eq:mult-second}--\eqref{eq:mult-zero} hold. Moreover
  \begin{equation}\label{eq:weight-size}
  Fw\le L^{1/t}\le C(n,\mu)^m C_2^{m/2}Fw+C(n,\mu)\Lambda Fw^2.
  \end{equation}
  \end{lemma}
  \begin{proof}
  Write
  \[
  S=\sum_i\lambda^{i-1}(F_iw_i)^t,\quad
  T=\Lambda^t s^{2t},\quad L=S+T,\quad \chi=T/L,
  \quad \theta=\frac{1-2\mu}{2(1-\mu)}.
  \]
  For positive summands $H_j$ with $H=\sum H_j$,
  \begin{align*}
  -\frac{H''}H-\varepsilon\left(\frac{H'}H\right)^2
  -\sum_j\frac{H_j}H\left[-\frac{H_j''}{H_j}
  -\varepsilon\left(\frac{H_j'}{H_j}\right)^2\right]
  =\frac\varepsilon2\sum_{i,j}\frac{H_iH_j}{H^2}
  \left(\frac{H_i'}{H_i}-\frac{H_j'}{H_j}\right)^2\ge0.
  \end{align*}
  Both $(F_iw_i)^\mu$ and $s^{2\mu}$ are concave. Apply \eqref{eq:power-concavity} to the summands of $L$ and retain the contribution from $T$. It gives
  \[
  -L^{1/t}\left(\frac{L''}L+\varepsilon\frac{(L')^2}{L^2}\right)
  \ge2t(\mu-t)L^{1/t}\frac{\chi}{s^2}
  =2t(\mu-t)\chi^{1-1/t}\Lambda
  \ge2t(\mu-t)\Lambda.
  \]
  This proves \eqref{eq:mult-second}. To establish
  \eqref{eq:mult-zero}, fix $s$ and put
  \[
  \nu_i=F_i/F,\qquad r_i=\lambda^{i-1}(F_iw_i)^t/S.
  \]
  Since $w_i\ge w$, we have $F_iw_i\ge\nu_i Fw$. For $j\ge2$,
  \[
  \sum_{i\ge j}\lambda^{i-1}(F_iw_i)^t
  \ge\lambda^{j-1}(Fw)^t\sum_{i\ge j}\nu_i,
  \qquad
  \sum_{i<j}\lambda^{i-1}(F_iw_i)^t
  \le\frac{\lambda^{j-1}}{\lambda-1}C_2^{t/2}(Fw)^t.
  \]
  It follows that
  \[
  \sum_{i\ge j}r_i\ge
  \frac{\sum_{i\ge j}\nu_i}{1+C_2^{t/2}/(\lambda-1)}.
  \]
  To see this implication explicitly, put
  $A_j=\sum_{i\ge j}\nu_i\le1$ and
  $x=C_2^{t/2}/(\lambda-1)$. The two preceding bounds give
  $\sum_{i\ge j}r_i\ge A_j/(A_j+x)\ge A_j/(1+x)$.
  For $j=1$ the inequality follows from $\sum_i r_i=\sum_i\nu_i=1$.
  Summation by parts with the ordered numbers $a_i=f_i/F_i$ gives
  \[
  \sum_i r_i a_i
  =a_1+\sum_{j=2}^m(a_j-a_{j-1})\sum_{i\ge j}r_i
  \ge\frac{f/F}{1+C_2^{t/2}/(\lambda-1)}.
  \]
  As $(F_iw_i)'/(F_iw_i)=(1-\Theta_i)f_i/F_i$,
  \[
  \frac{S'}{tS}\ge
  \frac{1-\theta}{1+C_2^{t/2}/(\lambda-1)}\frac fF.
  \]
  Put
  \[
  c=\frac t2\left(1-\theta-\frac{n-1}{n}\right)
  =\frac{t(n-1)(\mu-t_0)}{2n(1-\mu)}>0.
  \]
  Since $\lambda-1=t(1-\theta)C_2^{t/2}/c$, the inequality
  $(1+x)^{-1}\ge1-x$ gives
  \begin{equation}\label{eq:weight-positive}
  \frac{S'}S-d\frac fF\ge c\frac fF.
  \end{equation}
  The other summand satisfies $T'/T-df/F\ge-df/F$. Therefore
  \begin{equation}\label{eq:weight-combine}
  L^{1/t}\frac fF\left(\frac{L'}L-d\frac fF\right)
  \ge[c(1-\chi)-d\chi]L^{1/t}\frac{f^2}{F^2}.
  \end{equation}
  By \eqref{eq:tail-sum}, $S^{1/t}\ge(\sum_i(F_iw_i)^2)^{1/2}\ge Fw$. Lemma~\ref{lem:growth-tail} gives
  \[
  \frac{s^2}{Fw}\le\frac{s}{b_\mu\sqrt F}
  \le\frac{A_\mu\sqrt F}{b_\mu f}.
  \]
  To control the possible negative contribution of $T$ in
  \eqref{eq:weight-combine}, we compare its size with $S$:
  \[
  \frac\chi{1-\chi}=\frac TS
  \le\left(\frac{A_\mu\Lambda\sqrt F}{b_\mu f}\right)^t.
  \]
  Absorbing the fixed factor $A_\mu/b_\mu$ into $C$, the preceding
  estimate gives
  \[
  \chi\le C(1-\chi)\left(\frac{\Lambda\sqrt F}{f}\right)^t.
  \]
  Young's inequality, with conjugate exponents $2/(2-t)$ and $2/t$,
  gives, for every $a>0$,
  \[
  C(1-\chi)\left(\frac{\Lambda\sqrt F}{f}\right)^t
  \le a(1-\chi)^{2/(2-t)}
     +C_a\frac{\Lambda^2F}{f^2}
  \le a(1-\chi)+C_a\frac{\Lambda^2F}{f^2}.
  \]
  Here $C_a$ depends only on $n,\mu,a$. Choose $a=c/(2d)$ in \eqref{eq:weight-combine}. Then
  \begin{equation}\label{eq:weight-half}
  L^{1/t}\frac fF\left(\frac{L'}L-d\frac fF\right)
  \ge\frac c2(1-\chi)L^{1/t}\frac{f^2}{F^2}
  -C\Lambda^2\frac{L^{1/t}}F.
  \end{equation}
  For the remaining positive term, we seek a lower bound for its
  product with the error weight $L^{1/t}/F$. Since $0<t<1$,
  $\lambda\ge1$ and $F_i\le F$,
  \[
  S^{1/t}\ge\sum_iF_iw_i
  \ge\sqrt F\sum_i\nu_i\sqrt{F_i}w_i,
  \qquad
  \frac fF=\sum_i\frac{\nu_i}{\Theta_i\sqrt{F_i}w_i}
  \ge\theta^{-1}\sum_i\frac{\nu_i}{\sqrt{F_i}w_i}.
  \]
  The weighted Cauchy inequality
  $(\sum_i\nu_i b_i)(\sum_i\nu_i/b_i)\ge(\sum_i\nu_i)^2=1$,
  applied with $b_i=\sqrt{F_i}w_i$, therefore gives
  $S^{1/t}f/F\ge\sqrt F/\theta$. Since $S=(1-\chi)L$ and $L\ge S$,
  \[
  \left((1-\chi)L^{1/t}\frac{f^2}{F^2}\right)\frac{L^{1/t}}F
  =\frac{SL^{2/t-1}f^2}{F^3}
  \ge\frac{S^{2/t}f^2}{F^3}\ge\theta^{-2}.
  \]
  Substituting this in \eqref{eq:weight-half}, and using
  \[
  \frac{cF}{2\theta^2 L^{1/t}}
  +\frac{\delta^2\theta^2}{2c}\Lambda^2\frac{L^{1/t}}F
  \ge\delta\Lambda,
  \]
  proves \eqref{eq:mult-zero}. Choose $0<\delta\le2t(\mu-t)$ and enlarge $C_0$; both constants depend only on $n,\mu$.
  
  For the upper bound in \eqref{eq:weight-size}, $\lambda\ge2$ and $\lambda\le C(n,\mu)C_2^{t/2}$ imply
  \[
  S^{1/t}\le\sqrt{C_2}\left(\sum_{i=1}^m\lambda^{i-1}\right)^{1/t}Fw
  \le C(n,\mu)^m C_2^{m/2}Fw.
  \]
  Also $T^{1/t}=\Lambda s^2\le b_\mu^{-2}\Lambda Fw^2$. The inequality $(S+T)^{1/t}\le2^{1/t-1}(S^{1/t}+T^{1/t})$ finishes the proof.
  \end{proof}
  
  \begin{proof}[Proof of Theorem~\ref{thm:local-tail}]
  Fix $x\in\Omega$ and $0<R<d_\Omega(x)/2$; then
  $B_{2R}(x)\subset\Omega$. Put $\Lambda=R^{-1}+\sqrt\kappa$. Proposition~\ref{prop:localize} and Lemma~\ref{lem:weight} give $E/L^{1/t}\le C(n,\mu)\Lambda$ on $B_R(x)$. Since
  \[
  E=F\left(1+\tfrac{1}{2}|\nabla w(u)|^2\right),
  \]
  division by $Fw^2$ and \eqref{eq:weight-size} give, for a constant $D=D(n,\mu,m,C_2)$,
  \[
  w(u)^{-2}+\tfrac{1}{2}|\nabla\log w(u)|^2
  \le D\Lambda w(u)^{-1}+D\Lambda^2
  \le\tfrac12w(u)^{-2}+(D^2/2+D)\Lambda^2.
  \]
  Let $R\uparrow d_\Omega(x)/2$, or $R\to\infty$ if $d_\Omega(x)=\infty$.
  \end{proof}
  
  \begin{corollary}\label{cor:global-bounds}
  On a complete manifold with \eqref{eq:ric-lower}, an admissible positive entire solution satisfies
  \[
  |\nabla\log w(u)|+w(u)^{-1}\le C\sqrt\kappa.
  \]
  There is no such solution for $\kappa=0$. For $\kappa>0$, define $B_0$ by $w(B_0)=(C\sqrt\kappa)^{-1}$. Then
  \begin{equation}\label{eq:global-bounds}
  u\le B_0,\qquad
  \|\nabla u\|_\infty
  \le C\bigl(B_0+\sqrt{F(B_0)}w(B_0)\bigr)\sqrt\kappa.
  \end{equation}
  \end{corollary}
  \begin{proof}
  The first statement is Theorem~\ref{thm:local-tail} with $\Omega=M$.
  If $\kappa=0$, it forces $w(u)^{-1}=0$, impossible at any point
  where $u>0$ because the tail is finite there. For $\kappa>0$,
  $f(0)=f'(0)=0$ implies $f(s)=o(s)$ and hence $F(s)=o(s^2)$.
  In particular, for every $A>0$, one has $F(s)\le A^{-2}s^2$
  near zero, so integration of $F(s)^{-1/2}\ge A/s$ shows that
  $w(s)\to\infty$ as $s\downarrow0$. Since $w$ decreases strictly
  to zero at infinity, $B_0$ is uniquely defined and the amplitude
  bound follows. Since $F'=f>0$, for $0<s\le B_0$ we have
  \[
  \sqrt{F(s)}w(s)
  =\int_s^{B_0}\sqrt{\frac{F(s)}{F(r)}}\,dr+\sqrt{F(s)}w(B_0)
  \le B_0+\sqrt{F(B_0)}w(B_0).
  \]
  Use $|\nabla u|=\sqrt{F(u)}w(u)|\nabla\log w(u)|$.
  \end{proof}
  
  \subsection{Power mixtures}
  To obtain a bound uniform in the mixing measure, we group exponents
  into intervals of length at most one and apply the next lemma to
  the measure on each interval.

  \begin{lemma}\label{lem:short-mixture}
  Let $\nu$ be a nonzero finite positive measure supported in $[\alpha,\beta]\subset(1,\infty)$, with $\beta-\alpha\le1$. Set
  \[
  f(s)=\int s^p\,d\nu(p),\qquad F(s)=\int\frac{s^{p+1}}{p+1}\,d\nu(p),\qquad
  w(s)=\int_s^\infty F(t)^{-1/2}\,dt.
  \]
  The tail is finite and $(Fw)^{2/(\beta+3)}$ is concave.
  \end{lemma}
  \begin{proof}
  For $s\ge1$, $F(s)\ge\nu([\alpha,\beta])s^{\alpha+1}/(\beta+1)$, so the tail is finite. Introduce the probability measure
  \[
  d\varrho_s(p)=\frac{s^{p+1}}{(p+1)F(s)}\,d\nu(p).
  \]
  Let $\ell,V,T$ be the mean, variance and third centered moment of
  $p+1$ with respect to $\varrho_s$. Differentiating the normalized
  weight with respect to $\log s$ and taking moments gives
  \[
  \ell=sf/F,\quad s\partial_s\varrho_s=(p+1-\ell)\varrho_s,
  \quad s\ell'=V,\quad sV'=T.
  \]
  Writing $Q=Ff'/f^2$, differentiation gives
  \begin{equation}\label{eq:moment-Q}
  Q=1-\frac1\ell+\frac V{\ell^2},\qquad
  sQ'=\frac{\ell(V+T)-2V^2}{\ell^3}.
  \end{equation}
  The support bounds imply
  \[
  T+(\ell-\alpha-1)V
  =\int(p+1-\ell)^2(p-\alpha)\,d\varrho_s\ge0,
  \]
  and
  \[
  (\ell-\alpha-1)(\beta+1-\ell)-V
  =\int(p-\alpha)(\beta-p)\,d\varrho_s\ge0.
  \]
  Consequently
  \begin{align*}
  \ell(V+T)-2V^2
  &\ge[\ell(\alpha+2-\ell)-2(\ell-\alpha-1)(\beta+1-\ell)]V\\
  &=[(\beta+1)(\alpha+1-\beta)
  +(\beta+1-\ell)(\alpha+\beta+1-\ell)]V\ge0.
  \end{align*}
  The last expression is nonnegative because
  $\alpha+1-\beta\ge0$ and
  $\alpha+\beta+1-\ell\ge\alpha>0$. Thus $Q$ is nondecreasing.
  Also $\ell>2$ implies $Q>1/2$. For the upper bound, the support
  length restriction and $\ell\le\beta+1$ give
  \[
  \ell-\alpha-1\le\ell-\beta\le\frac\ell{\beta+1},\qquad
  V\le\frac{\ell(\beta+1-\ell)}{\beta+1}.
  \]
  Substitution in \eqref{eq:moment-Q} gives
  \[
  Q\le1-\frac1\ell+
  \frac{\beta+1-\ell}{(\beta+1)\ell}
  =\frac\beta{\beta+1}.
  \]
  
  For $s\to\infty$, $\sqrt F/f\le s/\sqrt{(\alpha+1)F}\to0$. Integrating
  \[
  (\sqrt F/f)'=(1/2-Q)F^{-1/2}
  \]
  from $s$ to infinity gives
  \[
  \Theta(s)=\frac{\sqrt F}{fw}
  =\frac1{w(s)}\int_s^\infty(Q(t)-1/2)F(t)^{-1/2}\,dt.
  \]
  Hence
  \begin{equation}\label{eq:moment-theta}
  Q(s)-\tfrac{1}{2}\le\Theta(s)\le\frac{\beta-1}{2(\beta+1)},
  \qquad \frac{F\Theta'}{f\Theta}=\Theta+\tfrac{1}{2}-Q\ge0.
  \end{equation}
  To obtain the concavity from these moment bounds, note first that
  \[
  \frac{(Fw)'}{Fw}=(1-\Theta)\frac fF,\qquad
  \left(\frac fF\right)'=(Q-1)\left(\frac fF\right)^2.
  \]
  Consequently, for $0<\mu<1/2$,
  \begin{align*}
  -\frac{((Fw)^\mu)''}{\mu(Fw)^\mu}
  &=\frac fF\Theta'
  +\left(\frac fF\right)^2(1-\Theta)
  \bigl[1-Q-\mu(1-\Theta)\bigr].
  \end{align*}
  Substituting $Q=\Theta+1/2-F\Theta'/(f\Theta)$ yields
  \[
  -\frac{((Fw)^\mu)''}{\mu(Fw)^\mu}
  =\left(\frac fF\right)^2\left\{
  \frac{F\Theta'}{f\Theta}
  +(1-\Theta)\bigl[\tfrac{1}{2}-\mu-(1-\mu)\Theta\bigr]\right\}.
  \]
  Choose $\mu=2/(\beta+3)$; then $(1/2-\mu)/(1-\mu)=(\beta-1)/(2(\beta+1))$. Both terms in braces are nonnegative by \eqref{eq:moment-theta}.
  \end{proof}
  
  \begin{proposition}\label{prop:mixtures}
  Every nonzero finite positive measure supported in $[p_-,p_+]\Subset(1,p_S(n))$ defines an admissible source $f(s)=\int s^p\,d\nu(p)$, with
  \[
  \mu=\frac2{p_++3},\qquad
  C_2=\left(\frac{p_+-1}{p_--1}\right)^2,
  \qquad m\le\max\{1,\lceil p_+-p_-\rceil\}.
  \]
  \end{proposition}
  \begin{proof}
  Differentiation under the integral on positive compact intervals is justified. For $0<s\le1$,
  \[
  f(s)\le\nu([p_-,p_+])s^{p_-},\qquad
  f'(s)\le p_+\nu([p_-,p_+])s^{p_--1},
  \]
  which proves the regularity and zero-value conditions for admissibility. Partition the exponent interval into the indicated number of ordered intervals of length at most one, taking them left-closed and right-open except for the last. Discard those of zero mass, and denote the endpoints of each remaining interval by $\alpha_j,\beta_j$. The corresponding components satisfy
  \[
  \alpha_j+1\le sf_j/F_j\le\beta_j+1
  \le\alpha_{j+1}+1\le sf_{j+1}/F_{j+1},
  \]
  so they are ordered. Lemma~\ref{lem:short-mixture} gives concavity with exponent $2/(\beta_j+3)$; composition with an increasing concave power gives the common exponent $\mu$.
  
  For $t\ge s>0$,
  \[
  (t/s)^{p_-+1}\le F_j(t)/F_j(s)\le(t/s)^{p_++1}.
  \]
  Integration gives
  \begin{equation}\label{eq:power-tail-bounds}
  \frac{2s}{(p_+-1)\sqrt{F_j(s)}}\le w_j(s)
  \le\frac{2s}{(p_--1)\sqrt{F_j(s)}}.
  \end{equation}
  The same inequalities hold for $F,w$. Therefore
  \[
  \sum_j(F_jw_j)^2\le\frac{4s^2F}{(p_--1)^2}
  \le\left(\frac{p_+-1}{p_--1}\right)^2(Fw)^2.
  \]
  
  \end{proof}
  
  \begin{proposition}\label{prop:local-mixtures}
  Let $(M^n,g)$ be complete, $n\ge2$, and $\Omega\subset M$ open, with
  $\operatorname{Ric}\ge-(n-1)\kappa g$ on $\Omega$, $\kappa\ge0$.
  Let $\nu$ be a nonzero finite positive measure supported in
  $[p_-,p_+]\Subset(1,p_S(n))$.
  Every positive $C^3$ solution of
  $-\Delta u=\int_{[p_-,p_+]}u^p\,d\nu(p)$ in $\Omega$ satisfies
  \[
  |\nabla\log u|^2+\int_{[p_-,p_+]}u^{p-1}\,d\nu(p)
  \le C(n,p_-,p_+)\bigl(d_\Omega^{-2}+\kappa\bigr).
  \]
  \end{proposition}
  \begin{proof}
  We apply Theorem~\ref{thm:local-tail} and Proposition~\ref{prop:mixtures}. The tail bounds and $sf/F\le p_++1$ give
  \[
  |\nabla\log u|
  =\frac{\sqrt{F(u)}w(u)}u|\nabla\log w(u)|
  \le\frac2{p_--1}|\nabla\log w(u)|,
  \qquad
  \frac{f(s)w(s)^2}s\le\frac{4(p_++1)}{(p_--1)^2}.
  \]
  Thus
  \[
  |\nabla\log u|^2+\frac{f(u)}u
  \le\frac{4(p_++1)}{(p_--1)^2}
  \bigl(|\nabla\log w(u)|^2+w(u)^{-2}\bigr),
  \]
  which is the assertion.
  \end{proof}
  
  \begin{proof}[Proof of Theorem~\ref{thm:intro-mixtures}]
  Write $f(s)=\sum_i a_i s^{p_i}$ with $a_i>0$, $\sum_i a_i=1$ and
  $p_i\in[p_-,p_+]$. Apply Proposition~\ref{prop:local-mixtures} with
  $\nu=\sum_i a_i\delta_{p_i}$ and $\Omega=B_{2R}(x_0)$.
  On $B_R(x_0)$ one has $d_\Omega\ge R$, which gives the stated estimate.
  \end{proof}

\section{Sharp comparison and rigidity}\label{sec:comparison}
  \subsection{Radial models}
  Throughout this section $f$ is admissible and $\kappa>0$. We use $k>0$
  for an auxiliary model-curvature parameter. For a model curvature $-k$,
  write
  \[
  \bar g_k=dr^2+S_k(r)^2g_{\mathbb S^{n-1}},\qquad
  S_k(r)=\frac{\sinh(\sqrt k r)}{\sqrt k}.
  \]
  Here $g_{\mathbb S^{n-1}}$ is the metric of the unit round sphere,
  and $r$ is distance from a fixed point, called the pole. The volume
  element and the radial Laplacian are
  \[
  d\mu_{\bar g_k}=S_k(r)^{n-1}\,dr\,d\omega,\qquad
  \Delta_{\bar g_k}U
  =\frac1{S_k^{n-1}}\bigl(S_k^{n-1}U'\bigr)'
  =U''+(n-1)\frac{S_k'}{S_k}U'.
  \]
  Since $S_k''=kS_k$ and $(S_k')^2=1+kS_k^2$, both the radial and
  tangential sectional curvatures are $-k$; in particular,
  $\operatorname{Ric}_{\bar g_k}=-(n-1)k\bar g_k$.
  Let $U_{k,B}$ denote the positive branch of the regular radial solution
  \begin{equation}\label{eq:radial}
  U''+(n-1)\frac{S_k'}{S_k}U'+f(U)=0,
  \qquad U(0)=B>0,\quad U'(0)=0,
  \end{equation}
  and define
  \begin{equation}\label{eq:threshold}
  A_k=\sup\{B>0:U_{k,B}>0\text{ on }[0,\infty)\}.
  \end{equation}
  
  The integral equation
  \begin{equation}\label{eq:radial-integral}
  U(r)=B-\int_0^r S_k(t)^{1-n}\int_0^t S_k(s)^{n-1}f(U(s))\,ds\,dt
  \end{equation}
  constructs the solution at the center. Since $S_k(s)/S_k(t)\le s/t$ for $0\le s\le t$, the right side has Lipschitz constant at most $LR^2/(2n)$ on $C[0,R]$, where $L$ is a Lipschitz constant for $f$ near $B$. For sufficiently small $R$ it preserves a neighborhood of the constant function $B$ and is a contraction. The choices are uniform when $(k,B)$ varies near a fixed positive pair. Moreover,
  \begin{equation}\label{eq:radial-flux}
  U'(r)=-S_k(r)^{1-n}\int_0^r S_k(t)^{n-1}f(U(t))\,dt,
  \end{equation}
  so the solution depends continuously in $C^1[0,R]$ on the parameters, and
  \[
  U'(r)=-\frac{f(B)}n r+O(r^3),\qquad
  U''(r)=-\frac{f(B)}n+O(r^2),\qquad U'''(r)=O(r).
  \]
  In particular, the radial function is $C^3$ at the origin and has Hessian $-f(B)g/n$ there.
  
  On the positive branch $U'<0$ for $r>0$, while
  \[
  \left(\frac{U'^2}2+F(U)\right)'
  =-(n-1)\frac{S_k'}{S_k}U'^2\le0.
  \]
  Thus $0<U\le B$ and $|U'|\le\sqrt{2F(B)}$. A finite endpoint of this branch must be a zero, and \eqref{eq:radial-flux} makes that zero transverse. Extend $f$ by zero on $(-\infty,0)$. This is a $C^1$ extension since $f(0)=f'(0)=0$; after the first zero, $S_k^{n-1}U'$ is constant. The extended solution exists on every finite interval and depends continuously in $C^1$ on $(k,B)$ there.
  
For a single power $f(s)=s^p$, $1<p<p_S(n)$, Mancini and Sandeep
\cite{Mancini:2008aa} prove existence, radial symmetry and uniqueness
up to hyperbolic isometries of positive $H^1$ solutions.
Bonforte, Gazzola, Grillo and V\'azquez \cite{Bonforte:2013aa}
identify the ground state as the radial threshold between positive
entire profiles and profiles with a finite first zero.
For the admissible source class used here, the following lemma proves
existence and attainment of the maximal central value, together with
parameter semicontinuity and the decay needed below.

  \begin{lemma}\label{lem:threshold}
  For every $k>0$, one has $0<A_k<\infty$, and $U_{k,A_k}$ is positive on $[0,\infty)$. Every globally positive radial profile tends to zero, together with its derivative. The threshold profile satisfies
  \begin{equation}\label{eq:fast-threshold}
  -\frac{U_{k,A_k}'}{U_{k,A_k}}\longrightarrow(n-1)\sqrt k,
  \qquad U_{k,A_k}(r)=O(e^{-cr})
  \quad\text{for every }0<c<(n-1)\sqrt k.
  \end{equation}
  Also $\limsup_{k\to k_0}A_k\le A_{k_0}$ for each $k_0>0$.
  \end{lemma}
  \begin{proof}
  Put $W=-U'$. Since $f'$ is continuous at zero and $f'(0)=0$,
  choose $B>0$ so small that
  \[
  \frac{4}{(n-1)^2k}\max_{[0,B]}|f'|<1.
  \]
  Initially $W<\frac{2}{(n-1)\sqrt k}f(U)$. At a first contact
  with this bound,
  \[
  \left(W-\frac{2f(U)}{(n-1)\sqrt k}\right)'
  =f(U)\left[1-2\coth(\sqrt k r)
  +\frac{4f'(U)}{(n-1)^2k}\right]<0,
  \]
  which is impossible at a first contact from below. Therefore
  \[
  \int_{U(r)}^B\frac{ds}{f(s)}
  =\int_0^r\frac{-U'}{f(U)}\,dt
  \le\frac{2r}{(n-1)\sqrt k}.
  \]
  Because $f(s)=o(s)$ at zero, the integral diverges as $U(r)\downarrow0$. This rules out a finite zero and proves $A_k>0$.
  
  For every globally positive profile, Corollary~\ref{cor:global-bounds} at the center gives
  \[
  w(B)^{-1}\le C(n,\mu,m,C_2)\sqrt k.
  \]
  Since $w(B)\to0$ as $B\to\infty$, the set in \eqref{eq:threshold} is bounded. Choose globally positive profiles with center values $B_j\to A_k$. Their extended solutions converge on every finite interval to $U_{k,A_k}$, which is nonnegative. A zero at a positive radius would have zero derivative, and uniqueness for the extended equation would force the solution to vanish identically. Thus $U_{k,A_k}>0$. The same estimate bounds $A_{k_j}$ when $k_j\to k_0>0$. Any positive subsequential limit $B$ gives a globally positive limit profile $U_{k_0,B}$, so $B\le A_{k_0}$; a zero limit already has this property. This proves upper semicontinuity.
  
  A positive profile decreases to some $\ell\ge0$. By l'H\^opital's rule,
  \[
  \frac{\int_0^r S_k(t)^{n-1}\,dt}{S_k(r)^{n-1}}
  \longrightarrow\frac1{(n-1)\sqrt k}.
  \]
  Splitting the flux integral at a large fixed radius and using $f(U)\to f(\ell)$ now gives
  \[
  -U'(r)\longrightarrow\frac{f(\ell)}{(n-1)\sqrt k}.
  \]
  If $\ell>0$, the limiting speed is positive, contradicting positivity on the whole half-line. Hence $\ell=0$ and $U'\to0$.
  
  For the threshold profile, write $\lambda=(n-1)\sqrt k$. Fix $0<c<\lambda$, and choose $\delta>0$ such that $f(s)<c(\lambda-c)s$ for $0<s\le\delta$. In the system
  \[
  U'=-W,\qquad W'=f(U)-(n-1)\sqrt k\coth(\sqrt k r)W,
  \]
  the region $0<U<\delta$, $0<W<cU$ is forward invariant. Indeed, $W'=f(U)>0$ at $W=0$, and
  \[
  (W-cU)'=f(U)-c\bigl((n-1)\sqrt k\coth(\sqrt k r)-c\bigr)U<0\qquad\text{at }W=cU.
  \]
  Inside the region $U(r)\ge U(R)e^{-c(r-R)}$, so a finite zero is impossible. If the threshold profile entered it, continuous dependence would place a profile with center value $B>A_k$ inside the same region after a finite time, with positivity up to that time. This contradicts the definition of $A_k$. Consequently $-U'/U\ge c$ once $U<\delta$, and integration gives the stated exponential bound.
  
  Finally, $y=-U'/U$ satisfies
  $y'=y^2-(n-1)\sqrt k\coth(\sqrt k r)y+f(U)/U$.
  Given $\varepsilon>0$, take $R$ such that
  $(n-1)\sqrt k\coth(\sqrt k r)\le\lambda+\varepsilon$
  on $[R,\infty)$. If $y(r_0)>\lambda+\varepsilon$ for some
  $r_0\ge R$, then $y'\ge y(y-\lambda-\varepsilon)>0$ thereafter, and
  \[
  r-r_0\le\int_{y(r_0)}^{y(r)}\frac{d\tau}{\tau(\tau-\lambda-\varepsilon)}
  \le\frac1{\lambda+\varepsilon}
  \log\frac{y(r_0)}{y(r_0)-\lambda-\varepsilon}.
  \]
  This forces blow-up in finite time, whereas the positive profile makes $y$ finite at every finite radius. The upper bound, together with the lower bound for every $c<\lambda$, proves \eqref{eq:fast-threshold}.
  \end{proof}
  
  Set $U_H=U_{\kappa,A_\kappa}$ and define its phase curve by
  \[
  \Psi_H(s)=[U_H'(U_H^{-1}(s))]^2\quad(0<s<A_\kappa),
  \qquad \Psi_H(A_\kappa)=0.
  \]
  \begin{theorem}\label{thm:gradient-general}
  Let $(M^n,g)$ be complete and connected without boundary, $n\ge2$,
  with $\operatorname{Ric}\ge-(n-1)\kappa g$, $\kappa>0$.
  Suppose $f$ is admissible.
  For the threshold $A_\kappa$ and model profile $U_H$ defined above,
  every positive $C^3$ solution of $-\Delta u=f(u)$ satisfies
  \begin{equation}\label{eq:general-gradient}
  u\le A_\kappa,\qquad |\nabla u|^2\le\Psi_H(u).
  \end{equation}
  Equality in the gradient bound at one point, including a critical point,
  implies that $(M,g)$ is isometric to $\mathbb H^n_{-\kappa}$ and
  $u(x)=U_H(d(o,x))$ for some $o\in M$. The model attains equality everywhere.
  \end{theorem}
  
  \begin{lemma}\label{lem:phase-center}
  The function $\Psi_H$ belongs to $C^2((0,A_\kappa])$ and extends $C^2$ through $A_\kappa$, with
  \begin{equation}\label{eq:phase-center}
  \Psi_H'(A_\kappa)=-\frac{2f(A_\kappa)}n,
  \qquad
  \Psi_H''(A_\kappa)=-\frac{4(n-1)\kappa+6f'(A_\kappa)}{n+2}.
  \end{equation}
  On $(0,A_\kappa)$ it satisfies
  \begin{equation}\label{eq:phase-ode}
  2(n-1)\Psi_H\bigl(\Psi_H''+2f'+2(n-1)\kappa\bigr)
  =(\Psi_H'+2f)(n\Psi_H'+2f).
  \end{equation}
  The center extension is smooth if $f$ is smooth near $A_\kappa$.
  \end{lemma}
  \begin{proof}
  Put $y=\cosh(\sqrt\kappa r)$ and $V(y)=U_H(r)$. Then
  \[
  \kappa(y^2-1)V''+n\kappa yV'+f(V)=0,\qquad V(1)=A_\kappa.
  \]
  The radial flux gives $V'(1+x)=O(1)$ as $x\downarrow0$. An integrating factor followed by $t=xz$ yields
  \begin{equation}\label{eq:center-integral}
  V'(1+x)=-\frac{(2+x)^{-n/2}}\kappa
  \int_0^1 z^{n/2-1}(2+xz)^{n/2-1}f(V(1+xz))\,dz.
  \end{equation}
  More explicitly, the equation for $V$ can be written
  \[
  \bigl((y^2-1)^{n/2}V'(y)\bigr)'
  =-\frac1\kappa(y^2-1)^{n/2-1}f(V(y)).
  \]
  Its integration constant is zero by the boundedness of $V'$ at
  $y=1$. Integrating from $1$ to $1+x$ and then substituting $t=xz$
  cancels the factor $x^{n/2}$ and gives
  \eqref{eq:center-integral}. The resulting formula is meaningful
  for negative $x$ as well, because $2+xz$ stays positive.
  After integration from zero to $x$, this equation is a contraction
  on a small interval $|x|\le\varepsilon$ and gives a $C^1$ extension
  through $y=1$. Since $A_\kappa>0$ and $f\in C^2((0,\infty))$,
  shrink the interval so that $V$ stays in a compact subinterval of
  $(0,\infty)$. Differentiating once gives $V\in C^2$ and twice gives
  $V\in C^3$; the argument iterates if $f$ is smooth. At the center,
  \[
  V'(1)=-\frac{f(A_\kappa)}{n\kappa}<0,\qquad
  (n+2)\kappa V''(1)+(n\kappa+f'(A_\kappa))V'(1)=0.
  \]
  The inverse function theorem and
  \[
  \Psi_H(V(y))=\kappa(y^2-1)V'(y)^2
  \]
  give the extension. Differentiating this identity gives
  \[
  \Psi_H'(A_\kappa)=2\kappa V'(1),\qquad
  \Psi_H''(A_\kappa)=2\kappa\left(1+3\frac{V''(1)}{V'(1)}\right),
  \]
  which proves \eqref{eq:phase-center}.
  
  For $s=U_H(r)$, the radial equation gives
  \[
  \Psi_H'(s)=2\left((n-1)\frac{S_\kappa'(r)}{S_\kappa(r)}
  \sqrt{\Psi_H(s)}-f(s)\right).
  \]
  Hence
  \[
  \frac{S_\kappa'(r)}{S_\kappa(r)}
  =\frac{\Psi_H'(s)+2f(s)}{2(n-1)\sqrt{\Psi_H(s)}}.
  \]
  On the other hand, $(S_\kappa'/S_\kappa)'=\kappa-(S_\kappa'/S_\kappa)^2$
  and $ds/dr=-\sqrt{\Psi_H(s)}$, so
  \[
  \frac{d}{ds}\left(\frac{S_\kappa'(r)}{S_\kappa(r)}\right)
  =\frac{(S_\kappa'(r)/S_\kappa(r))^2-\kappa}
  {\sqrt{\Psi_H(s)}}.
  \]
  Differentiating the preceding expression directly gives
  \[
  \frac{\Psi_H''+2f'}{2(n-1)\sqrt{\Psi_H}}
  -\frac{(\Psi_H'+2f)\Psi_H'}{4(n-1)\Psi_H^{3/2}}
  =\frac{(\Psi_H'+2f)^2}{4(n-1)^2\Psi_H^{3/2}}
  -\frac{\kappa}{\sqrt{\Psi_H}}.
  \]
  Multiplication by $4(n-1)^2\Psi_H^{3/2}$ and collection of terms give
  \eqref{eq:phase-ode}.
  \end{proof}
  
  \subsection{Comparison and rigidity}
  In the model coordinate $t=\cosh(\sqrt k r)$,
  \[
  \nabla_{\bar g_k}^2t=kt\bar g_k,\qquad
  |dt|_{\bar g_k}^2=k(t^2-1),\qquad \Delta_{\bar g_k}t=nkt.
  \]
  \begin{lemma}\label{lem:model-coordinate}
  Suppose $U_{k,B}$ has first zero $R<\infty$. Define
  \[
  V(t)=U_{k,B}\left(\frac{\operatorname{arcosh}t}{\sqrt k}\right),
  \quad 1\le t\le\cosh(\sqrt k R),\qquad T=V^{-1}.
  \]
  Then $V'<0$, $T\in C^3([0,B])$, and $T'$ never vanishes. For a solution $0<u\le B$ of \eqref{eq:main}, put
  \[
  v=T(u),\quad Q=|\nabla v|^2,\quad P=Q-k(v^2-1),\quad
  \alpha=-V''/V',\quad
  \mathcal L=\Delta-2\alpha(v)\langle\nabla v,\nabla\,\cdot\,\rangle.
  \]
  Primes on $V$ and $\alpha$ denote derivatives with respect to $t$;
  $\alpha$ and $\alpha'$ below are evaluated at $t=v$.
  The model satisfies $P=0$.
  Since $T'$ is nonzero, the sign of $P$ is the sign of the desired
  gradient comparison. The operator $\mathcal L$ has the same
  second-order part as $\Delta$; its first-order term cancels the
  derivative of $P$ in Bochner's formula.
  Then
  \begin{align}
  \Delta v&=nkv+\alpha(v)P,\label{eq:model-laplacian}\\
  \mathcal LP&=2|\nabla^2v-kvg|^2
  +2(\operatorname{Ric}+(n-1)\kappa g)(\nabla v,\nabla v)\notag\\
  &\quad+2(n-1)(k-\kappa)Q
  +2\bigl(k\alpha(v)v+\alpha'(v)Q\bigr)P.\label{eq:model-bochner}
  \end{align}
  For the model phase curve $\Psi_{k,B}(s)=[U_{k,B}'(U_{k,B}^{-1}(s))]^2$,
  \begin{equation}\label{eq:phase-defect}
  P=T'(u)^2\bigl(|\nabla u|^2-\Psi_{k,B}(u)\bigr).
  \end{equation}
  Near a fixed finite-zero model $(k_0,B_0)$, the inverse coordinates depend continuously in $C^1$ on every common height interval $[0,a]\subset[0,\min(B,B_0)]$.
  \end{lemma}
  \begin{proof}
  Equation \eqref{eq:center-integral}, with $\kappa$ replaced by $k$, gives $V\in C^3$ through $t=1$ and
  \[
  V'(1)=-\frac{f(B)}{nk},\qquad
  V''(1)=\frac{f(B)(nk+f'(B))}{n(n+2)k^2}.
  \]
  For $r>0$,
  \[
  V'(\cosh(\sqrt k r))=\frac{U_{k,B}'(r)}{\sqrt k\sinh(\sqrt k r)}<0,
  \]
  including $r=R$ by transversality. At this endpoint the radial equation is nonsingular: writing
  $h_k=(n-1)S_k'/S_k$, it gives
  \[
  U^{(3)}=-h_k'U'-h_kU''-f'(U)U'.
  \]
  The $C^1$ zero extension of $f$ makes this continuous through
  $r=R>0$. The change from $r$ to $t$ has nonzero derivative there,
  so $V\in C^3$ at that endpoint as well. Thus $|V'|$ has a positive minimum on the compact coordinate interval, and its inverse has the asserted regularity through both endpoints.
  
  Using $u=V(v)$ and the model equation,
  \[
  V'\Delta v+V''Q=-f(V)
  =k(v^2-1)V''+nkvV',
  \]
  which is \eqref{eq:model-laplacian}. To compute the defect, first
  differentiate this equation:
  \[
  \nabla\Delta v=nk\nabla v+\alpha'(v)P\nabla v
  +\alpha(v)\nabla P.
  \]
  With $Q=|\nabla v|^2$, Bochner's identity and the chain rule give
  \begin{align*}
  \Delta Q
  &=2|\nabla^2v|^2+2\langle\nabla v,\nabla\Delta v\rangle
    +2\operatorname{Ric}(\nabla v,\nabla v),\\
  \Delta\bigl(k(v^2-1)\bigr)&=2kv\Delta v+2kQ.
  \end{align*}
  Subtraction and cancellation of
  $2\alpha\langle\nabla v,\nabla P\rangle$ yield
  \[
  \mathcal LP=2|\nabla^2v|^2+2\operatorname{Ric}(\nabla v,\nabla v)
  +2(n-1)kQ-2nk^2v^2+2(\alpha'Q-k\alpha v)P.
  \]
  Since
  \[
  2|\nabla^2v-kvg|^2
  =2|\nabla^2v|^2-2nk^2v^2-4k\alpha vP,
  \]
  this is \eqref{eq:model-bochner}. Finally,
  $T'(s)^2\Psi_{k,B}(s)=k(T(s)^2-1)$ proves \eqref{eq:phase-defect}.
  
  For parameter dependence, we extend the fixed model a little beyond its transverse zero. On an enlarged coordinate interval $J=[1,\cosh(\sqrt{k_0}R_0)+\delta]$, choose $\delta>0$ so small that $\min_J|V_0'|>0$ and $V_0$ is negative at the right endpoint. The center integral and ordinary dependence away from the center give $V_{k,B}\to V_0$ in $C^1(J)$. For nearby parameters, the inverse images of the common height interval remain in $J$, and $|V_{k,B}'|\ge c>0$ there. Hence
  \[
  \|T_{k,B}-T_0\|_{C[0,a]}
  \le c^{-1}\|V_{k,B}-V_0\|_{C(J)}\to0.
  \]
  The identity $T'=1/(V'\circ T)$ proves convergence of the derivatives.
  \end{proof}
  We use the Omori--Yau maximum principle in its Laplacian form
\cite{Omori:1967aa,Yau:1975aa}.
\begin{lemma}\label{lem:maximum-sequence}
  Let $M$ be complete and connected without boundary with
  $\operatorname{Ric}\ge-(n-1)\kappa g$, $\kappa\ge0$.
  If $J\in C^2(M)$ is bounded above, there are points $x_j$ such that
  \[
  J(x_j)>\sup_M J-j^{-1},\qquad
  |\nabla J|(x_j)<j^{-1},\qquad \Delta J(x_j)<j^{-1}.
  \]
  \end{lemma}

  \begin{lemma}\label{lem:strict-gap}
  Suppose $M$ is complete without boundary and $\operatorname{Ric}\ge-(n-1)\kappa g$. In Lemma~\ref{lem:model-coordinate}, assume $k>\kappa$ and $|\nabla u|$ bounded. If $P\le0$ on $M$, then $\sup_M P<0$.
  \end{lemma}
  \begin{proof}
  The functions $v,\nabla v,\alpha(v),\alpha'(v)$ are bounded, and
  $v\ge1$. Indeed, $v$ takes values in the compact coordinate interval
  of a finite-zero profile, $T'$ is bounded there, and
  $\nabla v=T'(u)\nabla u$. These bounds are the reason for using
  finite-zero profiles first: the threshold inverse coordinate is
  unbounded as the height approaches zero.
  From \eqref{eq:model-bochner},
  \[
  \mathcal LP\ge2(n-1)(k-\kappa)Q+cP,
  \qquad c=2(k\alpha v+\alpha'Q).
  \]
  Suppose $\sup P=0$. Put $J_\varepsilon=P+\varepsilon(v^2-1)$, where $0<\varepsilon<\min(1,k/2)$. This perturbation has $\sup J_\varepsilon\ge0$. Lemma~\ref{lem:maximum-sequence} gives points $x_j$ with
  \[
  J_\varepsilon(x_j)>\sup J_\varepsilon-j^{-1},\qquad
  |\nabla J_\varepsilon|(x_j)<j^{-1},\qquad
  \Delta J_\varepsilon(x_j)<j^{-1}.
  \]
  At these points,
  \[
  -P\le\varepsilon(v^2-1)+j^{-1},\qquad
  Q\ge(k-\varepsilon)(v^2-1)-j^{-1},\qquad
  -P\le\frac\varepsilon{k-\varepsilon}Q+\frac{k}{k-\varepsilon}j^{-1}.
  \]
  Also,
  \[
  \mathcal L(v^2-1)=2nkv^2+2\alpha vP+(2-4\alpha v)Q.
  \]
  The coefficients above are bounded independently of $\varepsilon$. Using the last bound for $-P$ gives
  \[
  \mathcal LJ_\varepsilon(x_j)
  \ge\bigl[2(n-1)(k-\kappa)-C\varepsilon\bigr]Q
  +2n\varepsilon kv^2-Cj^{-1}.
  \]
  Choose $\varepsilon$ so that $C\varepsilon<(n-1)(k-\kappa)$. Since $v\ge1$, the right side has positive limit inferior. The drift of $\mathcal L$ is bounded, so the maximum sequence instead gives $\limsup_j\mathcal LJ_\varepsilon(x_j)\le0$, a contradiction.
  \end{proof}
  
  \begin{lemma}\label{lem:log-growth}
  Put $P_\mu=2/\mu-3$. Then $1<P_\mu<p_S(n)$ and
  \[
  \frac{s f_i'(s)}{f_i(s)}\le P_\mu,\qquad
  \frac{s f'(s)}{f(s)}\le P_\mu,\qquad
  \frac{f(s)}s\longrightarrow\infty\quad(s\to\infty).
  \]
  \end{lemma}
  \begin{proof}
  Put $\theta=(1-2\mu)/(2(1-\mu))$. For a component, omit the
  index and write $\Phi=(Fw)^\mu$ and $\Theta=\sqrt F/(fw)$.
  The tangent inequality for $\Phi$ and Lemma~\ref{lem:theta} give
  \[
  \frac{sf}F\le\frac1{\mu(1-\Theta)},\qquad
  0<\Theta\le\theta.
  \]
  Differentiating $Fw$ gives $(Fw)'=fw-\sqrt F$ and
  $(Fw)''=f'w-3f/(2\sqrt F)$. Thus $\Phi''\le0$ implies
  \[
  \frac{Ff'}{f^2}
  \le1-\mu+(2\mu-\tfrac12)\Theta+(1-\mu)\Theta^2
  \le(2-3\mu)(1-\Theta).
  \]
  For the second inequality, the difference between its right side
  and the preceding quadratic is
  \[
  \frac{\Theta+2}{2}
  \bigl[1-2\mu-2(1-\mu)\Theta\bigr]\ge0.
  \]
  At a point where $f'>0$, multiplication by the bound for $sf/F$
  gives $sf'/f\le P_\mu$; if $f'\le0$, this upper bound is immediate.
  Taking the mean over components with weights $f_i/f$ gives the
  same upper bound for the total source. Lemma~\ref{lem:theta}
  gives $\mu<1/2$, so $P_\mu>1$; the lower bound on $\mu$ in
  \eqref{eq:mu} gives $P_\mu<p_S(n)$.

  For the last assertion, restore the component indices and use
  $\Theta_i\le\theta$. H\"older's inequality gives
  \[
  F=\sum_i
  \left(\frac{F_i^{3/2}}{F_iw_i}\right)^{2/3}
  (F_iw_i)^{2/3}
  \le
  \left(\sum_i\frac{F_i^{3/2}}{F_iw_i}\right)^{2/3}
  \left(\sum_i(F_iw_i)^2\right)^{1/3}.
  \]
  Therefore \eqref{eq:component-order} implies
  \[
  f\ge\frac1\theta\sum_i\frac{F_i^{3/2}}{F_iw_i}
  \ge\frac{F^{3/2}}{\theta\bigl(\sum_i(F_iw_i)^2\bigr)^{1/2}}
  \ge\frac{\sqrt F}{\theta\sqrt{C_2}\,w}.
  \]
  Combining this with $\sqrt F\,w\ge b_\mu s$ from
  Lemma~\ref{lem:growth-tail}, we obtain
  \[
  \frac{f(s)}s
  \ge\frac{b_\mu}{\theta\sqrt{C_2}\,w(s)^2}
  \longrightarrow\infty,
  \]
  since $w(s)\to0$ as $s\to\infty$.
  \end{proof}
  
We adapt the radial energy and Pohozaev argument of
\cite{Bonforte:2013aa} to admissible sources, retaining the curvature
correction needed for uniform divergence of the first-zero slope.

  \begin{lemma}\label{lem:large-speed}
  Fix $k>0$ and $a>0$. For every $B>\max\{a,A_k\}$,
  $U_{k,B}$ has a finite first zero $R_B$. As $B\to\infty$,
  \[
  R_B\longrightarrow0,\qquad
  -U_{k,B}'(R_B)\longrightarrow\infty,
  \]
  and
  \begin{equation}\label{eq:large-speed}
  \lim_{B\to\infty}\ \inf_{0\le s\le a}\Psi_{k,B}(s)=\infty.
  \end{equation}
  \end{lemma}
  \begin{proof}
  Finiteness of $R_B$ follows from the definition of $A_k$.
  The radial function $U_{k,B}$ is positive on the hyperbolic ball
  of radius $R_B$. Theorem~\ref{thm:local-tail}, applied at its
  center, gives
  \[
  w(B)^{-1}\le C\bigl(R_B^{-1}+\sqrt k\bigr).
  \]
  Since $w(B)\to0$, it follows that $R_B\le2Cw(B)\to0$ for
  sufficiently large $B$.

  Set $p=P_\mu\in(1,p_S(n))$. Since $p+1=A_\mu$,
  Lemma~\ref{lem:growth-tail} gives
  \[
  sf(s)\le(p+1)F(s)\qquad(s>0).
  \]
  Integrating $F'/F\le(p+1)/s$ on $[1,s]$ gives
  $F(s)\le F(1)s^{p+1}$ for $s\ge1$, and hence
  $f(s)\le(p+1)F(1)s^p$ there. On $[0,1]$, the conditions
  $f\in C^1([0,\infty))$ and $f(0)=0$ give $f(s)\le Cs$.
  Thus
  \[
  f(s)\le C(s+s^p)\qquad(s\ge0).
  \]
  All constants below may depend on $n,k,f$ and its admissibility
  parameters, but not on $B$.

  Write $R=R_B$, $U=U_{k,B}$ and $S=S_k$. For large $B$ we have
  $R\le1$, so $r\le S(r)\le Cr$ and $1\le S'(r)\le C$ on $[0,R]$.
  Multiplication of \eqref{eq:radial} by $S^{n-1}U$ and integration
  give
  \[
  \mathcal E:=\int_0^R S^{n-1}U'^2\,dr
  =\int_0^R S^{n-1}Uf(U)\,dr>0.
  \]
  The boundary terms vanish because $U(R)=0$ and $U'(0)=0$.
  Since $U(r)=-\int_r^R U'(t)\,dt$, Cauchy's inequality yields
  \[
  U(r)\le\mathcal E^{1/2}
  \left(\int_r^R S(t)^{1-n}\,dt\right)^{1/2}
  \le C\mathcal E^{1/2}
  \begin{cases}
  r^{-(n-2)/2},&n\ge3,\\
  \bigl(\log(R/r)\bigr)^{1/2},&n=2.
  \end{cases}
  \]
  Put $\beta=n-(n-2)(p+1)/2>0$. Integrating these bounds gives
  \[
  \int_0^R S^{n-1}U^2\,dr\le CR^2\mathcal E,
  \qquad
  \int_0^R S^{n-1}U^{p+1}\,dr
  \le CR^\beta\mathcal E^{(p+1)/2}.
  \]
  For $n=2$, the second estimate uses $\beta=2$ and
  $\int_0^1 t(\log(1/t))^{(p+1)/2}\,dt<\infty$.
  The growth bound for $f$ and the energy identity now imply
  \[
  \mathcal E\le CR^2\mathcal E
  +CR^\beta\mathcal E^{(p+1)/2}.
  \]
  Absorbing the first term for small $R$ and dividing by
  $\mathcal E>0$, we obtain
  \begin{equation}\label{eq:radial-energy-lower}
  \mathcal E\ge cR^{-2\beta/(p-1)}
  =cR^{\,n-2-4/(p-1)}.
  \end{equation}

  Define the corrected radial Pohozaev quantity
  \[
  \begin{split}
  \mathcal P(r)={}&S^n
  \left(\frac{U'^2}{2}+F(U)-\frac{n-2}{4}kU^2\right)\\
  &+\frac{n-2}{2}S^{n-1}S'UU'.
  \end{split}
  \]
  The radial equation and $S''=kS$ give
  \begin{align*}
  \left[S^n\left(\frac{U'^2}{2}+F(U)\right)\right]'
  &=S^{n-1}S'\left(nF(U)-\frac{n-2}{2}U'^2\right),\\
  \bigl(S^{n-1}S'UU'\bigr)'
  &=S^{n-1}S'\bigl(U'^2-Uf(U)\bigr)+kS^nUU'.
  \end{align*}
  The derivative of the correction $-(n-2)kS^nU^2/4$
  cancels the remaining $UU'$ term. Consequently
  \begin{equation}\label{eq:radial-pohozaev}
  \mathcal P'
  =S^{n-1}S'
  \left(nF(U)-\frac{n-2}{2}Uf(U)
  -\frac{n(n-2)}4kU^2\right).
  \end{equation}
  Since $Uf(U)\le(p+1)F(U)$, the first two terms in parentheses
  are bounded below by $\beta F(U)$. The center expansion gives
  $\mathcal P(0)=0$, while
  $\mathcal P(R)=S(R)^nU'(R)^2/2$. Integrating
  \eqref{eq:radial-pohozaev} and using the energy identity and the
  preceding $L^2$ bound, we find
  \begin{align*}
  \frac12S(R)^nU'(R)^2
  &\ge\beta\int_0^R S^{n-1}F(U)\,dr
       -Ck\int_0^R S^{n-1}U^2\,dr\\
  &\ge\left(\frac{\beta}{p+1}-CkR^2\right)\mathcal E
  \ge\frac{\beta}{2(p+1)}\mathcal E
  \end{align*}
  for sufficiently large $B$. Together with
  \eqref{eq:radial-energy-lower} and $S(R)\le CR$, this gives
  \[
  U'(R)^2\ge cR^{-2-4/(p-1)}\longrightarrow\infty.
  \]
  Since $U'(R)<0$ by \eqref{eq:radial-flux}, the asserted
  divergence of the first-zero slope follows.

  Finally, radial energy monotonicity gives, at the radius where
  $U=s\in[0,a]$,
  \[
  \Psi_{k,B}(s)\ge U'(R)^2-2F(s)
  \ge U'(R)^2-2F(a).
  \]
  Since $F'=f>0$, this proves \eqref{eq:large-speed} uniformly on $[0,a]$.
  \end{proof}
  
  We use the barrier-family continuation framework of Andrews and Xiong
  \cite[Theorem~13]{Andrews:2019aa}, whose result is stated for compact
  manifolds. For the finite-zero profiles, Lemma~\ref{lem:strict-gap}
  supplies the strict gap needed for continuation on a complete manifold
  under the Ricci lower bound. Taking the central-value limit then
  proves the amplitude bound required for comparison with the
  threshold profile.

  \begin{proof}[Proof of the comparison in Theorem~\ref{thm:gradient-general}]
  Corollary~\ref{cor:global-bounds} gives $a=\sup_Mu\in(0,\infty)$ and $D=\|\nabla u\|_\infty<\infty$. Fix $k>\kappa$. For $B$ in the interval $I=(\max(a,A_k),\infty)$, the model has a finite transverse zero, and its inverse coordinate is defined on $[0,a]$. Denote the resulting defect by $P_B$ and set
  \[
  \mathcal B=\{B\in I:P_B\le0\text{ on }M\}.
  \]
  By Lemma~\ref{lem:large-speed}, $\Psi_{k,B}>D^2$ on $[0,a]$ for all sufficiently large $B$, so $\mathcal B\ne\varnothing$. The $C^1$ dependence of the inverse coordinates gives uniform convergence $P_B\to P_{B_0}$ on $M$ as $B\to B_0\in I$; here $u\in[0,a]$ and $|\nabla u|\le D$. Thus $\mathcal B$ is relatively closed. For $B_0\in\mathcal B$, Lemma~\ref{lem:strict-gap} gives $\sup P_{B_0}<0$, and uniform dependence makes $\mathcal B$ relatively open. Therefore $\mathcal B=I$.
  
  If $a>A_k$, the model with center value $a$ has a finite zero.
  Letting $B\downarrow a$ gives $P_a\le0$. At points $x_j$ with
  $u(x_j)\to a$, its inverse coordinate tends to one, and
  \[
  0\ge P_a(x_j)\ge-k\bigl(T_a(u(x_j))^2-1\bigr)\longrightarrow0.
  \]
  This contradicts Lemma~\ref{lem:strict-gap}. Hence
  \[
  a\le A_k,\qquad |\nabla u|^2\le\Psi_{k,B}(u)
  \quad(k>\kappa,\ B>A_k).
  \]
  Let
  $k_j\downarrow\kappa$. Upper semicontinuity gives
  \[
  a\le\liminf_j A_{k_j}\le\limsup_j A_{k_j}\le A_\kappa.
  \]
  For a fixed $B>A_\kappa$, one has $B>A_{k_j}$ for large $j$. Dependence on the curvature parameter then gives $|\nabla u|^2\le\Psi_{\kappa,B}(u)$. As $B\downarrow A_\kappa$, the hitting times and speeds converge at every fixed height $0<s<A_\kappa$, because $U_H'$ is nonzero there. This proves the gradient bound wherever $u<A_\kappa$. At a point with $u=A_\kappa$, the amplitude bound gives $\nabla u=0$, and $\Psi_H(A_\kappa)=0$.
  \end{proof}
  
  On compact Finsler spaces with nonnegative weighted Ricci curvature,
  Andrews and Xiong \cite[Theorem~21]{Andrews:2019aa} prove constancy
  under their structural assumptions when the solution attains a value
  at which the potential is stationary and maximal over the solution
  range.
  The equality case below determines the ambient hyperbolic manifold
  and the radial solution, including attainment of the model's central
  height.

  \begin{proof}[Proof of rigidity in Theorem~\ref{thm:gradient-general}]
  Use the threshold coordinate $T(U_H(r))=\cosh(\sqrt\kappa r)$ and put
  \[
  v=T(u),\qquad P=|\nabla v|^2-\kappa(v^2-1).
  \]
  Let $V=T^{-1}$, with $\alpha$ and $\mathcal L$ defined as in Lemma~\ref{lem:model-coordinate} for this $V$ and $k=\kappa$. The center extension makes $v$ a $C^3$ function even where $u=A_\kappa$. Since $u>0$, all coefficients in \eqref{eq:model-bochner} are locally bounded. Comparison gives $P\le0$, and the equality point gives $P=0$. Put $c=2(\kappa\alpha v+\alpha'|\nabla v|^2)$ and $c_+=\max\{c,0\}$.
Since $P\le0$, \eqref{eq:model-bochner} implies
$(\mathcal L-c_+)P\ge0$. Its coefficients are locally bounded and
its zeroth-order coefficient is nonpositive. The strong maximum
principle \cite[Theorem 3.5]{Gilbarg:2001aa}, followed along the
connected manifold, gives $P\equiv0$. The nonnegative terms in
\eqref{eq:model-bochner} therefore give
\begin{equation}\label{eq:concircular}
\nabla^2v=\kappa vg,\qquad |\nabla v|^2=\kappa(v^2-1).
\end{equation}
The function $v$ is nonconstant because $\Delta v=n\kappa v>0$.
Choose $p$ with $v(p)=\cosh(\sqrt\kappa r_0)>1$ and follow the
complete unit-speed geodesic initially opposite to $\nabla v$.
The two identities in \eqref{eq:concircular} give
\[
v(\gamma(t))=\cosh(\sqrt\kappa(r_0-t)).
\]
Thus $o=\gamma(r_0)$ is a critical point of $v$, with $v(o)=1$.
The classical concircular-field rigidity theorem
\cite{Tashiro:1965aa} now identifies $(M,g)$ with
$\mathbb H^n_{-\kappa}$. Along each geodesic from $o$,
$\nabla^2v=\kappa vg$ and the center data give
$v=\cosh(\sqrt\kappa d(o,\cdot))$, so $u=U_H(d(o,\cdot))$.
  \end{proof}
  
  \begin{proof}[Proof of Theorem~\ref{thm:intro-gradient}]
  Proposition~\ref{prop:mixtures} gives admissibility of $f$.
  Lemma~\ref{lem:threshold} and uniqueness of the radial initial-value
  problem identify the given profile with $U_{\kappa,A_\kappa}$.
  Its flux formula gives $U_H'<0$ for $r>0$, and the lemma gives
  $U_H(r)\to0$; thus the inverse has precisely the stated domain.
  Theorem~\ref{thm:gradient-general} gives $u\le A=A_\kappa$ and
  \eqref{eq:intro-gradient}, together with gradient equality rigidity.
  At a point where $u=A$, the amplitude bound gives
  $\nabla u=0=U_H'(0)$, so the same rigidity applies.
  \end{proof}
  
  The inverse-model distance bound has the two-point form used by
  Andrews and Xiong \cite{Andrews:2019aa}. We give its short derivation
  here to include the threshold endpoint $u=A_\kappa$.
  \begin{corollary}\label{cor:model-distance}
  Under Theorem~\ref{thm:gradient-general}, the function $r_u=U_H^{-1}(u)$, with $r_u=0$ where $u=A_\kappa$, is $1$-Lipschitz. In particular,
  \[
  |U_H^{-1}(u(x))-U_H^{-1}(u(y))|\le d_g(x,y).
  \]
  \end{corollary}
  \begin{proof}
  On $\{u<A_\kappa\}$, comparison gives $|\nabla r_u|=|\nabla u|/\sqrt{\Psi_H(u)}\le1$. For each $\varepsilon>0$, the continuous function $\max(r_u,\varepsilon)$ is constant near $\{r_u<\varepsilon\}$ and is locally Lipschitz across the level $r_u=\varepsilon$, with weak gradient bounded by one. Integrate along curves and let $\varepsilon\downarrow0$.
  \end{proof}
  
  For the explicit radial profile below at curvature $-1$, see
  \cite[Remark 3.8 and Appendix A]{Mancini:2008aa}; for related
  single-power background, see \cite{Ganguly:2015aa}.
  The proposition proves its gradient bound and rigidity
  under the Ricci lower bound.

  \begin{proposition}\label{prop:explicit-power}
  Let $a,\kappa>0$ and $n\ge2$. For $f(s)=a s^{1+2/(n-1)}$,
  \[
  A_\kappa=\left(\frac{n(n-1)\kappa}{a}\right)^{(n-1)/2},\qquad
  U_H(r)=A_\kappa\operatorname{sech}^{n-1}(\sqrt\kappa r).
  \]
  Every positive entire solution on a complete connected manifold without
  boundary satisfying \eqref{eq:ric-lower} obeys
  \begin{equation}\label{eq:explicit-gradient}
  |\nabla\log u|^2+\frac{a(n-1)}n u^{2/(n-1)}
  \le(n-1)^2\kappa,
  \end{equation}
  with the equality characterization in Theorem~\ref{thm:gradient-general}.
  \end{proposition}
  \begin{proof}
Write $k=n-1$ and $\lambda=k\sqrt\kappa$. Lemma~\ref{lem:threshold}
gives $U_H=O(e^{-cr})$ for every $c<\lambda$ and
$-U_H'/U_H\to\lambda$. Choosing $\lambda/2<c<\lambda$ shows that
$U_H\in H^1(\mathbb H^n_{-\kappa})$.
The displayed $\operatorname{sech}^k$ profile is the known positive
$H^1$ solution at this exponent
\cite[Theorem 1.3, Remark 3.8 and Appendix A]{Mancini:2008aa},
after curvature and amplitude scaling. Their uniqueness theorem at
source parameter $\lambda_{\mathrm{MS}}=0$ applies also when $n=2$.
Since both profiles have their unique maximum at the pole, uniqueness
up to hyperbolic isometries identifies them. Its phase curve is
\[
\Psi_H(s)=k^2\kappa s^2-\frac{ak}{k+1}s^{2+2/k}.
\]
Substitution in \eqref{eq:general-gradient} gives
\eqref{eq:explicit-gradient}, and the equality statement is
Theorem~\ref{thm:gradient-general}.
\end{proof}

\section{Scalar-curvature comparison for compactifications}\label{sec:boundary}
  \subsection{The source equation and boundary identities}
  Let $n\ge3$, $k=n-1$, and $\kappa>0$. Let $(\overline M,h)$ be a compact smooth manifold with nonempty boundary and connected interior $M$. We prescribe a smooth defining function $\rho$ by
  \begin{equation}\label{eq:q-class}
  q\in C^\infty([0,\infty)),\qquad
  q\ge0,\quad q'\ge0,\quad q''\le0,\quad q(s)>0\quad(s>0),
  \qquad -\Delta_h\rho=q(\rho).
  \end{equation}
  Set $g=\rho^{-2}h$ and $\gamma=h|_{T\partial M}$. Derivatives,
  inner products, and tensor norms without a subscript in this section
  are taken with respect to $h$.

  The inverse metrics and volume measures satisfy
  $g^{ij}=\rho^2h^{ij}$ and $dV_g=\rho^{-n}dV_h$. Substituting these
  into the divergence expression for the Laplacian gives
  \begin{align*}
  \Delta_g v
  &=\frac{\rho^n}{\sqrt{\det h}}
    \partial_i\!\left(\rho^{2-n}\sqrt{\det h}\,h^{ij}\partial_jv\right)\\
  &=\rho^2\Delta_hv-(n-2)\rho\langle d\rho,dv\rangle_h.
  \end{align*}
  Taking $u=\rho^k$ cancels the gradient terms:
  \begin{equation}\label{eq:exact-transform}
  -\Delta_h\rho=q(\rho)
  \quad\Longleftrightarrow\quad
  -\Delta_g u=f_q(u),\qquad
  u=\rho^k,\quad f_q(s)=k s^{1+1/k}q(s^{1/k}).
  \end{equation}
  Indeed,
  \begin{align*}
  \Delta_g(\rho^k)
  &=k\rho^{k+1}\Delta_h\rho
    +k\bigl(k-1-(n-2)\bigr)\rho^k|d\rho|_h^2\\
  &=k\rho^{k+1}\Delta_h\rho,
  \end{align*}
  because $k=n-1$. Also
  \[
  |du|_g^2=\rho^2|d(\rho^k)|_h^2
  =k^2\rho^{2k}|d\rho|_h^2.
  \]
  We verify admissibility directly, since $f_q$ need not be a finite
  sum of powers.
  
  \begin{lemma}\label{lem:q-admissible}
  For every $q$ in \eqref{eq:q-class}, $f_q$ is admissible with one component and
  \[
  \mu=\frac{k}{2(k+1)}>\frac{n-2}{2(n-1)}.
  \]
  The function $(Fw)^\mu$ is strictly concave.
  \end{lemma}
  \begin{proof}
  Write $r=s^{1/k}$ and $e=rq'(r)/q(r)$. Since $q(r)-rq'(r)\ge q(0)\ge0$, one has $0\le e\le1$. Differentiation and a change of variable give
  \[
  f_q'(s)=r\bigl((k+1)q(r)+rq'(r)\bigr),\qquad
  F(s)=k^2\int_0^r t^{2k}q(t)\,dt.
  \]
  In particular $f_q(0)=f_q'(0)=0$, with $f_q\in C^1([0,\infty))\cap C^\infty((0,\infty))$. The tangent-line inequality $q(r\tau)/q(r)\le1-e(1-\tau)$ gives
  \begin{align*}
  \frac{Ff_q'}{f_q^2}
  &=(k+1+e)\int_0^1\tau^{2k}\frac{q(r\tau)}{q(r)}\,d\tau\\
  &\le\frac{(k+1+e)(2k+2-e)}{(2k+1)(2k+2)}
  \le\frac{k+2}{2(k+1)}=1-\mu.
  \end{align*}
  The numerator of the last quadratic has derivative $k+1-2e>0$ on $[0,1]$.
  
  For large $s$, monotonicity and positivity of $q$ give
  $f_q(s)\ge c s^{1+1/k}$ and hence $F(s)\ge c_1s^{2+1/k}$.
  Therefore $F(s)^{-1/2}\le C s^{-1-1/(2k)}$ is integrable at
  infinity, which proves that the tail $w$ is finite. Put $G=Fw$
  and $\theta=\sqrt F/(f_qw)>0$. Differentiating $w'=-F^{-1/2}$ gives
  $G'=f_qw-\sqrt F$ and $G''=f_q'w-3f_q/(2\sqrt F)$. Thus
  \[
  \frac{(G^\mu)''}{\mu G^{\mu-2}f_q^2w^2}
  =\frac{Ff_q'}{f_q^2}-(1-\mu)+(\tfrac{1}{2}-2\mu)\theta
  -(1-\mu)\theta^2<0.
  \]
  Here $\mu\ge1/3$, so the terms containing $\theta$ have strictly negative sum. Also $\mu-(n-2)/(2(n-1))=1/(2k(k+1))>0$.
  With one component, the ordering and component-comparison conditions
  in Section~\ref{sec:local} hold with $C_2=1$. The formula for $f_q'$
  and $0\le e\le1$ also give
  \begin{equation}\label{eq:q-slope}
  1+\frac1k\le\frac{s f_q'(s)}{f_q(s)}\le1+\frac2k.
  \end{equation}
  \end{proof}
  
  If $C=\max_{\overline M}|d\rho|_h$, every curve $c$ in $M$ satisfies
  \[
  \left|\frac d{dt}\log\rho(c(t))\right|\le C|\dot c(t)|_g.
  \]
  Every curve approaching the boundary has infinite length, and
  compactness excludes any other escape. Thus $g$ is complete.
  By Lemma~\ref{lem:q-admissible}, Theorem~\ref{thm:gradient-general}
  applies to the positive smooth function $u=\rho^k$ on the connected
  interior whenever $\operatorname{Ric}\ge-k\kappa g$. It gives
  \[
  \rho\le A_\kappa^{1/k},\qquad
  |d\rho|_h^2\le\frac{\Psi_H(\rho^k)}{k^2\rho^{2k}}.
  \]
  Let $U_H$ be the threshold profile for $f_q$ on
  $\mathbb H^n_{-\kappa}$. In geodesic polar coordinates,
  \[
  g_{\mathbb H}=dr^2+S_\kappa(r)^2g_{\mathbb S^k},\qquad
  S_\kappa(r)=\frac{\sinh(\sqrt\kappa r)}{\sqrt\kappa}.
  \]
  The coordinate $x=2e^{-\sqrt\kappa r}$ approaches zero at infinity,
  and a direct substitution gives
  \[
  g_{\mathbb H}=\frac{1}{\kappa x^2}
  \left(dx^2+(1-x^2/4)^2g_{\mathbb S^k}\right).
  \]
  \begin{lemma}\label{lem:model-compactification}
  The function $\rho_H=U_H^{1/k}$ defines a smooth compactification $h_H=\rho_H^2g_{\mathbb H}$ of the hyperbolic ball. In the coordinate $x=2e^{-\sqrt\kappa r}$,
  \[
  U_H=x^k v(x),\qquad
  h_H=\kappa^{-1}v(x)^{2/k}
  \bigl(dx^2+(1-x^2/4)^2g_{\mathbb S^k}\bigr),
  \]
  where $v$ is smooth and $v(0)=v_0>0$. Its boundary is round, with scalar curvature
  \begin{equation}\label{eq:sharp-constant}
  R_{q,\kappa}=k(k-1)\kappa v_0^{-2/k}>0.
  \end{equation}
  \end{lemma}
  \begin{proof}
  Put $\lambda=k\sqrt\kappa$ and $p=1+1/k$. By Lemma~\ref{lem:threshold}, $U_H=O(e^{-cr})$ for every $c<\lambda$. Choose $\lambda/p<c<\lambda$. Since $f_q(s)\le Cs^p$ near zero, the flux
  \[
  L(r)=S_\kappa(r)^k(-U_H'(r))
  =\int_0^r S_\kappa(t)^k f_q(U_H(t))\,dt
  \]
  has derivative $O(e^{(\lambda-pc)r})$. Thus $L\to L_\infty\in(0,\infty)$ and
  \[
  U_H(r)=\int_r^\infty\frac{L(t)}{S_\kappa(t)^k}\,dt
  \sim\frac{L_\infty(2\sqrt\kappa)^k}{k\sqrt\kappa}e^{-k\sqrt\kappa r}.
  \]
  Hence $v=U_H/x^k$ has a finite positive limit.
  
  With $D=1-x^2/4$, substitution into the radial equation gives
  \[
  (x^{k+1}D^kv')'=x^kD^k
  \left(\frac{k^2x}{2D}v-\frac{k}{\kappa}v^{1+1/k}q(xv^{1/k})\right).
  \]
  The right side is integrable at zero. The limit of $x^{k+1}D^kv'$ there must be zero, since a nonzero limit would make $v$ diverge like $x^{-k}$. Integrating and substituting $t=x\tau$ gives
  \begin{align*}
  v'(x)=D(x)^{-k}\int_0^1\tau^kD(x\tau)^k
  \left[\frac{k^2x\tau}{2D(x\tau)}v(x\tau)
  -\frac{k}{\kappa}v(x\tau)^{1+1/k}q(x\tau v(x\tau)^{1/k})\right]d\tau.
  \end{align*}
  The right side is continuous at zero, so $v\in C^1$. Since $v(0)>0$ and $q$ is smooth, the same formula improves $v\in C^j$ to $v\in C^{j+1}$. Thus $v$ is smooth up to $x=0$, and $\rho_H=xv^{1/k}$ is a smooth defining function. At the center, $U_H$ is smooth, radial, and positive, so $h_H=U_H^{2/k}g_{\mathbb H}$ is smooth there as well. Equation~\eqref{eq:exact-transform} gives $-\Delta_{h_H}\rho_H=q(\rho_H)$.
  \end{proof}
  
  The hypothesis $\operatorname{Ric}\ge-(n-1)\kappa g$ is a
  pointwise inequality of quadratic forms. Its difference from the
  hyperbolic value and its $h$-trace are
  \begin{equation}\label{eq:ric-excess}
  \mathcal R=\operatorname{Ric}_g+(n-1)\kappa g,
  \qquad S=\operatorname{tr}_h\mathcal R
  =\rho^{-2}\bigl(R_g+n(n-1)\kappa\bigr).
  \end{equation}
  In this formula $\mathcal R$ is a covariant two-tensor, while the
  operation of taking its trace depends on the metric:
  \[
  \operatorname{tr}_g\mathcal R
  =\rho^2\operatorname{tr}_h\mathcal R=\rho^2S.
  \]
  Consequently the scalar decay $O(\rho^2)$ means that $S$ remains
  bounded, and the decay $o(\rho^2)$ means that $S$ tends to zero.
  We organize the source terms in the conformal trace by
  \begin{equation}\label{eq:IJ}
  I(s)=\int_0^s q(t)\,dt,\qquad
  J(s)=\frac{2}{ns^2}\bigl(sq(s)-I(s)\bigr).
  \end{equation}
  There is no singularity in $J$ at zero. Indeed,
  \[
  sq(s)-I(s)=\int_0^s t q'(t)\,dt,
  \qquad
  J(s)=\frac2n\int_0^1\tau q'(s\tau)\,d\tau.
  \]
  Hence $J$ is smooth at zero, with
  \[
  J(0)=q'(0)/n,\qquad J'(0)=2q''(0)/(3n).
  \]
  
  \begin{lemma}\label{lem:boundary-jets}
  Suppose
  \begin{equation}\label{eq:asymptotic-big}
  \mathcal R\ge0,\qquad R_g+n(n-1)\kappa=O(\rho^2).
  \end{equation}
  Then $\mathcal R$ and $S$ extend smoothly to $\overline M$, and on the boundary
  \[
  |d\rho|^2=\kappa,\qquad
  \nabla^2\rho=-\frac{q(0)}n h,\qquad
  \mathcal R(\nabla\rho,\cdot)=0.
  \]
  The smooth function
  \begin{equation}\label{eq:T-def}
  T=\frac{R_h-S}{n(n-1)}-J(\rho)
  =\frac{|d\rho|^2-\kappa+2I(\rho)/n}{\rho^2}
  \end{equation}
  has the boundary identity
  \begin{equation}\label{eq:boundary-scalar}
  R_\gamma-S|_{\partial M}
  =(n-2)(n+1)T|_{\partial M}
  +(n-1)(n-2)\left(\frac{q'(0)}n+\frac{q(0)^2}{n^2\kappa}\right).
  \end{equation}
  If the scalar error is $o(\rho^2)$, then $\mathcal R|_{\partial M}=0$.
  \end{lemma}
  \begin{proof}
For $g=\rho^{-2}h$, the standard conformal Ricci formula
(see, for example, \cite{Besse:2008aa}) is
\begin{equation}\label{eq:conformal-ricci}
\operatorname{Ric}_g=\operatorname{Ric}_h
+(n-2)\rho^{-1}\nabla^2\rho
+\bigl(\rho^{-1}\Delta\rho-(n-1)\rho^{-2}|d\rho|^2\bigr)h.
\end{equation}
  Thus $\rho^2\mathcal R$ is smooth. To use scalar decay to control
  the whole tensor, diagonalize $\mathcal R$ in an $h$-orthonormal
  frame in the interior. Its eigenvalues $\lambda_i$ are nonnegative,
  so
  \[
  |\mathcal R|_h=(\textstyle\sum_i\lambda_i^2)^{1/2}
  \le\sum_i\lambda_i=S.
  \]
  The decay assumption bounds $S$. In coordinates $(\rho,y)$, every
  component $F$ of $\rho^2\mathcal R$ is therefore smooth and
  $O(\rho^2)$. Thus $F(0,y)=\partial_\rho F(0,y)=0$, and Taylor's
  formula with integral remainder gives
  \[
  F(\rho,y)=\rho^2\int_0^1(1-t)
  \partial_\rho^2F(t\rho,y)\,dt.
  \]
  This proves the smooth divisibility by $\rho^2$ and hence the smooth
  extension of $\mathcal R$ and $S$. Under little-oh decay, $S\to0$;
  the same eigenvalue inequality gives $\mathcal R|_{\partial M}=0$.
  
  Let $E=\operatorname{Ric}_h-(R_h/n)h$ be the trace-free Ricci
  tensor of $h$. Taking the $h$-trace of
  \eqref{eq:conformal-ricci}, and using $\Delta\rho=-q(\rho)$, gives
  \[
  S=R_h-\frac{2(n-1)q(\rho)}\rho
    -\frac{n(n-1)}{\rho^2}(|d\rho|^2-\kappa).
  \]
  The trace-free part removes every multiple of $h$; since
  $(\nabla^2\rho)^\circ=\nabla^2\rho+(q(\rho)/n)h$, the two
  identities become
  \begin{align}
  \frac{R_h-S}{n(n-1)}
  &=\frac{|d\rho|^2-\kappa}{\rho^2}+\frac{2q(\rho)}{n\rho},\label{eq:conformal-trace}\\
  \mathcal A:=\nabla^2\rho+\frac{q(\rho)}n h
  &=\frac{\rho}{n-2}(\mathcal R^\circ-E).\label{eq:A-def}
  \end{align}
  Here the circle denotes the trace-free part with respect to $h$.
  Multiplying the first equation by $\rho^2$ and taking its boundary
  value gives $|d\rho|^2=\kappa$. The smoothness of the tensors in
  the second equation gives
  $\nabla^2\rho=-q(0)h/n$ there. Subtracting $J(\rho)$ in
  \eqref{eq:conformal-trace} proves \eqref{eq:T-def} and shows that
  $T$ extends smoothly, despite its quotient expression.
  To obtain its normal derivative relation, differentiate
  \[
  |d\rho|^2=\kappa-\frac2n I(\rho)+\rho^2T.
  \]
  Because $d|d\rho|^2=2\nabla^2\rho(\nabla\rho,\cdot)$, this gives
  \begin{equation}\label{eq:A-normal}
  \mathcal A(\nabla\rho,\cdot)=\rho T\,d\rho+\frac{\rho^2}{2}dT.
  \end{equation}
  
  The connection difference for $g=\rho^{-2}h$ is
  \[
  \Gamma(g)^k_{ij}-\Gamma(h)^k_{ij}
  =-\rho^{-1}\bigl(\delta_i^k\rho_j+\delta_j^k\rho_i
  -h_{ij}h^{k\ell}\rho_\ell\bigr).
  \]
  Contracting the covariant derivative of a symmetric covariant tensor
  with $g^{ij}=\rho^2h^{ij}$ gives the first equality below.
  Since $\operatorname{tr}_g\mathcal R=\rho^2S$ and the added
  constant multiple of $g$ is parallel, contracted Bianchi gives the
  last equality:
  
  \[
  \operatorname{div}_g\mathcal R
  =\rho^2\operatorname{div}_h\mathcal R
  -(n-2)\rho\mathcal R(\nabla\rho,\cdot)+\rho S\,d\rho
  =\tfrac12d(\rho^2S).
  \]
  Cancel the terms $\rho S\,d\rho$ and divide by $\rho$:
  \[
  (n-2)\mathcal R(\nabla\rho,\cdot)
  =\rho\left(\operatorname{div}_h\mathcal R-\tfrac12dS\right).
  \]
  Hence $\mathcal R(\nabla\rho,\cdot)=0$ on $\partial M$.
  
  We use the convention
  $\mathrm{II}(X,Y)=h(\nabla_X\nu,Y)$ for tangent vectors and
  $H=\operatorname{tr}_\gamma\mathrm{II}$. Since $\rho$ is positive
  inside, $\nu=-\nabla\rho/\sqrt\kappa$ is the outward unit normal.
  Consequently
  \[
  \mathrm{II}(X,Y)=-\frac{\nabla^2\rho(X,Y)}{\sqrt\kappa}
  =\frac{q(0)}{n\sqrt\kappa}\gamma(X,Y),\qquad
  H=\frac{(n-1)q(0)}{n\sqrt\kappa}.
  \]
  On a boundary collar, extend the outward normal by
  $\nu=-\nabla\rho/|d\rho|$. Evaluating \eqref{eq:A-normal}
  on $\nabla\rho$ and dividing by $\rho|d\rho|^2$ gives
  \[
  \frac{\mathcal A(\nu,\nu)}{\rho}
  =T+\frac{\rho}{2|d\rho|^2}\langle d\rho,dT\rangle
  \longrightarrow T\quad\text{at }\partial M.
  \]
  On the other hand, \eqref{eq:A-def} and
  $\mathcal R(\nu,\nu)=0$ give
  \[
  T=\frac{-S/n-E(\nu,\nu)}{n-2},\qquad
  E(\nu,\nu)=-\frac Sn-(n-2)T.
  \]
  The contracted Gauss equation gives
  \[
  R_\gamma=R_h-2\operatorname{Ric}_h(\nu,\nu)
  +H^2-|\mathrm{II}|_\gamma^2.
  \]
  Inserting $\operatorname{Ric}_h(\nu,\nu)=R_h/n+E(\nu,\nu)$ and
  $H^2-|\mathrm{II}|^2=(n-1)(n-2)q(0)^2/(n^2\kappa)$ yields
  \begin{align*}
  R_\gamma
  &=\frac{n-2}{n}R_h-2E(\nu,\nu)
  +\frac{(n-1)(n-2)q(0)^2}{n^2\kappa}\\
  &=(n-2)(n+1)T+S
  +(n-1)(n-2)\left(\frac{q'(0)}n+\frac{q(0)^2}{n^2\kappa}\right),
  \end{align*}
  as claimed.
  \end{proof}
  \begin{corollary}\label{cor:decay-geometry}
  Under the source equation and $\mathcal R\ge0$, the condition
  $R_g+n(n-1)\kappa=O(\rho^2)$ is equivalent to the boundary normalization
  \begin{equation}\label{eq:boundary-normalization}
  |d\rho|_h^2=\kappa,\qquad
  H=\frac{(n-1)q(0)}{n\sqrt\kappa}\quad\text{on }\partial M,
  \end{equation}
  where $H$ is the trace of the second fundamental form for the outward normal. If only the first equality is imposed, then
  $H\le(n-1)q(0)/(n\sqrt\kappa)$. Equality everywhere forces the boundary to be umbilic and gives \eqref{eq:asymptotic-big}.
  \end{corollary}
  \begin{proof}
  One implication follows from Lemma~\ref{lem:boundary-jets}. For the converse suppose $Z=|d\rho|^2=\kappa$ on the boundary and put $a=\nabla^2\rho(\nu,\nu)$ there. The tensor $\mathcal B=\rho\mathcal R$ is smooth up to the boundary, since $Z-\kappa$ is divisible by $\rho$. Its boundary value, by \eqref{eq:conformal-ricci}, is
  \[
  \mathcal B_0=(n-2)\nabla^2\rho-\bigl(q(0)+2(n-1)a\bigr)h.
  \]
  Here $\partial_\rho Z=2a$, using the transverse field $\nabla\rho/Z$. Nonnegativity gives $\mathcal B_0\ge0$, while
  \[
  \mathcal B_0(\nu,\nu)=-na-q(0),\qquad
  \operatorname{tr}_h\mathcal B_0=2(n-1)(-na-q(0)),
  \qquad H=\frac{q(0)+a}{\sqrt\kappa}.
  \]
  This proves the mean-curvature inequality. Equality in \eqref{eq:boundary-normalization} gives $a=-q(0)/n$, hence $\operatorname{tr}_h\mathcal B_0=0$. A nonnegative tensor with zero trace vanishes. Thus $\mathcal B$ is divisible by $\rho$, so $\mathcal R$ is smooth and bounded and its scalar trace with respect to $g$ is $O(\rho^2)$. The formula for $\mathcal B_0$ also gives $\nabla^2\rho=-q(0)h/n$ at the boundary.
  \end{proof}
  
  \begin{lemma}\label{lem:defect-equation}
  In the interior, put $Z=|d\rho|^2$ and let $\mathcal A$ be as in \eqref{eq:A-def}. Then
  \begin{equation}\label{eq:T-bochner}
  \begin{split}
  \Delta T+\frac{n+2}{\rho}\langle d\rho,dT\rangle
  -\frac{2q(\rho)}\rho T
  ={}&\frac2{\rho^2}|\mathcal A|^2
  +\frac2{\rho^2}\mathcal R(\nabla\rho,\nabla\rho)
  -\frac{n-1}{\rho}J'(\rho)Z.
  \end{split}
  \end{equation}
  The identity in the interior does not require a sign or decay assumption on $\mathcal R$.
  \end{lemma}
  \begin{proof}
  The scalar Bochner identity is
  \[
  \frac12\Delta|d\rho|^2
  =|\nabla^2\rho|^2+\langle d\rho,d\Delta\rho\rangle
  +\operatorname{Ric}_h(\nabla\rho,\nabla\rho).
  \]
  The source equation gives
  $\langle d\rho,d\Delta\rho\rangle=-q'Z$, while trace-freeness of
  $\mathcal A$ gives $|\nabla^2\rho|^2=|\mathcal A|^2+q^2/n$.
  Solving \eqref{eq:conformal-ricci} for $\operatorname{Ric}_h$ gives
  \[
  \operatorname{Ric}_h
  =\mathcal R-\frac{n-2}{\rho}\nabla^2\rho
  +\left(\frac q\rho+\frac{(n-1)(Z-\kappa)}{\rho^2}\right)h.
  \]
  Finally,
  $\nabla^2\rho(\nabla\rho,\nabla\rho)
  =\langle d\rho,dZ\rangle/2$. Combining these identities gives
  \begin{align*}
  \frac{1}{2}\Delta Z
  ={}&|\mathcal A|^2+\frac{q^2}{n}-q'Z
  +\mathcal R(\nabla\rho,\nabla\rho)
  -\frac{n-2}{2\rho}\langle d\rho,dZ\rangle\\
  &+\left(\frac q\rho+\frac{(n-1)(Z-\kappa)}{\rho^2}\right)Z.
  \end{align*}
  For the other side of the calculation, use
  $Z=\kappa-2I(\rho)/n+\rho^2T$ and the chain and product rules:
  \begin{align*}
  \Delta I(\rho)&=-q^2+q'Z,\\
  \Delta(\rho^2T)&=\rho^2\Delta T
    +4\rho\langle d\rho,dT\rangle+(2Z-2\rho q)T,\\
  \langle d\rho,dZ\rangle
  &=-\frac{2q}{n}Z+2\rho TZ+\rho^2\langle d\rho,dT\rangle.
  \end{align*}
  Hence
  \begin{align*}
  \Delta Z+\frac{n-2}{\rho}\langle d\rho,dZ\rangle
  ={}&\rho^2\Delta T+(n+2)\rho\langle d\rho,dT\rangle
  -2\rho qT+\frac{2q^2}{n}\\
  &+\left(2(n-1)T-\frac{2q'}n-\frac{2(n-2)q}{n\rho}\right)Z.
  \end{align*}
  Double the first Bochner expression and move its first-order term
  to the left. Equating it with the last expression cancels $2q^2/n$
  and $2(n-1)TZ$. The coefficient of the remaining $Z$ term, after
  division by $\rho^2$, is
  \[
  -\frac{2(n-1)}{n\rho^2}
  \left(q'-\frac{2q}\rho+\frac{2I}{\rho^2}\right)
  =-\frac{n-1}\rho J'(\rho),
  \]
  which proves \eqref{eq:T-bochner}.
  \end{proof}
  
  \subsection{Scalar-curvature comparison and rigidity}
  Let $\rho_* = A_\kappa^{1/k}$ for the transformed source, and define
  on $0<s\le\rho_*$
  \begin{equation}\label{eq:model-boundary-functions}
  P_H(s)=\frac{\Psi_H(s^k)}{k^2s^{2k}},\qquad
  T_H(s)=\frac{P_H(s)-\kappa+2I(s)/n}{s^2}.
  \end{equation}
  For the hyperbolic model, Lemma~\ref{lem:boundary-jets} applies with $\mathcal R=0$ and gives
  \[
  P_H(\rho_H)=|d\rho_H|_{h_H}^2,\qquad
  T_H(\rho_H)=\frac{R_{h_H}}{n(n-1)}-J(\rho_H).
  \]
  Both right sides are smooth in the boundary coordinate $x$. Since $\rho_H=xv^{1/k}$ with $v(0)>0$, the inverse function theorem shows that $P_H$ and $T_H$ are smooth at $s=0$. They are smooth at $\rho_*$ by Lemma~\ref{lem:phase-center}, since $f_q$ is smooth near $A_\kappa>0$. The model scalar curvature at height $s$ is
  \[
  \mathcal R_H(s)=n(n-1)(T_H(s)+J(s)).
  \]
  Thus $R_{h_H}=\mathcal R_H(\rho_H)$.
  
  \begin{theorem}\label{thm:boundary-general}
  Suppose \eqref{eq:q-class} and \eqref{eq:asymptotic-big} hold. Then
  \begin{align}
  R_h-S&\le\mathcal R_H(\rho)\quad\text{in }M,\label{eq:interior-scalar}\\
  R_\gamma-S|_{\partial M}&\le R_{q,\kappa}\quad\text{on }\partial M.\label{eq:boundary-corrected}
  \end{align}
  Equality at one point in either comparison implies that $(\overline M,h,\rho)$ is isometric to $(\overline{\mathbb H^n},h_H,\rho_H)$. In particular, $\mathcal R\equiv0$ and $\partial M$ is connected and round. If the scalar error is $o(\rho^2)$, the boundary comparison is $R_\gamma\le R_{q,\kappa}$. The hyperbolic model attains the comparisons for the prescribed source $q$.
  If $q$ is not affine on $[0,\max\rho_H]$, then
  \begin{equation}\label{eq:intro-strict}
  R_{q,\kappa}<(n-1)(n-2)
  \left(\frac{q'(0)}n+\frac{q(0)^2}{n^2\kappa}\right).
  \end{equation}
  \end{theorem}

The correction in \eqref{eq:boundary-corrected} is necessary. For $q(s)=s$,
consider
\[
\overline M=[-\pi/2,\pi/2]\times\mathbb S^{n-1},\qquad
h=dt^2+g_{\mathbb S^{n-1}},\qquad \rho=\cos t.
\]
Then $-\Delta_h\rho=\rho$. In the coordinate
$\tau=\operatorname{arsinh}(\tan t)$,
\[
g=\rho^{-2}h=d\tau^2+\cosh^2\tau\,g_{\mathbb S^{n-1}},\qquad
\operatorname{Ric}_g+(n-1)g=2(n-2)g_{\mathbb S^{n-1}}\ge0.
\]
This metric is complete, and
\[
R_g+n(n-1)=2(n-1)(n-2)\rho^2,\qquad
R_\gamma=(n-1)(n-2)>R_{q,1}=\frac{(n-1)(n-2)}n,
\]
where the model value is computed in Corollary~\ref{cor:affine-model}.
The corrected quantity is
$R_\gamma-S|_{\partial M}=-(n-1)(n-2)$, in agreement with
\eqref{eq:boundary-corrected}. Thus the order-two scalar term cannot
be discarded under the big-$O$ hypothesis.

  The following argument adapts the classical Hopf barrier
\cite[Chapter~3]{Gilbarg:2001aa} to the unbounded normal drift
$(n+2)\rho^{-1}\nabla\rho$. We include the barrier estimates because
this coefficient is singular at the boundary.

\begin{lemma}\label{lem:boundary-minimum}
  Let $h$ be smooth on a boundary collar and let $\rho$ be a smooth defining function with $|d\rho|^2=\kappa>0$ on the boundary. Suppose $w$ is smooth and nonnegative on the closed collar and
  \begin{equation}\label{eq:singular-operator}
  \mathcal Lw:=\Delta w+\frac{n+2}\rho\langle d\rho,dw\rangle
  +\langle b,dw\rangle+\left(-\frac{2a}\rho+d\right)w\le0
  \end{equation}
  in its interior, where $a\ge0$ is constant and $b,d$ are bounded. If $w(p)=0$ at a boundary point, then $w$ vanishes on a neighborhood of $p$.
  \end{lemma}
  \begin{proof}
  Multiplying \eqref{eq:singular-operator} by $\rho$ and taking the
  limit at $p$ gives
  \[
  (n+2)\langle d\rho,dw\rangle(p)\le0.
  \]
  Since $w\ge0$ and $w(p)=0$, its inward normal derivative is nonnegative. Thus this derivative is zero.
  
  Replace $d$ by $\min(d,0)$; the inequality remains valid and the zero-order coefficient is nonpositive. On a connected collar, the interior strong minimum principle gives either $w\equiv0$ or $w>0$ in the interior. In the latter case, extend $h$ smoothly across the boundary and choose a small geodesic ball $B_R(c)$ contained in the collar and internally tangent to its boundary at $p$, with no other boundary contact. The distance $r=d_h(c,\cdot)$ is smooth on $R/2\le r\le R$. Set
  \[
  \phi=e^{-\alpha r^2}-e^{-\alpha R^2}.
  \]
  At $p$, $\langle d\rho,d(r^2)\rangle=-2R\sqrt\kappa<0$. Therefore
  \[
  \frac{\langle d\rho,d(r^2)\rangle}{\rho}\le C
  \]
  on the open annulus: the numerator is negative close to $p$, and $\rho$ is bounded away from zero on the rest. Since the ball lies inside $M$, its distance to the sphere satisfies $R-r\le s$, where $s$ is distance to $\partial M$. Also $\rho\ge c_0s$ near $p$, so
  \[
  \frac{R^2-r^2}{\rho}\le C,
  \qquad
  \frac{1-e^{-\alpha(R^2-r^2)}}\rho\le C\alpha.
  \]
  Direct differentiation gives
  \begin{align*}
  e^{\alpha r^2}\mathcal L\phi
  ={}&\alpha^2|d(r^2)|^2
  -\alpha\left(\Delta(r^2)+\frac{n+2}{\rho}\langle d\rho,d(r^2)\rangle
  +\langle b,d(r^2)\rangle\right)\\
  &+\left(-\frac{2a}{\rho}+d\right)
  \left(1-e^{-\alpha(R^2-r^2)}\right).
  \end{align*}
  The preceding bounds and $|d(r^2)|^2=4r^2\ge R^2$ on the annulus give
  \[
  \mathcal L\phi\ge e^{-\alpha r^2}
  \bigl(\alpha^2R^2-C\alpha-C\bigr)>0
  \]
  for sufficiently large $\alpha$.
  
  Choose $\varepsilon>0$ such that $w\ge\varepsilon\phi$ on $\{r=R/2\}$. On $\{r=R\}$, $\phi=0$. At an interior negative minimum of $w-\varepsilon\phi$, its Laplacian and zero-order term would be nonnegative, contradicting $\mathcal L(w-\varepsilon\phi)<0$. Thus $w\ge\varepsilon\phi$ throughout the annulus. At $p$, this gives a positive inward normal derivative, contradicting the first paragraph. Therefore $w\equiv0$ near $p$.
  \end{proof}
  
  \begin{proof}[Proof of comparison and rigidity in Theorem~\ref{thm:boundary-general}]
  Theorem~\ref{thm:gradient-general} and \eqref{eq:exact-transform}
  give $\rho\le\rho_*$ and
  \begin{equation}\label{eq:W-def}
  W:=T_H(\rho)-T
  =\frac{P_H(\rho)-|d\rho|_h^2}{\rho^2}
  =\frac{\Psi_H(u)-|du|_g^2}{k^2\rho^{2k+2}}\ge0.
  \end{equation}
  By Lemma~\ref{lem:boundary-jets}, the first expression for $W$ is
  smooth on $\overline M$.
  The trace relation gives the exact interior identity
  \begin{equation}\label{eq:interior-deficit}
  \mathcal R_H(\rho)-(R_h-S)=n(n-1)W.
  \end{equation}
  Subtracting the boundary formula in Lemma~\ref{lem:boundary-jets} from its model version gives
  \begin{equation}\label{eq:boundary-deficit}
  R_{q,\kappa}-(R_\gamma-S|_{\partial M})
  =(n-2)(n+1)W|_{\partial M}.
  \end{equation}
  The inequalities follow from $W\ge0$. At an interior equality point, \eqref{eq:W-def} gives equality in Theorem~\ref{thm:gradient-general}, including when the point is critical.
  
  Suppose equality holds at a boundary point. By \eqref{eq:boundary-deficit}, $W$ vanishes there. Restrict to a collar where $Z=|d\rho|^2>0$, and write
  \[
  B=T_H(\rho)+\frac\rho2T_H'(\rho),\qquad
  C=\frac{2n}{n-1},\qquad
  A=T_H''(\rho)+\frac{n+2}\rho T_H'(\rho)+\frac{n-1}\rho J'(\rho).
  \]
  Here $B$ is smooth and $\rho^2A$ is bounded and smooth. For $\eta=\nabla\rho/\sqrt Z$, \eqref{eq:A-normal} gives
  \[
  \frac{\mathcal A(\eta,\eta)}\rho
  =B-W-\frac\rho{2Z}\langle d\rho,dW\rangle.
  \]
  To retain a sign on subtracting the model equation, choose an
  $h$-orthonormal frame $e_1=\eta,e_2,\ldots,e_n$. Since
  $\operatorname{tr}_h\mathcal A=0$, Cauchy's inequality gives
  \begin{align*}
  |\mathcal A|^2
  &=\mathcal A_{11}^2+2\sum_{\alpha=2}^n\mathcal A_{1\alpha}^2
    +\sum_{\alpha,\beta=2}^n\mathcal A_{\alpha\beta}^2\\
  &\ge\mathcal A_{11}^2+
    \frac1{n-1}\left(\sum_{\alpha=2}^n\mathcal A_{\alpha\alpha}\right)^2
  =\frac n{n-1}\mathcal A_{11}^2.
  \end{align*}
  In the radial model the mixed entries vanish and the tangential
  entries are equal, so equality holds. For a scalar function
  $F(\rho_H)$, the chain rule reads
  \[
  \Delta_{h_H}F(\rho_H)
  =P_H(\rho_H)F''(\rho_H)-q(\rho_H)F'(\rho_H).
  \]
  The model version of \eqref{eq:T-bochner} is consequently
  \[
  P_H(s)\left(T_H''+\frac{n+2}sT_H'+\frac{n-1}sJ'\right)
  -q(s)T_H'-\frac{2q(s)}sT_H
  =C\left(T_H+\frac s2T_H'\right)^2.
  \]
  By the chain rule,
  \begin{align*}
  \Delta(T_H(\rho))&+\frac{n+2}{\rho}\langle d\rho,d(T_H(\rho))\rangle
  -\frac{2q(\rho)}{\rho}T_H(\rho)\\
  &=AZ-q(\rho)T_H'
  -\frac{2q(\rho)}\rho T_H-\frac{n-1}\rho J'Z.
  \end{align*}
  Subtracting \eqref{eq:T-bochner}, inserting the model equation, and
  using $Z=P_H(\rho)-\rho^2W$ gives
  \begin{align*}
  \Delta W&+\frac{n+2}{\rho}\langle d\rho,dW\rangle
  -\frac{2q(\rho)}{\rho}W\\
  &=-\rho^2AW+CB^2-\frac2{\rho^2}|\mathcal A|^2
  -\frac2{\rho^2}\mathcal R(\nabla\rho,\nabla\rho).
  \end{align*}
  In the square term write
  \begin{align*}
  \frac2{\rho^2}|\mathcal A|^2
  ={}&\frac2{\rho^2}
  \left(|\mathcal A|^2-\frac n{n-1}\mathcal A(\eta,\eta)^2\right)\\
  &+C\left(B-W-\frac{\rho}{2Z}\langle d\rho,dW\rangle\right)^2.
  \end{align*}
  The identity
  \begin{align*}
  B^2-&\left(B-W-\frac{\rho}{2Z}\langle d\rho,dW\rangle\right)^2\\
  &=2B\left(W+\frac{\rho}{2Z}\langle d\rho,dW\rangle\right)
  -\left(W+\frac{\rho}{2Z}\langle d\rho,dW\rangle\right)^2
  \end{align*}
  then yields
  \begin{align}
  \Delta W&+\left(\frac{n+2}\rho-\frac{CB\rho}{Z}\right)\langle d\rho,dW\rangle
  +\left(\rho^2A-\frac{2q(\rho)}\rho-2CB\right)W\notag\\
  &=-\frac2{\rho^2}\left(|\mathcal A|^2-\frac n{n-1}\mathcal A(\eta,\eta)^2\right)\notag\\
  &\quad-C\left(W+\frac{\rho}{2Z}\langle d\rho,dW\rangle\right)^2
  -\frac2{\rho^2}\mathcal R(\nabla\rho,\nabla\rho)\le0.
  \label{eq:W-identity}
  \end{align}
  Lemma~\ref{lem:boundary-minimum} applies with
  \[
  a=q(0),\quad b=-\frac{CB\rho}{Z}\nabla\rho,\qquad
  d=\rho^2A-2CB-\frac{2(q(\rho)-q(0))}\rho.
  \]
  Indeed $Z\to\kappa>0$, while $B$, $T_H'$, $T_H''$ and $J'$ are
  smooth; explicitly,
  $\rho^2A=\rho^2T_H''+(n+2)\rho T_H'+(n-1)\rho J'$.
  Also $(q(\rho)-q(0))/\rho$ extends smoothly. Thus $b$ and $d$
  are bounded, as required by the lemma. Hence $W$ vanishes in an
  interior neighborhood of the boundary point. Equation~\eqref{eq:W-def} gives equality in the interior gradient comparison, and Theorem~\ref{thm:gradient-general} identifies $g$ and $u$ with the hyperbolic model.
  
The resulting interior isometry $\mathcal I$ preserves both $\rho$ and
$h=\rho^2g$. It therefore intertwines the smooth transverse fields
$X=\nabla_h\rho/|d\rho|_h^2$ and
$X_H=\nabla_{h_H}\rho_H/|d\rho_H|_{h_H}^2$ on small boundary collars.
Starting at a positive regular level, their flow coordinates extend
$\mathcal I$ and its inverse smoothly to height zero. The identities
$\mathcal I^*h_H=h$ and $\mathcal I^*\rho_H=\rho$ extend by continuity,
so the compactified triples are isometric.
  \end{proof}
  \begin{lemma}\label{lem:auxiliary}
  Under the hypotheses of Theorem~\ref{thm:boundary-general}, $T\le0$ on $\overline M$ and
  \begin{equation}\label{eq:auxiliary}
  R_\gamma-S|_{\partial M}
  \le(n-1)(n-2)\left(\frac{q'(0)}n+\frac{q(0)^2}{n^2\kappa}\right).
  \end{equation}
  If equality holds at one boundary point, the compactification is the hyperbolic model, $T\equiv0$, $\operatorname{Ric}_h^\circ=0$, and $q$ is affine on $[0,\max\rho]$.
  \end{lemma}
  \begin{proof}
  Integration by parts in the definition of $J$ gives
  \[
  J'(s)=\frac{2}{ns^3}\int_0^s t^2q''(t)\,dt\le0.
  \]
  Every term on the right side of \eqref{eq:T-bochner} is therefore
  nonnegative. Its left side has divergence form
  \[
  \rho^{-n-2}\operatorname{div}_h(\rho^{n+2}\nabla T)
  -\frac{2q(\rho)}\rho T.
  \]
  Set $T_+=\max\{T,0\}$. Multiplying by $\rho^{n+2}T_+$ and
  integrating by parts gives
  \begin{align*}
  &-\int_M\rho^{n+2}|dT_+|^2\,dV_h
  -2\int_Mq(\rho)\rho^{n+1}T_+^2\,dV_h\\
  &\qquad=\int_M\rho^{n+2}T_+
  \left(\frac{2|\mathcal A|^2}{\rho^2}
  +\frac{2\mathcal R(\nabla\rho,\nabla\rho)}{\rho^2}
  -\frac{n-1}{\rho}J'Z\right)dV_h\ge0.
  \end{align*}
  To justify the identity, first integrate on $M_\varepsilon=
  \{\rho>\varepsilon\}$. Its outward normal is
  $-\nabla\rho/|d\rho|$, and the boundary term is
  \[
  \int_{\partial M_\varepsilon}
  \varepsilon^{n+2}T_+\partial_\nu T\,dA_h.
  \]
  The level areas are bounded near zero, and $T,dT$ are bounded,
  so this term is $O(\varepsilon^{n+2})$. The possibly singular
  integrands on the right are integrable after weighting:
  $\mathcal A=O(\rho)$, $\mathcal R$ and $J'$ are bounded, and
  the weight supplies at least $\rho^n$. Passing to the limit proves
  the displayed identity. Its left side is nonpositive and its right
  side is nonnegative, so both vanish. Since $q(\rho)>0$ in the
  interior, the second integral forces $T_+=0$. Substituting
  $T|_{\partial M}\le0$ into \eqref{eq:boundary-scalar} proves
  \eqref{eq:auxiliary}.
  
  Applying the bound to the hyperbolic model gives the same upper bound for $R_{q,\kappa}$. Equality in \eqref{eq:auxiliary} therefore forces equality in Theorem~\ref{thm:boundary-general} and identifies the compactification with the hyperbolic model. The boundary formula also gives $T=0$ at that point. The function $-T\ge0$ satisfies \eqref{eq:singular-operator}, with $a=q(0)$, $b=0$, and $d=-2(q(\rho)-q(0))/\rho$. Lemma~\ref{lem:boundary-minimum} and the interior strong minimum principle give $T\equiv0$. Equation~\eqref{eq:T-bochner} then yields $\mathcal A\equiv0$ and $J'(\rho)Z\equiv0$. Since $\mathcal R=0$, \eqref{eq:A-def} gives $\operatorname{Ric}_h^\circ=0$.
  
  Since $T=0$, $Z=\kappa-2I(\rho)/n$. At a maximum point of
  $\rho$, $Z=0$, and strict positivity of $q$ implies
  $Z>0$ whenever $0<\rho<\max_{\overline M}\rho$. Thus
  $J'=0$ on $(0,\max_{\overline M}\rho)$. Writing the constant value
  of $J$ as $J_0$, its definition gives
  $sq(s)-I(s)=(n/2)J_0s^2$. Differentiation yields
  $q'(s)=nJ_0$, proving the affine assertion.
  Applied to the model, this equality characterization gives
  \eqref{eq:intro-strict} for every $q$ that is not affine on
  $[0,\max\rho_H]$ and completes Theorem~\ref{thm:boundary-general}.
  \end{proof}
  
  \begin{proof}[Proof of Theorem~\ref{thm:intro-boundary}]
  Lemma~\ref{lem:model-compactification} gives the model
  $(h_H,\rho_H)$ used in the statement. The scalar decay gives
  $S|_{\partial M}=0$ by Lemma~\ref{lem:boundary-jets}.
  The comparison and rigidity follow from
  Theorem~\ref{thm:boundary-general}.
  \end{proof}
  
  For $q(s)=a+bs$, the model is computed directly below from the
  prescribed-source equation. The endpoint cases are the round
  hemisphere and the Euclidean ball, while $a,b>0$ gives the corresponding
  spherical-cap interpolation. The gradient estimate on general complete
  manifolds then follows from Theorem~\ref{thm:gradient-general}.
  
  \begin{samepage}
  \begin{corollary}\label{cor:affine-model}
  Let $q(s)=a+bs$, where $a,b\ge0$ and $a+b>0$. Put
  \[
  C_0=\frac{n\kappa}{\sqrt{a^2+n\kappa b}},\qquad
  D_0=\frac{a}{\sqrt{a^2+n\kappa b}}.
  \]
  Then
  \begin{equation}\label{eq:affine-model}
  \rho_H(r)=\frac{C_0}{\cosh(\sqrt\kappa r)+D_0},\qquad
  U_H=\rho_H^k,\qquad
  A_\kappa=\left(\frac{C_0}{1+D_0}\right)^k,
  \end{equation}
  and
  \begin{equation}\label{eq:affine-constant}
  R_{q,\kappa}=k(k-1)\left(\frac bn+\frac{a^2}{n^2\kappa}\right).
  \end{equation}
  For every positive entire solution of
  \[
  -\Delta u=k\bigl(a u^{1+1/k}+b u^{1+2/k}\bigr)
  \]
  on a complete connected manifold without boundary with $\operatorname{Ric}\ge-k\kappa g$, one has
  \begin{equation}\label{eq:affine-gradient}
  |d\log u|_g^2+\frac{k^2}{n}
  \bigl(2a u^{1/k}+b u^{2/k}\bigr)\le k^2\kappa.
  \end{equation}
  Equality at one point has the rigidity in Theorem~\ref{thm:gradient-general}.
  \end{corollary}
  \end{samepage}
  \begin{proof}
  Define $\widehat\rho$ by the expression in \eqref{eq:affine-model}. For $z=\cosh(\sqrt\kappa r)$, direct differentiation gives
  \[
  -\Delta_{g_{\mathbb H}}(\widehat\rho^k)
  =\frac{kn\kappa C_0^k(1+D_0z)}{(z+D_0)^{k+2}}
  =k\widehat\rho^{k+1}(a+b\widehat\rho).
  \]
  Thus $\widehat h=\widehat\rho^2g_{\mathbb H}$ has the prescribed source. For $x=2e^{-\sqrt\kappa r}$,
  $\widehat\rho=C_0x/(1+D_0x+x^2/4)$. Its boundary metric is $\kappa^{-1}C_0^2g_{\mathbb S^k}$. Applying the sharp comparison to this compactification and Lemma~\ref{lem:auxiliary} to the threshold model gives
  \[
  k(k-1)\left(\frac bn+\frac{a^2}{n^2\kappa}\right)
  \le R_{q,\kappa}
  \le k(k-1)\left(\frac bn+\frac{a^2}{n^2\kappa}\right).
  \]
  Equality rigidity identifies $\widehat\rho$ with $\rho_H$; their unique centers agree. This proves \eqref{eq:affine-model} and \eqref{eq:affine-constant}.
  
  A direct differentiation, or $T_H=0$, gives
  \[
  P_H(s)=\kappa-\frac{2a}{n}s-\frac bn s^2,
  \qquad
  \Psi_H(u)=k^2u^2\left(\kappa-\frac{2a}{n}u^{1/k}-\frac bn u^{2/k}\right).
  \]
  Apply Theorem~\ref{thm:gradient-general} to obtain \eqref{eq:affine-gradient} and its equality statement.
  
Equality in Lemma~\ref{lem:auxiliary} gives $T_H=0$ and
$\operatorname{Ric}_{h_H}^{\circ}=0$. As $h_H$ is conformally flat,
it has constant sectional curvature $b/n$; this is the standard
Einstein space-form characterization (see \cite{Besse:2008aa}).
Its outward boundary principal curvature is $a/(n\sqrt\kappa)$.
For $b>0$ it is therefore the spherical cap with angular radius
$\arccos D_0$, and for $a=0$ the round hemisphere of geodesic radius
$\pi\sqrt{n/b}/2$. For $b=0$, it is the Euclidean ball of radius
$R_0=n\sqrt\kappa/a$, with
\[
\rho_H=\frac{a}{2n}(R_0^2-R^2),\qquad 0\le R\le R_0.
\]
  \end{proof}
  
\section{Existence of prescribed-source compactifications}\label{sec:existence}
  Let $n\ge3$, $k=n-1$, and let $q$ satisfy \eqref{eq:q-class}.
  For a metric $dt^2+g_t$ on an interval times a closed cross-section,
  the coordinate formula for divergence is
  \[
  \Delta z=\frac{1}{\sqrt{\det g}}\partial_i
     \bigl(\sqrt{\det g}\,g^{ij}\partial_jz\bigr).
  \]
  If the volume form is $j(t)\,dt\,dV_N$ and $z=z(t)$, this reduces to
  \[
  \Delta z=j^{-1}(jz')',\qquad
  \int_M |dz|_g^2\,dV_g=\operatorname{Vol}(N)\int jz'^2\,dt.
  \]
  We minimize the corresponding semilinear energy on a radial
  Nehari set.

  For $f=f_q$ in \eqref{eq:exact-transform}, set
  \[
  p_1=1+1/k,\qquad p_2=1+2/k.
  \]
  Concavity of $q$ and \eqref{eq:q-slope} imply
  \begin{equation}\label{eq:variational-growth}
  f(s)+sf'(s)\le C(s^{p_1}+s^{p_2}),\qquad
  sf'(s)\ge p_1f(s),\qquad
  (p_1+1)F(s)\le sf(s).
  \end{equation}
  The lower bound in \eqref{eq:q-slope} integrates from $t$ to $s$ to give
  $f(t)\le f(s)(t/s)^{p_1}$ for $0<t\le s$, and hence
  $F(s)\le sf(s)/(p_1+1)$. The upper growth follows from
  $q(r)\le q(0)+q'(0)r$. In addition $f(s)/s\to0$ at zero and $f(s)/s\to\infty$ at infinity. The latter follows from $q(s^{1/k})\ge q(1)>0$ for $s\ge1$. We extend the quotient $f(s)/s$ continuously by zero at $s=0$.
  
  \begin{lemma}\label{lem:nehari}
  Let $X$ be one of the weighted radial Sobolev spaces used below, with density $j$ and equivalent Hilbert norm $\|v\|^2=\int jv'^2$. Suppose $X$ is closed under absolute value and embeds compactly in $L^{p_1+1}(j\,dt)$ and $L^{p_2+1}(j\,dt)$. Then
  \[
  \mathcal E(v)=\frac{1}{2}\|v\|^2-\int jF(|v|)\,dt
  \]
  attains its infimum on
  \[
  \mathcal N=\left\{v\in X\setminus\{0\}:\|v\|^2=\int j|v|f(|v|)\,dt\right\}
  \]
  at a nonnegative nonzero function $u$, and $\mathcal E'(u)=0$ on $X$.
  \end{lemma}
  \begin{proof}
We use the standard Nehari natural-constraint argument
(see \cite{Willem:1996aa}), in the weighted form needed for the two
radial spaces below. The growth bounds make the energy and constraint
$C^1$. Strict increase of $f(s)/s$, from zero to infinity, gives a unique
Nehari point on each nonzero ray. On $\mathcal N$,
\[
\|v\|\ge\delta_0>0,\qquad
\mathcal E(v)\ge\left(\frac12-\frac1{p_1+1}\right)\|v\|^2.
\]
Take a nonnegative minimizing sequence. Compactness in both stated
Lebesgue spaces gives a nonzero weak limit $v\ge0$, convergence of the
nonlinear integrals, and $\|v\|^2\le\int jvf(v)$.
Thus $\theta v\in\mathcal N$ for some $0<\theta\le1$.
The increasing function $B(s)=sf(s)/2-F(s)$ gives
\[
\mathcal E(\theta v)=\int jB(\theta v)\le\int jB(v)
=\inf_{\mathcal N}\mathcal E,
\]
so $u=\theta v$ is a minimizer. For
$K(v)=\|v\|^2-\int j|v|f(|v|)$,
\[
K'(u)[u]\le-(p_1-1)\|u\|^2<0,\qquad \mathcal E'(u)[u]=0.
\]
The constraint multiplier therefore vanishes, proving
$\mathcal E'(u)=0$.
\end{proof}

  Suppose that $j(t)$ is comparable to $e^{kt}$ for large $t$,
  that $u(t)>0$ tends to zero, and that
  \begin{equation}\label{eq:positive-flux}
  B_0(t):=-j(t)u'(t)>0,\qquad B_0'(t)=j(t)f(u(t)).
  \end{equation}
  If a Sobolev estimate gives $u(t)=O(e^{-kt/2})$, then $0<B_0(\infty)<\infty$. Indeed, for small $s$, $f(s)\le Cs^{p_1}$. A bound $u(t)=O(e^{-\sigma t})$ gives $B_0'(t)=O(e^{(k-p_1\sigma)t})$, while
  \[
  u(t)=\int_t^\infty\frac{B_0(s)}{j(s)}\,ds.
  \]
  Integration first bounds $B_0$, and then $u$:
  \[
  u(t)=
  \begin{cases}
  O(e^{-p_1\sigma t}),&p_1\sigma<k,\\
  O((1+t)e^{-kt}),&p_1\sigma=k,\\
  O(e^{-kt}),&p_1\sigma>k.
  \end{cases}
  \]
  Starting at $\sigma=k/2$ reaches the last case after finitely many iterations; in the equality case first use $\sigma=k-\varepsilon$ with $p_1(k-\varepsilon)>k$. Integrability of $B_0'$ and positivity of the flux prove the assertion.
  
  In particular, once the positive flux has a finite limit,
  $j(t)=j_\infty e^{kt}(1+o(1))$ gives
  \[
  u(t)=\frac{B_\infty}{kj_\infty}e^{-kt}(1+o(1)),
  \qquad B_\infty>0.
  \]

  \subsection{Two-ended compactifications}
  Let $t\in\mathbb R$ be the normal distance coordinate in the
  warped product $dt^2+a(t)^2h_N$.
  
  Let $(N^k,h_N)$ be connected, closed, and Einstein with $\operatorname{Ric}_{h_N}=-(k-1)h_N$. Fix an integer $\ell\ge k-1$ and choose
  \begin{equation}\label{eq:epsilon-family}
  0\le\varepsilon<\frac1{\ell+1},\qquad
  c_\varepsilon=\varepsilon(1-\varepsilon)^\ell,
  \end{equation}
  and define the even warping function $a_\varepsilon$ by
  \begin{equation}\label{eq:warping-ode}
  a''=a-\ell c_\varepsilon a^{-2\ell-1},\qquad
   a(0)=\sqrt{1-\varepsilon},\quad a'(0)=0.
  \end{equation}
  For $\varepsilon=0$ this is $a_0(t)=\cosh t$.
  
  \begin{theorem}\label{thm:two-ended}
  The metric
  \begin{equation}\label{eq:two-ended-metric}
  g_\varepsilon=dt^2+a_\varepsilon(t)^2h_N\quad\text{on }\mathbb R\times N
  \end{equation}
  is smooth and complete, and
  \begin{equation}\label{eq:two-ended-ricci}
  \operatorname{Ric}_{g_\varepsilon}+k g_\varepsilon
  =c_\varepsilon a_\varepsilon^{-2\ell-2}
  \left(k\ell\,dt^2+(\ell-k+1)a_\varepsilon^2h_N\right)\ge0.
  \end{equation}
  For every $q$ satisfying \eqref{eq:q-class}, there is a positive even function $\rho(t)$ such that $h=\rho^2g_\varepsilon$ extends smoothly to a compactification with two boundary components and $-\Delta_h\rho=q(\rho)$. In this compactification,
  \begin{equation}\label{eq:two-ended-decay}
  R_{g_\varepsilon}+n(n-1)=O(\rho^{2\ell+2})=o(\rho^2).
  \end{equation}
  Both boundary metrics are the same positive constant multiple of $h_N$ and have negative scalar curvature. For $\varepsilon>0$ the complete metric is not Einstein. For $\varepsilon>0$ and $\ell\ge k$, its Ricci excess is positive definite at every interior point. For $\varepsilon=0$ it is Einstein but not hyperbolic space. The sharp boundary comparison is strict throughout this family.
  \end{theorem}
  \begin{proof}
  Multiplying \eqref{eq:warping-ode} by $2a'$ gives
  \[
  \bigl(a'^2-a^2-c_\varepsilon a^{-2\ell}\bigr)'=0.
  \]
  At $t=0$ its value is
  $-(1-\varepsilon)-\varepsilon=-1$. Thus the first integral is
  \begin{equation}\label{eq:warping-first-integral}
  a'^2=a^2-1+c_\varepsilon a^{-2\ell}.
  \end{equation}
  At the initial value $a_* =\sqrt{1-\varepsilon}$ the right side is zero, and
  \[
  \frac{a''}a=1-\ell c_\varepsilon a^{-2\ell-2}
  \ge\frac{1-(\ell+1)\varepsilon}{1-\varepsilon}=:\eta_\varepsilon>0
  \quad\text{when }a\ge a_*.
  \]
  Thus $a$ increases strictly for $t>0$, is strictly convex, and cannot approach a finite upper limit. Since $a''\le a$, it has no finite-time blow-up. The even solution is positive and global. The first integral gives $a'/a\to1$ and
  $(\log a)'=1+O(a^{-2})$. The lower exponential growth from $a''\ge\eta_\varepsilon a$ makes the error integrable. Therefore $a(t)e^{-t}$ has a finite positive limit as $t\to\infty$.
  
The warped-product Ricci formulas (see \cite{Besse:2008aa}) give
\[
\operatorname{Ric}(\partial_t,\partial_t)=-k a''/a,\qquad
\operatorname{Ric}|_{TN}
=\operatorname{Ric}_{h_N}-\bigl(aa''+(k-1)a'^2\bigr)h_N.
\]
Substituting \eqref{eq:warping-ode}--\eqref{eq:warping-first-integral}, we obtain
\[
\operatorname{Ric}+kg
=c_\varepsilon a^{-2\ell-2}
\left(k\ell\,dt^2+(\ell-k+1)a^2h_N\right).
\]
  This proves \eqref{eq:two-ended-ricci}. The two eigenvalues of the
  Ricci excess relative to $g$ are
  $k\ell c_\varepsilon a^{-2\ell-2}$ in the normal direction and
  $(\ell-k+1)c_\varepsilon a^{-2\ell-2}$ in each tangential direction.
  Thus $\ell\ge k-1$ gives the lower bound, while
  $\ell\ge k$ and $\varepsilon>0$ give a strict margin in every
  direction at each finite point. Completeness follows since the base line is complete, $N$ is compact, and $a\ge a_*>0$.
  
  To construct the defining function, put $j=a^k$ and take the even subspace $X$ of the weighted Sobolev space with squared norm $\int_{\mathbb R}j(v'^2+v^2)\,dt$. For compactly supported $v$, the substitution $z=j^{1/2}v$ gives
  \begin{equation}\label{eq:warped-coercivity}
  \int_{\mathbb R}jv'^2\,dt
  =\int_{\mathbb R}\left[z'^2+
  \left(\frac k2\frac{a''}a+\frac{k(k-2)}4\frac{a'^2}{a^2}\right)z^2\right]dt.
  \end{equation}
  Since $k\ge2$, the potential is bounded below by
  $k\eta_\varepsilon/2>0$ and bounded above. Thus the gradient norm is equivalent to the weighted Sobolev norm and controls $\|z\|_{H^1(\mathbb R)}$. The one-dimensional Sobolev inequality yields
  \[
  |v(t)|\le C\|v\|a(t)^{-k/2}.
  \]
  For every finite $r>2$,
  \[
  \int_{|t|>T}j|v|^r\,dt
  \le\sup_{|t|>T}|v|^{r-2}\int_{\mathbb R}jv^2\,dt
  \le C_r e^{-k(r-2)T/2}\|v\|^r.
  \]
  Local Rellich compactness and this tail bound prove the compact embeddings required by Lemma~\ref{lem:nehari}.
  
  That lemma gives a nonnegative nonzero even weak solution of
  \begin{equation}\label{eq:warped-u}
  u''+k(a'/a)u'+f(u)=0.
  \end{equation}
  The equation holds against arbitrary compactly supported tests on the line: the odd part of a test pairs to zero with the even solution and even density. Ordinary local regularity gives a classical solution, and uniqueness at a zero gives $u>0$. It is smooth, because $f$ is smooth at every positive argument. Its flux is
  \[
  B_0(t)=-j(t)u'(t)=\int_0^t j(s)f(u(s))\,ds>0\quad(t>0).
  \]
  Thus $u'<0$ for $t>0$. The pointwise Sobolev estimate gives $u=O(e^{-kt/2})$, so the decay argument above proves $0<B_\infty:=B_0(\infty)<\infty$.
  
  Near either end use
  $x=1/a(t)$, reversing $t$ on the negative end. Then
  \[
  \frac{dx}{dt}=-x\sqrt{D(x)},\qquad
  a'^2=x^{-2}D(x),\qquad
  D(x)=1-x^2+c_\varepsilon x^{2\ell+2}.
  \]
  The first integral gives
  \[
  g=x^{-2}\left(\frac{dx^2}{D(x)}+h_N\right).
  \]
  Since $u(t)=\int_t^\infty B_0(s)/a(s)^k\,ds$ and $a'/a\to1$, l'H\^opital's rule gives
  \[
  u=x^k v(x),\qquad v(x)\longrightarrow v_0=B_\infty/k>0.
  \]
  In this coordinate the radial operator is
  \[
  \Delta u=x^2Du_{xx}
    +\bigl((1-k)xD+\tfrac12x^2D'\bigr)u_x.
  \]
  Substitute $u=x^kv$, use
  $f(x^kv)=kx^{k+1}v^{1+1/k}q(xv^{1/k})$, and cancel the
  order-$x^k$ terms. After division by $x^{k+2}$ the equation is
  \[
  Dv''+\left(\frac{k+1}{x}D+\frac{D'}2\right)v'
   +\frac{kD'}{2x}v
   +\frac{k}{x}v^{1+1/k}q(xv^{1/k})=0.
  \]
  Multiplication by $x^{k+1}/\sqrt D$ gives
  \begin{equation}\label{eq:warped-v}
  (x^{k+1}\sqrt D\,v')'
  =-\frac{kx^k}{\sqrt D}
  \left(\frac{D'}2v+v^{1+1/k}q(xv^{1/k})\right).
  \end{equation}
  The right side is integrable at zero. A nonzero limit of $x^{k+1}\sqrt D\,v'$ would make $v$ unbounded, so that limit is zero. Therefore
  \begin{align*}
  v'(x)=-\frac{k}{\sqrt{D(x)}}\int_0^1\frac{\tau^k}{\sqrt{D(x\tau)}}
  \left[\frac{D'(x\tau)}2v(x\tau)
  +v(x\tau)^{1+1/k}q(x\tau v(x\tau)^{1/k})\right]d\tau.
  \end{align*}
  This gives $v\in C^1$ at zero and then, by differentiating under the integral, $v\in C^\infty$. Positivity of $v_0$ removes any fractional-power singularity; in particular $v'(0)=-kq(0)v_0^{1+1/k}/(k+1)$. Consequently
  \[
  \rho=xv^{1/k},\qquad
  h=v^{2/k}\left(\frac{dx^2}{D(x)}+h_N\right)
  \]
  is a smooth nondegenerate compactification. Evenness gives the same boundary metric $v_0^{2/k}h_N$ at both ends. In the finite part, including $t=0$, $u$ is smooth and positive, so $\rho=u^{1/k}$ is smooth. The exact transform gives the prescribed-source equation.
  
  Taking the trace of \eqref{eq:two-ended-ricci} gives
  \[
  R_g+n(n-1)=k(2\ell-k+1)c_\varepsilon x^{2\ell+2}.
  \]
  Since $\rho/x$ has a positive boundary limit, this proves \eqref{eq:two-ended-decay}. The boundary scalar curvature is $-k(k-1)v_0^{-2/k}<0$, so the comparison is strict because $R_{q,1}>0$. For $\varepsilon>0$, the radial and tangential eigenvalues of the Ricci excess differ, since $k\ell\ne\ell-k+1$; hence $g_\varepsilon$ is not Einstein. At $\varepsilon=0$, its two ends exclude an isometry with hyperbolic space.
  \end{proof}
  \begin{corollary}\label{cor:nonconstant-scalar}
  For $\varepsilon=0$ in Theorem~\ref{thm:two-ended}, if $q(0)>0$, the compactified scalar curvature $R_h$ is nonconstant. In particular this holds for $q(s)=a+bs$ and $q(s)=a+b(1-e^{-s})$, $a,b>0$.
  \end{corollary}
  \begin{proof}
  Here $\mathcal R=0$. The boundary scalar curvature is negative, so \eqref{eq:boundary-scalar} gives $T|_{\partial M}<0$. If $R_h$ were constant, \eqref{eq:T-def} would give $dT=-J'(\rho)d\rho$. Multiply \eqref{eq:T-bochner} by $\rho$ and restrict to the boundary. Since $|d\rho|^2=1$ and $\mathcal A=O(\rho)$,
  \[
  (n+2)\langle d\rho,dT\rangle-2q(0)T=-(n-1)J'(0).
  \]
  It follows that $2q(0)T=-3J'(0)\ge0$, contradicting $T<0$ and $q(0)>0$.
  \end{proof}

  The radial equation and its compactness depend on the base only through
  its volume. The same construction therefore applies to a connected
  closed base with $\operatorname{Ric}_{h_N}\ge-(k-1)h_N$. The base curvature
  then contributes the order-two term in the scalar error:
  \begin{align*}
  \operatorname{Ric}_{g_\varepsilon}+kg_\varepsilon
  &=\operatorname{Ric}_{h_N}+(k-1)h_N+c_\varepsilon a^{-2\ell-2}
  \left(k\ell\,dt^2+(\ell-k+1)a^2h_N\right),\\
  R_{g_\varepsilon}+nk
  &=a^{-2}\bigl(R_{h_N}+k(k-1)\bigr)
  +k(2\ell-k+1)c_\varepsilon a^{-2\ell-2}.
  \end{align*}
  Thus the scalar error is $O(\rho^2)$, and on either boundary component
  \[
  S|_{\partial M}=v_0^{-2/k}\bigl(R_{h_N}+k(k-1)\bigr),\qquad
  R_\gamma-S|_{\partial M}=-k(k-1)v_0^{-2/k}.
  \]
  These are the corrected boundary quantities in
  Theorem~\ref{thm:boundary-general}. Since the base Ricci excess is
  nonnegative, its trace vanishes exactly when it vanishes. Hence the
  Einstein condition on the base is equivalent to the disappearance of
  this order-two scalar coefficient.
  
  \begin{proposition}\label{prop:volume-deformations}
  Let $(N^k,h_N)$ be as in Theorem~\ref{thm:two-ended}. Fix $\ell\ge k$, $0<\varepsilon<1/(\ell+1)$, and a compact interval $J\subset\mathbb R$. Let $B(t,y)$ be a smooth $h_N$-self-adjoint trace-free endomorphism of $TN$, supported in $J\times N$. Set
  \[
  h_B(t)(X,Y)=h_N(e^{B(t)}X,Y),\qquad
  \widetilde g=dt^2+a_\varepsilon(t)^2h_B(t).
  \]
  There is $\delta>0$, depending only on the fixed geometry and $J$, such that $\|B\|_{C^2}<\delta$ implies
  \[
  \operatorname{Ric}_{\widetilde g}\ge-k\widetilde g.
  \]
  For every source $q$ in \eqref{eq:q-class}, the defining function $\rho_q$ constructed in Theorem~\ref{thm:two-ended} satisfies
  \[
  -\Delta_{\rho_q^2\widetilde g}\rho_q=q(\rho_q).
  \]
  The compactification agrees with the undeformed one near both boundary components, and its scalar error is $o(\rho_q^2)$. Whenever $h_B'(t)$ is nonzero at a regular level of $\rho_q$, that level is nonumbilic.
  \end{proposition}
  \begin{proof}
  Since $\operatorname{tr}_{h_N}B=0$, one has $\det(e^B)=1$ and $dV_{h_B(t)}=dV_{h_N}$. Hence, for every function $z=z(t)$,
  \begin{equation}\label{eq:preserved-radial-laplacian}
  \Delta_{\widetilde g}z
  =a_\varepsilon^{-k}\partial_t(a_\varepsilon^k z')
  =z''+k(a_\varepsilon'/a_\varepsilon)z'.
  \end{equation}
  Thus the positive solution $u_q=\rho_q^k$ of the undeformed equation solves exactly the same semilinear equation for $\widetilde g$. The conformal transform \eqref{eq:exact-transform} gives the claimed source equation.
  
  On $J\times N$, \eqref{eq:two-ended-ricci} has a positive lower eigenvalue. The map $g\mapsto\operatorname{Ric}+kg$ is continuous in $C^2$ on this compact set, so all sufficiently small deformations preserve its positive definiteness. Outside the support, the metric is unchanged. Smoothness, completeness, and the asymptotic scalar condition follow from equality with the old metric and defining function near infinity.
  
  For the normal $\partial_t$, the shape operator of a level in $\widetilde g$ is
  \[
  \frac{a_\varepsilon'}{a_\varepsilon}\operatorname{Id}
  +\frac{1}{2} h_B^{-1}h_B'.
  \]
  Its second term is trace-free, since $\partial_t\log\det h_B=0$. Thus $h_B'\ne0$ gives a nonumbilic level. The conformal change by $\rho_q^2$ multiplies differences of principal curvatures by $\rho_q^{-1}$ and does not remove the trace-free part.
  \end{proof}
  \subsection{AdS--Schwarzschild compactifications}
  The central bolt is the smooth
  codimension-two submanifold where the circle factor collapses;
  its normal two-plane is a smooth disk.
  
  \begin{theorem}\label{thm:schwarzschild}
  Let $n\ge4$, $k=n-1$, and $c>0$. Set
  \[
  V(r)=1+r^2-cr^{2-k}.
  \]
  Let $r_0$ be its unique positive zero and give $\theta$ period $4\pi/V'(r_0)$. The metric
  \begin{equation}\label{eq:schwarzschild}
  g=\frac{dr^2}{V(r)}+V(r)\,d\theta^2+r^2g_{\mathbb S^{n-2}},\qquad r\ge r_0,
  \end{equation}
  extends smoothly to a complete Einstein metric with $\operatorname{Ric}=-kg$ on $\mathbb R^2\times\mathbb S^{n-2}$. For every $q$ in \eqref{eq:q-class}, there is a radial defining function $\rho$ such that $h=\rho^2g$ is a smooth compactification on $D^2\times\mathbb S^{n-2}$ with $-\Delta_h\rho=q(\rho)$. Its connected boundary metric has the form
  \[
  \gamma=C^2(d\theta^2+g_{\mathbb S^{n-2}}),\qquad
  R_\gamma=\frac{(n-2)(n-3)}{C^2}>0.
  \]
  The complete metric is not locally symmetric and the regular levels of $\rho$ are not all umbilic. The sharp boundary comparison is strict.
  \end{theorem}
  \begin{proof}
The smooth complete Einstein metric \eqref{eq:schwarzschild}, with circle
period $4\pi/V'(r_0)$, is the standard Euclidean AdS--Schwarzschild metric;
see \cite{Anderson:2006aa}. Here
$V'(r)=2r+c(k-2)r^{1-k}>0$, so its positive root is unique.
Let $t$ be distance from the bolt, with $t(r_0)=0$ and $dr/dt=\sqrt V$. For later use, its curvature gives
\begin{equation}\label{eq:schwarzschild-weyl}
|W_g|_g^2=k(k-1)^2(k-2)c^2r^{-2k}.
\end{equation}
This nonconstant norm excludes local symmetry.
  In $r$-coordinates the volume density is $r^{k-1}$, since the
  $V^{-1}$ and $V$ factors cancel in the determinant. Hence
  \[
  \Delta z=\frac1{r^{k-1}}\partial_r(r^{k-1}Vz_r)
  \quad\text{for }z=z(r).
  \]
  In the distance coordinate $dr/dt=\sqrt V$, the radial volume
  density, up to a constant, is
  \[
  j(t)=r(t)^{k-1}\sqrt{V(r(t))}.
  \]
  Near zero $j(t)=tA(t^2)$, with $A$ smooth and positive. At infinity $r(t)$ is comparable to $e^t$ and $j(t)$ to $e^{kt}$. Moreover
  \[
  \frac{j'}j=(k-1)\frac{\sqrt V}{r}+\frac{V'}{2\sqrt V}>0,
  \qquad \lim_{t\downarrow0}j'/j=\infty,
  \qquad\lim_{t\to\infty}j'/j=k.
  \]
  Thus $j'/j\ge\delta>0$. For smooth compactly supported radial $v$,
  \[
  \delta\int_0^\infty jv^2\,dt
  \le\int_0^\infty j'v^2\,dt
  =-2\int_0^\infty jvv'\,dt
  \le2\left(\int jv^2\right)^{1/2}\left(\int jv'^2\right)^{1/2}.
  \]
  The boundary term at zero vanishes because $j(0)=0$. Let $X$ be the completion, in the weighted $H^1$ norm, of smooth compactly supported functions even near zero. The displayed inequality shows that its gradient norm is equivalent.
  
  Near zero, identifying $v(t)$ with $v(|z|)$ on a disk in $\mathbb R^2$ identifies the weighted $H^1$ norm with a two-dimensional radial Sobolev norm. Rellich's theorem gives compactness in every finite $L^r$ on that disk. Away from zero it is the ordinary one-dimensional compactness. On $[1,\infty)$, $j'/j$ is bounded and $z=j^{1/2}v$ satisfies $\|z\|_{H^1(1,\infty)}\le C\|v\|$. Therefore
  \[
  |v(t)|\le C\|v\|e^{-kt/2},\qquad
  \int_T^\infty j|v|^r\,dt\le C_re^{-k(r-2)T/2}\|v\|^r\quad(r>2).
  \]
  This proves the compact embeddings required in Lemma~\ref{lem:nehari}. We obtain a nonnegative nonzero weak solution
  \begin{equation}\label{eq:schwarzschild-u}
  (ju')'+jf(u)=0.
  \end{equation}
  
Averaging a compactly supported smooth test over the isometric group
$\mathbb S^1\times\mathrm{SO}(n-1)$ gives an admissible radial test, smooth
and even at the bolt. Since $u$ is invariant under this group, the radial
weak equation is therefore the full weak equation on the smooth metric,
including the bolt. Subcritical elliptic regularity and the strong minimum
principle (see \cite[Chapters~8--9]{Gilbarg:2001aa}) give a positive smooth
solution. Its radial restriction is even in the normal disk coordinate $t$.
Thus $j(0)u'(0)=0$, and integration of \eqref{eq:schwarzschild-u} yields
\[
-j(t)u'(t)=\int_0^t j(s)f(u(s))\,ds=:B_0(t)>0\qquad(t>0).
\]
  The pointwise estimate gives $u=O(e^{-kt/2})$ at infinity. The common decay argument proves $0<B_\infty:=\lim B_0(t)<\infty$. In the $r$ coordinate,
  \[
  -r^{k-1}V(r)u_r=B_0(t(r)),\qquad
  u(r)=\int_r^\infty\frac{B_0(t(s))}{s^{k-1}V(s)}\,ds.
  \]
  Since $V(r)/r^2\to1$, $r^ku(r)\to v_0=B_\infty/k>0$. Set $x=1/r$, $D(x)=1+x^2-cx^k$, and $u=x^kv(x)$. Equation \eqref{eq:schwarzschild-u} becomes
  \[
  (x^{k+1}Dv')'=-kx^k\bigl(D'v+v^{1+1/k}q(xv^{1/k})\bigr).
  \]
  The flux at zero must vanish because $v\to v_0$. Thus
  \[
  v'(x)=-\frac{k}{D(x)}\int_0^1\tau^k
  \left[D'(x\tau)v(x\tau)
  +v(x\tau)^{1+1/k}q(x\tau v(x\tau)^{1/k})\right]d\tau.
  \]
  As in \eqref{eq:warped-v}, this gives $v\in C^\infty$ at zero, with $v_0>0$ and $v'(0)=-kq(0)v_0^{1+1/k}/(k+1)$. It follows that
  \[
  \rho=xv(x)^{1/k},\qquad
  h=v(x)^{2/k}\left(\frac{dx^2}{D(x)}+D(x)d\theta^2+g_{\mathbb S^{n-2}}\right)
  \]
  is the asserted smooth compactification. Here $C=v_0^{1/k}$, and the exact transform proves its source equation.
  
  The collapsing circle gives topology $D^2\times\mathbb S^{n-2}$
  and connected boundary $\mathbb S^1\times\mathbb S^{n-2}$.
  The circle has zero scalar curvature and
  the unit $(n-2)$-sphere has scalar curvature $(n-2)(n-3)$, giving
  the stated positive boundary scalar curvature after scaling by
  $C^2$.

  Since $u_r<0$, the regular $\rho$-levels are the $r$-levels. Their circle and sphere principal curvatures in $g$, for outward normal $\sqrt V\partial_r$, are $V'/(2\sqrt V)$ and $\sqrt V/r$. Their difference is
  \[
  \frac{rV'-2V}{2r\sqrt V}=\frac{-2+kcr^{2-k}}{2r\sqrt V},
  \]
  which is not identically zero. Under $h=\rho^2g$, a principal curvature transforms as $\lambda_i^h=\rho^{-1}(\lambda_i^g+\nu_g(\log\rho))$, so the difference is multiplied by $\rho^{-1}$. The levels are therefore not all umbilic. Finally equality in the sharp boundary comparison would make the interior hyperbolic, contrary to \eqref{eq:schwarzschild-weyl}.
  \end{proof}

\section{Normalization under compactly supported metric perturbations}\label{sec:nonsymmetric}
  Let $n\ge3$, $k=n-1$, and let $q$ satisfy \eqref{eq:q-class}.
  The deformations in Proposition~\ref{prop:volume-deformations}
  preserve the radial Laplacian and hence the previously constructed
  solution. For a general smooth metric $g$ on $\mathbb R\times N$
  agreeing with $g_\varepsilon$ outside a compact set, we seek a
  new positive solution in the full space $H^1(M,g)$.
  The compact embedding used for radial minimization in
  Section~\ref{sec:existence} is no longer available: minimizing
  sequences may move out along an expanding end.

  \subsection{Variational existence}
  The hyperbolic Nehari level provides the energy threshold for
  escape along the unchanged ends. For background on concentration and
  compactness, see \cite{Lions:1984aa,Skrzypczak:2020aa,Ganguly:2025aa}.

  For a complete metric $g$ define
  \begin{align*}
  \mathcal E_g(v)&=\frac{1}{2}\int_M|dv|_g^2\,dV_g-\int_M F(|v|)\,dV_g,\\
  \mathcal N_g&=\left\{v\in H^1(M)\setminus\{0\}:
  \int_M|dv|_g^2\,dV_g=\int_M|v|f(|v|)\,dV_g\right\}.
  \end{align*}
  Set
  \begin{equation}\label{eq:energy-levels}
  c_g=\inf_{\mathcal N_g}\mathcal E_g,
  \qquad c_{\mathbb H}=\inf_{\mathcal N_{g_{\mathbb H}}}\mathcal E_{g_{\mathbb H}},
  \end{equation}
  where $g_{\mathbb H}$ has sectional curvature $-1$. We write
  $f=f_q$ and $F(s)=\int_0^s f(\tau)\,d\tau$ in these definitions.
  By Lemma~\ref{lem:nehari}, the Nehari identity selects the unique
  maximum of $\mathcal E_g$ on each nonzero ray. A solution attaining
  $c_g$ is called a ground state.
  Write $\|v\|_g^2=\int_M|dv|_g^2\,dV_g$; on the metrics below
  this norm is equivalent to the full $H^1$ norm.
  
  In the next proposition, $g_\varepsilon$ is any fixed metric of Theorem~\ref{thm:two-ended}, with integer $\ell\ge k-1$ and $0\le\varepsilon<1/(\ell+1)$.
  
  \begin{proposition}\label{prop:energy-gap}
  Let $g$ be a smooth metric on $\mathbb R\times N$ agreeing with $g_\varepsilon$ outside a compact set, and let $q$ satisfy \eqref{eq:q-class}. Then $0<c_g<\infty$ and $0<c_{\mathbb H}<\infty$. If
  \begin{equation}\label{eq:strict-energy-gap}
  c_g<c_{\mathbb H},
  \end{equation}
  there is a positive smooth solution $u\in H^1(M)$ of $-\Delta u=f_q(u)$ with $\mathcal E_g(u)=c_g$.
  \end{proposition}
  \begin{proof}
  Formula \eqref{eq:warped-coercivity}, applied at each $y\in N$ and integrated in $y$, gives
  \[
  \int_M|dv|_{g_\varepsilon}^2\,dV_{g_\varepsilon}
  \ge\frac{k\eta_\varepsilon}{2}\int_M v^2\,dV_{g_\varepsilon}
  \quad(v\in C_c^\infty(M)).
  \]
  In fact, with $j(t)=a_\varepsilon(t)^k$,
  \[
  |dv|_{g_\varepsilon}^2=|\partial_t v|^2
  +a_\varepsilon^{-2}|d_Nv|_{h_N}^2,
  \qquad dV_{g_\varepsilon}=j(t)\,dt\,dV_{h_N},
  \]
  so the one-dimensional inequality controls the full $L^2$ norm
  after integration over $N$, and the angular energy is nonnegative.
  The metrics $g$ and $g_\varepsilon$ agree outside a compact set;
  on that compact set the eigenvalues of one relative to the other
  have positive lower and finite upper bounds. Their volume forms
  and gradient integrals are therefore uniformly comparable.
  Consequently $\int v^2\,dV_g\le C\|v\|_g^2$, and the gradient
  norm is equivalent to the full $H^1(g)$ norm. The compact part and the explicit ends have coordinate balls with uniformly controlled ellipticity and derivatives, and a covering with uniformly bounded overlap. Local Sobolev inequalities therefore give
  \[
  H^1(M,g)\longrightarrow L^r(M,g),\qquad 2\le r\le 2^*:=\frac{2n}{n-2}.
  \]
On hyperbolic space the usual Poincar\'e and Sobolev inequalities give
the same norm equivalence and embeddings; in particular,
$\int|dv|^2\ge k^2\int v^2/4$ (see \cite{Mancini:2008aa}).
  Since $n\ge3$,
  \[
  2<p_1+1=2+\frac1{n-1}<p_2+1=\frac{2n}{n-1}
      <\frac{2n}{n-2}=2^*.
  \]
  For every $v\in H^1(M)\setminus\{0\}$, the ray argument in Lemma~\ref{lem:nehari} gives a unique $\theta>0$ such that $\theta v\in\mathcal N_g$. The lower norm bound on $\mathcal N_g$ and
  \[
  \mathcal E_g(v)\ge\left(\frac{1}{2}-\frac1{p_1+1}\right)
  \int_M|dv|_g^2\,dV_g\quad(v\in\mathcal N_g)
  \]
  prove positivity of $c_g$; a compactly supported test proves finiteness. The hyperbolic argument is identical.
  
For $K_g(v)=\|v\|_g^2-\int_M|v|f(|v|)\,dV_g$, the Nehari set is a
closed $C^1$ hypersurface separated from zero.
The natural-constraint estimate
$K_g'(v)[v]\le-(p_1-1)\|v\|_g^2$ and the positive lower norm bound
are the same as in Lemma~\ref{lem:nehari}.
Ekeland's principle \cite[Theorem~1.1]{Ekeland:1974aa} on the complete
Nehari set gives a bounded minimizing sequence with vanishing constrained
derivative. Testing its multiplier relation by $v_j$ forces the multipliers
to zero, and the growth bound keeps $K_g'(v_j)$ bounded in $H^{-1}$. Hence
\begin{equation}\label{eq:full-ps}
\mathcal E_g(v_j)\longrightarrow c_g,
  \qquad
\mathcal E_g'(v_j)\longrightarrow0\quad\text{in }H^{-1}(M).
\end{equation}
  Pass to a weak $H^1$ limit $v$ and to strong local $L^r$ limits for
  $r<2^*$. The growth of $f$ permits passage to the weak equation on
  every compact set, giving
  $-\Delta v=\operatorname{sgn}(v)f(|v|)$. By
  \eqref{eq:variational-growth},
  \[
  B(s)=\frac12s f(s)-F(s)
  \ge\left(\frac12-\frac1{p_1+1}\right)sf(s)\ge0\quad(s\ge0).
  \]
  Equation \eqref{eq:full-ps} gives
  \begin{equation}\label{eq:nonnegative-energy-density}
  \int_M B(|v_j|)\,dV_g
  =\mathcal E_g(v_j)-\tfrac{1}{2}\mathcal E_g'(v_j)[v_j]
  \longrightarrow c_g.
  \end{equation}
  
Suppose $v=0$. The compact part and the explicit ends give bounded geometry:
positive injectivity radius and uniform bounds on all covariant derivatives
of curvature. By the bounded-geometry vanishing lemma
\cite{Skrzypczak:2020aa},
\[
\sup_{z\in M}\int_{B_1(z)}|v_j|^2\,dV_g\longrightarrow0
\quad\Longrightarrow\quad
v_j\longrightarrow0\text{ in }L^r(M),\quad2<r<2^*.
\]
Indeed, the bounded $H^1$ pullbacks to moving normal charts tend to zero
in $L^2$ and hence weakly in $H^1$, as required by that lemma.
Applied to $v_j$, this would give
  \[
  \|v_j\|_g^2=\int_M|v_j|f(|v_j|)\,dV_g
  \le C\sum_{i=1}^2\|v_j\|_{L^{p_i+1}}^{p_i+1}\longrightarrow0,
  \]
  contrary to the lower norm bound on $\mathcal N_g$. Thus, after a subsequence, there are $z_j$ and $\beta>0$ with
  \[
  \int_{B_1(z_j)}|v_j|^2\,dV_g\ge\beta.
  \]
  Since the local weak limit is zero, $z_j$ leaves every compact set.
  
  On the positive end write $z_j=(t_j,y_j)$,
  $t_j\to\infty$, set $a_j=a_\varepsilon(t_j)$, and choose
  orthonormal coordinates in $T_{y_j}N$. Use the maps
  \[
  \Phi_j(s,z)=\bigl(t_j+s,\exp_{y_j}^{h_N}(z/a_j)\bigr).
  \]
  The factor $a_j^{-1}$ compensates for the tangential metric
  $a_\varepsilon(t)^2h_N$. If $h_{N,ab}^{(j)}$ are the normal-coordinate
  coefficients at $y_j$, then
  \[
  \Phi_j^*g
  =ds^2+\frac{a_\varepsilon(t_j+s)^2}{a_j^2}
  h_{N,ab}^{(j)}(z/a_j)\,dz^a dz^b.
  \]
  Since $a_\varepsilon(t)e^{-t}$ has a positive limit,
  $a_\varepsilon(t_j+s)/a_j\to e^s$. The warping ODE gives
  the same convergence after differentiation on bounded $s$-intervals.
  As $N$ is closed, its normal-coordinate estimates are uniform in
  $y_j$, and $h_{N,ab}^{(j)}(z/a_j)\to\delta_{ab}$ with all
  derivatives on bounded $z$-sets. Moreover
  $a_j\operatorname{inj}(N,h_N)\to\infty$, so the coordinate domains
  exhaust $\mathbb R\times\mathbb R^k$. Thus
  \[
  \Phi_j^*g\longrightarrow ds^2+e^{2s}|dz|^2
  \quad\text{smoothly on every fixed compact set}.
  \]
  With $\tau=e^{-s}$ the
  limit is $\tau^{-2}(d\tau^2+|dz|^2)$ on the whole
  $\mathbb H^n$. On the negative end use $-t$ in place of $t$.
  
  Pull back $v_j$ to these expanding domains. On each fixed compact set the pulled-back $H^1$ norms are uniformly bounded. A diagonal subsequence converges weakly in $H^1_{\mathrm{loc}}$, strongly in $L^r_{\mathrm{loc}}$ for $r<2^*$, and almost everywhere to a function $w$. A fixed coordinate box contains $B_1(z_j)$ for all large $j$: any path of length at most one starting at $z_j$ stays in $|t-t_j|\le1$, where $a_\varepsilon(t)$ is comparable to $a_\varepsilon(t_j)$. Its base displacement is therefore $O(a_\varepsilon(t_j)^{-1})$. The mass bound and strong local $L^2$ convergence imply $w\ne0$.
  
  For a compactly supported smooth test on the limiting hyperbolic space, its pushforward has uniformly bounded $H^1(g)$ norm. Equation \eqref{eq:full-ps}, smooth metric convergence, and strong local subcritical convergence show that
  \[
  -\Delta_{g_{\mathbb H}}w=\operatorname{sgn}(w)f(|w|).
  \]
  The uniform global $H^1$ bound on $v_j$, lower
  semicontinuity on each fixed limit region, and exhaustion give
  $w\in H^1(\mathbb H^n)$. The subcritical growth makes the weak equation continuous on
$H^1(\mathbb H^n)$. Density of compactly supported smooth functions
therefore allows testing by $w$, so $w\in\mathcal N_{g_{\mathbb H}}$.
Fatou's lemma on each fixed coordinate box, followed by exhaustion, now gives
  \[
  c_{\mathbb H}\le\mathcal E_{g_{\mathbb H}}(w)
  =\int_{\mathbb H^n}B(|w|)\,dV_{g_{\mathbb H}}
  \le c_g.
  \]
  This contradicts \eqref{eq:strict-energy-gap}.
  
  It follows that $v\ne0$. The same density argument gives $v\in\mathcal N_g$, and Fatou's lemma applied to \eqref{eq:nonnegative-energy-density} gives
  \[
  c_g\le\mathcal E_g(v)=\int_M B(|v|)\,dV_g\le c_g.
  \]
  Thus its absolute value also minimizes on $\mathcal N_g$, so the multiplier calculation makes $u=|v|$ a weak solution. The subcritical growth in \eqref{eq:variational-growth} allows iteration of local $W^{2,p}$ estimates and Sobolev embedding from $u\in H^1_{\mathrm{loc}}$ to $u\in L^\infty_{\mathrm{loc}}$. Local elliptic regularity then gives a classical nonnegative solution, and the strong minimum principle makes it positive. Because $f$ is smooth at positive arguments, it is smooth in the interior.
  \end{proof}
  
  Fix $\ell\ge k$ and set
  \[
  \varepsilon_c=\frac1{\ell+1},\qquad
   a_c=\sqrt{\frac\ell{\ell+1}}.
  \]
  At $\varepsilon=\varepsilon_c$, the initial data in
  \eqref{eq:warping-ode} give the constant solution $a\equiv a_c$.
  
  \begin{proposition}\label{prop:long-neck-gap}
  For every fixed source $q$ in \eqref{eq:q-class},
  \begin{equation}\label{eq:level-to-zero}
  \lim_{\varepsilon\uparrow\varepsilon_c}c_{g_\varepsilon}=0.
  \end{equation}
  In particular, there is $\varepsilon_0<\varepsilon_c$ such that \eqref{eq:strict-energy-gap} holds whenever $\varepsilon_0<\varepsilon<\varepsilon_c$. For every fixed $\varepsilon$ in this interval and compact $K\subset M$, there is $\delta>0$ such that the gap also holds for each smooth metric $\widetilde g$ satisfying $\widetilde g=g_\varepsilon$ outside $K$ and $\|\widetilde g-g_\varepsilon\|_{C^0(K,g_\varepsilon)}<\delta$.
  \end{proposition}
  \begin{proof}
  Fix $d>0$ and, for $\varepsilon$ close to $\varepsilon_c$, let $L_\varepsilon$ be the first positive time at which $a_\varepsilon=a_c+d$. Ordinary smooth dependence in \eqref{eq:warping-ode} gives convergence to $a_c$ on every fixed finite interval. Since each subcritical $a_\varepsilon$ tends to infinity, $L_\varepsilon$ is finite and tends to infinity. On $[-L_\varepsilon,L_\varepsilon]$, the volume density $a_\varepsilon^k$ has fixed positive upper and lower bounds.
  
  Choose $\chi\in C_c^\infty((-1,1))$, $0\le\chi\le1$,
  with $\chi=1$ on $[-1/2,1/2]$, and put
  $\phi_\varepsilon(t,y)=\chi(t/L_\varepsilon)$.
  This function is constant along $N$, so
  \[
  \int_M|d\phi_\varepsilon|^2\,dV_{g_\varepsilon}
  =\frac{\operatorname{Vol}(N,h_N)}{L_\varepsilon^2}
  \int_{-L_\varepsilon}^{L_\varepsilon}
  a_\varepsilon(t)^k|\chi'(t/L_\varepsilon)|^2\,dt.
  \]
  The fixed upper and lower bounds for $a_\varepsilon^k$ give
  \begin{equation}\label{eq:neck-test}
  A_\varepsilon:=\int_M|d\phi_\varepsilon|^2\,dV_{g_\varepsilon}
  \le C L_\varepsilon^{-1},
  \qquad
  \operatorname{Vol}_{g_\varepsilon}\{\phi_\varepsilon=1\}\ge c L_\varepsilon.
  \end{equation}
  Let $\theta_\varepsilon>0$ be the unique ray maximum, equivalently $\theta_\varepsilon\phi_\varepsilon\in\mathcal N_{g_\varepsilon}$. Its equation gives
  \[
  A_\varepsilon
  =\int_M\phi_\varepsilon^2
  \frac{f(\theta_\varepsilon\phi_\varepsilon)}{\theta_\varepsilon\phi_\varepsilon}
  \,dV_{g_\varepsilon}
  \ge cL_\varepsilon\frac{f(\theta_\varepsilon)}{\theta_\varepsilon}.
  \]
  Combining this identity with \eqref{eq:neck-test} gives
  $f(\theta_\varepsilon)/\theta_\varepsilon\le C L_\varepsilon^{-2}$.
  Since $f(s)/s$ is strictly increasing and positive for $s>0$,
  $\theta_\varepsilon\to0$. Fix a small $r_0>0$. Concavity and
  $q(0)\ge0$ give, for $0<r\le r_0$,
  \[
  q(r)\ge\left(1-\frac r{r_0}\right)q(0)
  +\frac r{r_0}q(r_0)\ge\frac{q(r_0)}{r_0}r.
  \]
  Therefore, using $f(s)=k s^{1+1/k}q(s^{1/k})$,
  \[
  c_2\theta_\varepsilon^{2/k}
  \le\frac{f(\theta_\varepsilon)}{\theta_\varepsilon}
  \le C L_\varepsilon^{-2},\qquad
  \theta_\varepsilon\le C L_\varepsilon^{-k}.
  \]
  Combining this with \eqref{eq:neck-test} yields the explicit estimate
  \begin{equation}\label{eq:neck-energy-rate}
  0<c_{g_\varepsilon}
  \le\max_{\theta>0}\mathcal E_{g_\varepsilon}(\theta\phi_\varepsilon)
  \le\frac{1}{2}\theta_\varepsilon^2A_\varepsilon
  \le C L_\varepsilon^{-(2k+1)}\longrightarrow0.
  \end{equation}
  Since $c_{\mathbb H}$ is fixed and positive, choose
  $\varepsilon_0<\varepsilon_c$ so that the maximum along this same
  test ray is below $c_{\mathbb H}$ for every
  $\varepsilon_0<\varepsilon<\varepsilon_c$.

Fix such an $\varepsilon$, a compact $K$, and the same nonzero test
$\phi=\phi_\varepsilon$. Its ray equation is
\[
\int_M|d\phi|_g^2\,dV_g
=\int_M\phi^2\frac{f(\theta\phi)}{\theta\phi}\,dV_g.
\]
Under $C^0(K)$ convergence of metrics fixed off $K$, both sides converge
uniformly for $\theta$ in compact intervals. Strict increase of $f(s)/s$
and its limits at zero and infinity keep the unique positive root away
from zero and infinity and make it continuous. The selected ray maximum
therefore remains below $c_{\mathbb H}$ for all sufficiently small changes,
which proves the last assertion.
  \end{proof}
  \subsection{Smooth compactification}
  The two-sided decay estimate below holds for every source in
  \eqref{eq:q-class}. For the additional smooth extension of a
  nonradial solution, we use the local condition
  \begin{equation}\label{eq:source-reflection}
  q(s)=sQ(s^2)\quad\text{for }0\le s<s_0,
  \qquad Q\text{ smooth near }0.
  \end{equation}
  Equivalently, $q$ has a smooth odd extension across zero. Examples
  include $q(s)=bs$ with $b>0$, $q(s)=\tanh s$, and
  $q(s)=s/\sqrt{1+s^2}$.

  \begin{lemma}\label{lem:nonradial-end}
  Let $g$ agree outside a compact set with $g_\varepsilon$ as above, and let $u\in H^1(M)$ be a positive smooth solution of $-\Delta u=f_q(u)$. On either end use $x=1/a_\varepsilon(t)$, with the sign of $t$ reversed on the negative end. Then
  \begin{equation}\label{eq:sharp-end-order}
  cx^k\le u(x,y)\le Cx^k
  \end{equation}
  for sufficiently small $x$. If \eqref{eq:source-reflection} also holds, there is a smooth positive function $v(x,y)$, even in $x$ at zero, such that
  \begin{equation}\label{eq:nonradial-expansion}
  u=x^kv(x,y).
  \end{equation}
  Consequently $\rho=u^{1/k}$ is a smooth defining function and $h=\rho^2g$ is a smooth nondegenerate compactification. Its boundary metric is $v(0,y)^{2/k}h_N$.
  \end{lemma}
  \begin{proof}
  The local coordinate balls used above have uniform Sobolev constants and bounded volumes. Since
  \[
  \frac{f(u)}u\le C\bigl(u^{1/k}+u^{2/k}\bigr)
  \]
  the coefficient $a=f(u)/u$ has uniformly bounded local $L^n$
  norm: $n/k<2$ controls the first term using the uniform upper
  bound for the volumes of $B_2(z)$ and the global $L^2$ norm, while
  $2<2n/k<2^*$ controls the second term by the global Sobolev
  embeddings. Write $A=\sup_z\|a\|_{L^n(B_2(z))}<\infty$.
The local $L^2$-to-$L^\infty$ estimate for $-\Delta u=au$,
with $a\in L^n$ (see \cite[Chapter~8]{Gilbarg:2001aa}), gives
\[
\sup_{B_1(z)}u\le C\|u\|_{L^2(B_2(z))}.
\]
The constant depends only on $A$ and the uniform coordinate bounds,
and is therefore independent of $z$.
  The right side tends to zero as $z$ leaves compact sets, since $u\in L^2(M)$. Thus $u\to0$ uniformly on the ends, and both $u$ and $f(u)/u$ are bounded globally.
  
  On the positive end, $dx/dt=-x\sqrt{D(x)}$, where
  $D(x)=1-x^2+c_\varepsilon x^{2\ell+2}$; reverse $t$ on the
  negative end. The volume density in $(x,y)$ coordinates is
  $x^{-k-1}D^{-1/2}\,dx\,dV_{h_N}$. Substitution into the
  divergence formula for the Laplacian gives
  \[
  \Delta u=x^{k+1}\sqrt D\,
  \partial_x\bigl(x^{1-k}\sqrt D\,u_x\bigr)
  +x^2\Delta_{h_N}u.
  \]
  Expanding this expression yields
  \begin{equation}\label{eq:full-end-laplacian}
  g=x^{-2}\left(\frac{dx^2}{D(x)}+h_N\right),\qquad
  \Delta=x^2D\partial_x^2+\bigl((1-k)xD+\tfrac12x^2D'\bigr)\partial_x
  +x^2\Delta_{h_N}.
  \end{equation}
  Choose $\sigma$ with
  \[
  \max\left\{\frac k2,\frac k{p_1}\right\}<\sigma<k,
  \qquad \beta:=p_1\sigma>k.
  \]
  Such a choice is possible because $p_1>1$. The radial power
  calculation from \eqref{eq:full-end-laplacian} is
  \[
  -\Delta x^\sigma
  =\sigma(k-\sigma)D x^\sigma
  -\frac\sigma2 D'x^{\sigma+1}
  =\sigma(k-\sigma)x^\sigma+O(x^{\sigma+2}).
  \]
  Put $\mathfrak a=f(u)/u$. On a sufficiently far end the outward unit radial field $X$ satisfies $\operatorname{div}_gX\ge k/2$. For a compactly supported $w$ with zero trace at the finite boundary,
  \[
  \frac k2\int w^2\,dV_g
  \le-2\int wXw\,dV_g
  \le2\|w\|_2\|dw\|_2,
  \qquad
  \int|dw|_g^2\,dV_g\ge\frac{k^2}{16}\int w^2\,dV_g.
  \]
  Here $X=\partial_t$ on the positive end has divergence
  $ka_\varepsilon'/a_\varepsilon\to k$, which justifies the
  lower bound $k/2$; on the negative end use $-\partial_t$.
  Since $\mathfrak a\to0$, shrink the end so that
  \[
  \mathfrak a<\min\{k^2/32,\sigma(k-\sigma)/4\},\qquad
  -\Delta x^\sigma\ge\tfrac12\sigma(k-\sigma)x^\sigma.
  \]
  Choose $C$ to dominate the finite-boundary trace of $u$. Then $Cx^\sigma$ is a supersolution of $-\Delta-\mathfrak a$. Moreover
  \[
  \int_0^{x_0}\bigl(x^{2\sigma}+|d(x^\sigma)|_g^2\bigr)
  x^{-k-1}\,dx<\infty
  \quad\text{because }2\sigma>k.
  \]
  Both $u$ and $x^\sigma$ belong to $H^1$ on this end, and
  the choice of $C$ makes the positive part
  $w=(u-Cx^\sigma)_+$ have zero trace at its finite boundary.
  Multiplying by compactly supported cutoffs tending to one at
  infinity approximates $w$ in $H^1$, so $w\in H_0^1$ there and
  is an admissible test for $(-\Delta-\mathfrak a)(u-Cx^\sigma)\le0$. We obtain
  \[
  0\ge\int\bigl(|dw|_g^2-\mathfrak a w^2\bigr)\,dV_g
  \ge\frac{k^2}{32}\int w^2\,dV_g.
  \]
  Thus $u\le Cx^\sigma$.
  
  Since $u\to0$ and $f(s)\le Cs^{p_1}$ for small $s$,
  the first bound gives $f(u)\le Cx^\beta$ with $\beta>k$.
  We can now improve the order from $x^\sigma$ to $x^k$.
  The divergence expression above shows that a radial harmonic
  function satisfies $(x^{1-k}\sqrt D\,H')'=0$. The positive
  solution vanishing at infinity is therefore
  \[
  H_0(x)=\int_0^x\frac{s^{k-1}}{\sqrt{D(s)}}\,ds\sim\frac{x^k}{k}.
  \]
  The function $H_0$ belongs to $H^1$ on this collar. Its finite-boundary value is positive, as is the minimum of $u$ there. Choose $c>0$ with $cH_0\le u$ on that boundary and test the superharmonic inequality by $(cH_0-u)_+\in H_0^1$. This gives $u\ge cH_0$. For an upper bound choose $B>0$ so large that
  \[
  -\Delta(-Bx^\beta)\ge Cx^\beta
  \]
  on a sufficiently small collar; here
  $-\Delta(-Bx^\beta)=B\beta(\beta-k)x^\beta+O(x^{\beta+2})$.
  Choose $A$ so that $AH_0-Bx^\beta$ is positive and dominates $u$ at the finite boundary. Testing the difference by $(u-AH_0+Bx^\beta)_+\in H_0^1$ gives
  $u\le AH_0-Bx^\beta\le AH_0$. This proves \eqref{eq:sharp-end-order}.
  
  Put $v=x^{-k}u$. Then
  \[
  u_x=kx^{k-1}v+x^kv_x,\qquad
  u_{xx}=k(k-1)x^{k-2}v+2kx^{k-1}v_x+x^kv_{xx}.
  \]
  On substituting into \eqref{eq:full-end-laplacian}, the two
  terms of order $x^kv$ cancel. Dividing
  $\Delta u+f(u)=0$ by $x^{k+2}$ gives
  \begin{equation}\label{eq:v-full-end}
  D\left(v_{xx}+\frac{k+1}{x}v_x\right)+\frac{D'}2v_x
  +\Delta_{h_N}v+\frac{kD'}{2x}v
  +\frac{k}{x}v^{1+1/k}q(xv^{1/k})=0.
  \end{equation}
  The two-sided bounds give $0<c\le v\le C$. Interior $W^{2,p}$ estimates for $u$ in uniformly controlled rescaled coordinates, followed by Sobolev embedding with $p>n$, give
  \begin{equation}\label{eq:rough-v-derivative}
  |v_x|+|d_Nv|_{h_N}\le C/x.
  \end{equation}
  Indeed, $|d\log x|_g=\sqrt D$ is bounded on the end, so
  $x$ is comparable to its center value on every fixed-radius
  $g$-ball. On a ball centered at $z$ with $x(z)=x_0$,
  \[
  \sup_{B_2(z)}\bigl(u+f(u)\bigr)\le Cx_0^k,
  \]
  because $u\le Cx^k$ and $f(u)/u$ is bounded.
  The uniform interior gradient estimate in the controlled
  $g$-coordinates gives $|du|_g(z)\le Cx_0^k$.
  Since $g^{xx}=x^2D$ and its tangential inverse is
  $x^2h_N^{-1}$, it follows that
  \[
  |u_x|+|d_Nu|_{h_N}\le Cx^{k-1},\qquad
  |v_x|+|d_Nv|_{h_N}
   \le x^{-k}\bigl(|u_x|+|d_Nu|_{h_N}\bigr)
        +kx^{-k-1}u\le C/x.
  \]
  
  Under \eqref{eq:source-reflection}, the last term in
  \eqref{eq:v-full-end} equals
  \[
  \frac{k}{x}v^{1+1/k}q(xv^{1/k})
  =k v^{1+2/k}Q(x^2v^{2/k}),
  \]
  after shrinking the collar so that $x_0C^{1/k}<s_0$, where
  $v\le C$ and $s_0$ is from \eqref{eq:source-reflection}.
  This is smooth in $x^2$ and $v>0$. For a radial function in
  $\mathbb R^{k+2}$,
  \[
  \Delta_z V(r)=V''(r)+\frac{k+1}{r}V'(r).
  \]
  Let $z\in\mathbb R^{k+2}$ and define
  $\widehat v(y,z)=v(|z|,y)$ for $z\ne0$. Since $D$ is an even polynomial,
  equation \eqref{eq:v-full-end} becomes
  \begin{equation}\label{eq:lifted-end}
  \begin{split}
  D(|z|)\Delta_z\widehat v+\Delta_{h_N}\widehat v
  &+\frac{D'(|z|)}{2|z|}z\cdot\nabla_z\widehat v\\
  &+\frac{kD'(|z|)}{2|z|}\widehat v
  +k\widehat v^{1+2/k}Q\bigl(|z|^2\widehat v^{2/k}\bigr)=0.
  \end{split}
  \end{equation}
  All quotients $D'(|z|)/|z|$ are smooth functions of $|z|^2$.
  Since $D(0)=1$ and $h_N$ is positive definite, the principal symbol
  in a coordinate patch of $N$ is
  \[
  D(|z|)|\zeta|^2+h_N^{ab}(y)\eta_a\eta_b,
  \]
  uniformly bounded above and below by positive multiples of
  $|\zeta|^2+|\eta|^2$ on a small product neighborhood. Thus the
  lifted equation is uniformly elliptic with smooth coefficients;
  its nonlinearity is smooth on the positive range of $\widehat v$.
  
  Estimate \eqref{eq:rough-v-derivative} gives
  $|\nabla_{y,z}\widehat v|\le C/|z|$. The radial density in the
  added $k+2$ variables is $r^{k+1}\,dr$, so
  \[
  \int_{|z|<x_0}|\nabla_{y,z}\widehat v|^2\,dz
  \le C\int_0^{x_0}r^{k-1}\,dr<\infty.
  \]
  Choose any bounded values on the measure-zero set $z=0$.
  For a compactly supported test function, the boundary term in
  integration by parts for a first derivative on $\{|z|>r\}$ is
  $O(r^{k+1})$, because $\widehat v$ is bounded and the normal sphere
  has area $O(r^{k+1})$. Letting $r\downarrow0$ shows that its
  weak first derivatives across the axis are the locally square
  integrable derivatives computed off it. Thus the extension belongs
  to $H^1_{\mathrm{loc}}$.
  To check that the equation also extends, integrate
  it first over $\{|z|>r\}$. Writing $r=|z|$, the $z$ part of \eqref{eq:lifted-end} is
  \[
  \operatorname{div}_z(D\nabla_z\widehat v)
  -\frac{D'(r)}{2r}z\cdot\nabla_z\widehat v.
  \]
  The boundary flux against a fixed smooth test function is bounded by
  \[
  C r^{k+1}\sup_{|z|=r}|\partial_r\widehat v|
  \le C r^k\longrightarrow0.
  \]
  The other integrals converge by the local $H^1$ bound and
  boundedness of the nonlinear term. Thus \eqref{eq:lifted-end}
  holds weakly across the axis.
  
For this $H^1_{\mathrm{loc}}$ weak solution of
\eqref{eq:lifted-end}, the nonlinear term is smooth on the range $c\le\widehat v\le C$.
Interior $L^p$ regularity, Sobolev bootstrap and Schauder estimates
\cite[Chapters~8--9]{Gilbarg:2001aa} therefore give
$\widehat v\in C^\infty$.
Rotational invariance off the axis extends across it by
  continuity. Restricting the smooth function $\widehat v$ to the
  line $z=(x,0,\ldots,0)$ now gives a smooth extension of $v$ to
  negative $x$. The rotation sending this line to its negative shows
  that the extension is even. Its boundary value remains positive
  by \eqref{eq:sharp-end-order}. Finally
  \[
  \rho=xv^{1/k},\qquad
  h=v^{2/k}\left(\frac{dx^2}{D(x)}+h_N\right)
  \]
  are smooth, $d\rho\ne0$ on the boundary,  and $h$ is nondegenerate. In particular,
  \[
  \rho|_{x=0}=0,\qquad
  \partial_x\rho|_{x=0}=v(0,y)^{1/k}>0,\qquad
  h|_{T\partial M}=v(0,y)^{2/k}h_N.
  \]
  Since $v$ and $D$ are even, the boundary is totally geodesic and
  $h$ extends smoothly to the double.
  \end{proof}

  \begin{samepage}
  \begin{theorem}\label{thm:nonsymmetric}
  Fix $n\ge3$, $k=n-1$, an integer $\ell\ge k$, and a connected closed Einstein manifold $(N^k,h_N)$ with $\operatorname{Ric}_{h_N}=-(k-1)h_N$. Let $q$ satisfy \eqref{eq:q-class} and \eqref{eq:source-reflection}. There is $\varepsilon_0\in(0,1/(\ell+1))$ such that, for every $\varepsilon\in(\varepsilon_0,1/(\ell+1))$ and every compact $K\subset\mathbb R\times N$, there is $\delta>0$ for which each smooth metric $\widetilde g$ satisfying
  \[
  \widetilde g=g_\varepsilon\quad\text{outside }K,
  \qquad \|\widetilde g-g_\varepsilon\|_{C^2(K,g_\varepsilon)}<\delta
  \]
  admits a smooth positive defining function $\rho$ on a compactification with two boundary components such that
  \begin{equation}\label{eq:nonsymmetric-conclusion}
  \begin{gathered}
  h=\rho^2\widetilde g,\qquad -\Delta_h\rho=q(\rho),
  \qquad \operatorname{Ric}_{\widetilde g}\ge-k\widetilde g,\\
  R_{\widetilde g}+n(n-1)=O(\rho^{2\ell+2})=o(\rho^2),
  \qquad R_{h|_{T\partial M}}<R_{q,1}\quad\text{everywhere on }\partial M.
  \end{gathered}
  \end{equation}
  The positive function $u=\rho^k$ minimizes the energy on the full Nehari set. Its boundary metrics are $v_\pm(y)^{2/k}h_N$ for smooth positive functions $v_\pm$, not necessarily constant or equal.
  \end{theorem}
  \end{samepage}
  \begin{proof}
  Choose $\varepsilon_0$ by Proposition~\ref{prop:long-neck-gap},
  increasing it to a positive number if needed. Fix $\varepsilon$ and
  $K$. By \eqref{eq:two-ended-ricci},
  $\operatorname{Ric}_{g_\varepsilon}+kg_\varepsilon$ is positive
  definite. Compactness gives $m_K>0$ with
  $\operatorname{Ric}_{g_\varepsilon}+kg_\varepsilon
  \ge m_Kg_\varepsilon$ on $K$. The continuous dependence of Ricci
  curvature on the metric, its inverse and its first two derivatives
  gives, for sufficiently small $\delta$,
  \[
  \left|\operatorname{Ric}_{\widetilde g}+k\widetilde g
  -\operatorname{Ric}_{g_\varepsilon}-kg_\varepsilon\right|_{g_\varepsilon}
  \le C_K\delta<\frac{m_K}{2}
  \quad\text{on }K.
  \]
  Thus $\operatorname{Ric}_{\widetilde g}+k\widetilde g\ge0$
  there. Outside $K$ the metric is unchanged, so the same inequality
  already holds. Small $C^0$ size also makes $\widetilde g$ uniformly
  equivalent to $g_\varepsilon$, which preserves completeness.
  
  Decrease $\delta$ further so that the test function in Proposition~\ref{prop:long-neck-gap} still has ray maximum below $c_{\mathbb H}$. Proposition~\ref{prop:energy-gap} gives a positive full-space ground state $u\in H^1(M)$. Lemma~\ref{lem:nonradial-end} gives $u=x^kv(x,y)$ with $v$ smooth and positive at both ends. In the compact part $u>0$ is smooth. Thus $\rho=u^{1/k}$ and $h=\rho^2\widetilde g$ give a smooth compactification, and \eqref{eq:exact-transform} gives its source equation.
  
  On the unchanged ends,
  \[
  R_{\widetilde g}+n(n-1)=k(2\ell-k+1)c_\varepsilon x^{2\ell+2}.
  \]
  Because $\rho/x=v^{1/k}$ extends smoothly with a positive
  boundary value, $x$ and $\rho$ are bounded above and below by fixed
  positive multiples of one another on each end. The displayed
  scalar excess is therefore $O(\rho^{2\ell+2})$, which is
  $o(\rho^2)$ since $\ell\ge k\ge2$. Theorem~\ref{thm:boundary-general} now applies. Equality at one boundary point would give a hyperbolic ball and a connected boundary, whereas this compactification has two boundary components. The inequality is therefore strict at every boundary point.  The boundary metrics follow directly from the last formula in
  Lemma~\ref{lem:nonradial-end}.
  \end{proof}
  \begin{proof}[Proof of Theorem~\ref{thm:intro-nonsymmetric}]
  Write $D(r)=r^2-1+c_\varepsilon r^{-2\ell}$, with
  $c_\varepsilon=\varepsilon(1-\varepsilon)^\ell$ and
  $r_*=\sqrt{1-\varepsilon}$.
  Then $D(r_*)=0$ and
  $D'(r_*)=2[1-(\ell+1)\varepsilon]/\sqrt{1-\varepsilon}>0$.
  In signed distance $t$ from the doubling hypersurface,
  $r=a(t)$ is smooth and even, and satisfies
  $a'^2=D(a)$ and $a''=a-\ell c_\varepsilon a^{-2\ell-1}$.
  Thus the smooth double in the statement is exactly
  \eqref{eq:two-ended-metric}.
  Apply Theorem~\ref{thm:nonsymmetric} to the given $n,\ell,N,q$.
  The maximal-center model agrees with the threshold compactification
  by the argument in the proof of Theorem~\ref{thm:intro-boundary},
  with $\kappa=1$.
  \end{proof}

  \vspace{2ex}
  
  \noindent\textbf{Acknowledgments.} The author acknowledges support from the National Natural Science Foundation of China (Grant Nos.~12571055 and 12671062), the Natural Science Foundation of Hunan Province (Grant No.~2026JJ20014), the 111 Project (Grant No.~D23017), the Program for Science and Technology Innovative Research Team in Higher Educational Institutions of Hunan Province, China, the Hunan Research Center of the Basic Discipline Fundamental Algorithmic Theory and Novel Computational Methods, and the Key Laboratory of Mathematical Theory of Fluid Mechanics and Its Applications for Higher Education Institutions in Hunan Province.
  
  \vspace{2ex}
  
  \noindent\textbf{AI assistance.} The research direction, initial mathematical framework and main ideas were developed by the author. During revision, OpenAI Codex assisted with mathematical exposition and manuscript organization, expansion and checking of proof calculations, and selected reference checks. The author is responsible for the final manuscript and its mathematical content.
  
  %\bibliographystyle{abbrvnat}
  %\bibliography{ref}

  \end{document}